\documentclass[11pt]{amsart}

\usepackage{amsmath,amscd}
\usepackage{amssymb}
\usepackage{amsthm}

\usepackage{graphicx}

\usepackage{mathrsfs}
\usepackage[ocgcolorlinks, linkcolor=blue]{hyperref}
\usepackage{tikz-cd}
\usepackage{calc}
             {\begin{list}{\arabic{enumi}.}{\usecounter{enumi}%
              \setlength{\labelsep}{0.5em}%
              \settowidth{\labelwidth}{(\arabic{enumi})}%
              \setlength{\leftmargin}{\labelwidth+\labelsep}}}%
             {\end{list}}

\usepackage{comment}
\usepackage{mathtools}

\DeclareRobustCommand{\SkipTocEntry}[5]{}

\usepackage{xcolor}
\definecolor{LOcolor}{RGB}{150,100,0}

\DeclareMathOperator{\WF}{WF}

\newtheorem{Theorem}{Theorem}[section]

\newtheorem{Lemma}[Theorem]{Lemma}
\newtheorem{Corollary}[Theorem]{Corollary}

\newtheorem{Proposition}[Theorem]{Proposition}

\theoremstyle{definition}
\newtheorem{Definition}[Theorem]{Definition}

\newtheorem{Remark}[Theorem]{Remark}

\numberwithin{equation}{section}

\newcommand{\mR}{\mathbb{R}}                    %
\newcommand{\abs}[1]{\lvert #1 \rvert}          %
\newcommand{\norm}[1]{\lVert #1 \rVert}         %

\newcommand{\ol}[1]{\overline{#1}}

\newcommand{\G}{\mathcal G}
\newcommand{\supp}{\mathrm{supp}}

\newcommand{\eps}{\varepsilon}

\newcommand{\p}{\partial}
\newcommand{\dd}{\mathrm{d}}

\newcommand{\BB}{\mathbf B}
\newcommand{\SSS}{\mathbf S}

\newcommand{\ccdot}{\,\cdot\,}
\newcommand{\s}{\hspace{0.5pt}}

\newcounter{sidenote}

\begin{document}

\title{Generalized boundary rigidity and minimal surface transform} 
\author[L. Busch]{Leonard Busch}
\address{Korteweg-de Vries Institute, University of Amsterdam, Amsterdam, Netherlands}
\email{l.a.busch@uva.nl}
\author[T. Liimatainen]{Tony Liimatainen}
\address{Department of Mathematics and Statistics, University of Jyv\"askyl\"a, PO Box 35, 40014 Jyv\"askyl\"a, Finland}
\email{tony.t.liimatainen@jyu.fi}
\author[M. Salo]{Mikko Salo}
\address{Department of Mathematics and Statistics, University of Jyv\"askyl\"a, PO Box 35, 40014 Jyv\"askyl\"a, Finland}
\email{mikko.j.salo@jyu.fi}
\author[L. Tzou]{Leo Tzou}
\address{School of Mathematics and Statistics, University of Melbourne, Melbourne, Australia}
\email{leo.tzou@unimelb.edu.au}

\begin{abstract}
We study a generalized boundary rigidity problem, which investigates whether the areas of embedded minimal surfaces can uniquely determine a Riemannian manifold with boundary. We prove that for a conformal perturbation of an analytic metric in dimension \( n+1 \) (\( n \geq 2 \)), the metric is determined by these volumes under an ampleness condition. Furthermore, we establish H\"older stability for this determination. This result extends earlier works in dimension \( 2+1 \).%

Instead of relying on reductions to Calder\'on type problems and complex geometrical optics solutions, we study the linearized forward operator that gives rise to the \emph{minimal surface transform}, a generalization of the X-ray/Radon transform. We demonstrate that this transform fits into the framework of double fibration transforms and satisfies the Bolker condition in the sense of Guillemin. Under certain assumptions, including a foliation condition, we prove invertibility of this transform on an analytic manifold as well as recovery of the analytic wave front set.

The methods developed in this paper offer new tools for addressing the generalized boundary rigidity problem and expand the scope of applications of double fibration transforms. We anticipate that these techniques will also be applicable to other geometric inverse problems. Beyond mathematics, our results have implications for the AdS/CFT correspondence in physics.

\end{abstract}

\maketitle

\section{Introduction}\label{sec_introduction}

In this work, we investigate whether the areas of embedded minimal surfaces in a compact Riemannian manifold $(M,g)$ that extend up to the boundary $\partial M$ uniquely determine an unknown metric $g$. The smooth manifold $M$ is assumed to be known, with $\dim(M) = n+1$ for $n \geq 2$, and the minimal surfaces $\Sigma \subset M$ are of dimension $n$ (i.e., codimension $1$). An embedded minimal surface is given by an embedding $X$ from an $n$-dimensional manifold into $(M,g)$, satisfying the minimal surface equation in local coordinates \cite{hubeny2007covariant}:
\begin{equation}\label{eq:embedded_min_surf}
g^{ab}\left(\partial_{a} \partial_{b} X^{\mu} + \Gamma^{\mu}_{\nu\lambda} \partial_{a} X^{\nu} \partial_{b} X^{\lambda} - \Gamma^{\gamma}_{ab} \partial_{\gamma} X^{\mu}\right) = 0,
\end{equation}
where $a,b$ run over $1, \dots, n$ and Greek indices run over $1, \dots, n+1$. This condition means that the embedded surface has vanishing mean curvature. Throughout, we adopt the term minimal surface rather than minimal hypersurface for these submanifolds of codimension $1$. This problem can be interpreted as a \emph{generalized boundary rigidity problem}, due to the fact that areas of minimal surfaces generalize the lengths of geodesics that appear in the classical boundary rigidity problem (see for instance \cite{zbMATH02201336, zbMATH07375200} or \cite[Chapter 11]{zbMATH07625517}).

We study the problem in a perturbative setting, meaning that the unknown metric is a priori assumed to be close enough to a known one. 
As we demonstrate, the perturbative problem reduces to analyzing the linearized problem, which involves the invertibility of an integral transform of symmetric $2$-tensors over minimal surfaces,  which we call the \emph{minimal surface transform}. We further reduce our study to perturbations in conformal class, which reduces the transform to act on scalar functions. 

The core of the paper then involves the study of the minimal surface transform, which we show to be in the class of double fibration transforms satisfying the Bolker condition in the sense of Guillemin \cite{MR812288}. Recently, this class of double fibration transforms in the analytic category was studied in \cite{MST}, whose results we use and adapt.

To show that the minimal surface transform is in the above class, we generate enough minimal surfaces  as graphs of functions $u=u(x)$. To this aim, we consider the minimal surface equation \eqref{eq:embedded_min_surf} for $X(x)=(x,u(x))$  in Fermi coordinates $(x,t)$ with respect to any fixed smooth hypersurface (not necessarily a minimal one). In these coordinates the equation reads (see \cite{carstea2024calder}):
\begin{eqnarray}
\label{eq: MSE}
-|h_u|^{-1/2} \p_j \left( \frac{ |h_u|^{1/2} h_u^{jk} \p_k u}{\sqrt{1+|du|_{h_u}^2}} \right) + C(u, \nabla u) = 0, %
\end{eqnarray}
where \[
 C(u,\nabla u)=\frac{1}{2}\frac{1}{(1+\abs{\nabla u}^2_{h_u})^{1/2}}(\p_th_u^{-1})(\nabla u,\nabla u)+\frac{1}{2}(1+\abs{\nabla u}^2_{h_u})^{1/2}\s \text{Tr}(h_u^{-1}\p_th_u).
\]  
Here we used the notation $h_u(x)=g(x,u(x))$. We note that $u\equiv 0$ is a solution to the equation if and only if $\{t=0\}$ is a minimal surface. If $\Sigma=\{t=0\}$ is itself a minimal surface, the linearization of the minimal surface equation is the Schr\"odinger equation, also called stability equation (see e.g. \cite{MR2780140}), 
\begin{eqnarray}
\label{eq: stability equation}
( \Delta_{h_u} + q)u = 0, %
\end{eqnarray}
where $\Delta_{h_u}$ is the Laplace-Beltrami operator of the metric $h_u$, $q := -|A|^2 -{\rm Ric}(N,N)$, $|A|$ is the norm of the second fundamental form of $\Sigma$, and $N$ is the normal vector field of $\Sigma$. Solutions to the above equation correspond to normal variations of $\Sigma$ by minimal surfaces. %

\subsection{Determining a conformal factor from areas of minimal surfaces}

Let us first set up notation in order to state our results. Let $(M,g)$ be a Riemannian manifold, which we without loss of generality assume to be a subset of a slightly larger compact Riemannian manifold $(\widetilde{M},g)$.  We will assume that we know the metric in $\widetilde{M} \setminus M$ and that we know the areas of all minimal surfaces that have boundary in $\widetilde{M} \setminus M$. From these measurements we would like to determine information on the unknown metric $g$ in $M$. 

Next we wish to define a measurement operator that describes the areas of minimal surfaces with respect to a metric $g$. One-dimensional minimal surfaces, i.e.\ geodesics, are easily parametrized by points $(x,v) \in \partial SM$ where $SM$ is the unit sphere bundle, and $x \in \p M$ is the starting point and $v$ is the direction of the geodesic. Then one could consider the measurement operator 
\[
\tau_g: \partial SM \to \mathbb{R}
\]
such that $\tau_g(x,v)$ is the length of the $g$-geodesic through $(x,v)$.

For higher dimensional minimal surfaces there is no such simple parametrization as for geodesics, and this makes the definition of the measurement operator more technical. We start by fixing a background metric ${g}^{\circ}$ in $\widetilde{M}$ and by considering a fixed embedded minimal surface $\Sigma$ in $(\widetilde{M}, {g}^{\circ})$ with smooth boundary $\p \Sigma \subset \widetilde{M} \setminus M$. Then we choose Fermi coordinates $(x,t)$ in $\widetilde{M}$ such that $\Sigma$ corresponds to $\{t=0\}$, and we consider minimal surfaces $\Sigma^f$ that are small perturbations of $\Sigma$. These surfaces are given as graphs in Fermi coordinates, 
\[
\Sigma^f = \{ (x, u^f(x)) \,:\, x \in \Sigma \},
\]
with $u^f = u_{{g}^{\circ}}^f$ being the solution of the minimal surface equation \eqref{eq: MSE} with given Dirichlet data $f$, i.e. 
\[
u^f|_{\p \Sigma} = f.
\]
Here we assume that the stability equation corresponding to $\Sigma$ is well-posed, which can be arranged by shrinking $\Sigma$ slightly (see Proposition \ref{prop:makeadmissible}), and that $f \in W_{\delta}$ where 
\[
W_{\delta} = \{ f \in C^{\infty}(\p \Sigma) \,:\, \norm{f}_{C^{2,\gamma}(\p \Sigma)} \leq \delta \}
\]
with $\gamma>0$ some constant and $\delta > 0$ sufficiently small (so that the minimal surface equation has a unique small solution as stated in Proposition \ref{prop: smooth dependence of solution}).

Above we considered minimal surfaces $\Sigma^f$ with respect to the background metric ${g}^{\circ}$ as graphs over $\Sigma$ parametrized by small Dirichlet data $f \in W_{\delta}$. By an implicit function theorem argument (see Theorem \ref{thm: smooth dependence}), if ${g}$ is a small perturbation of ${g}^{\circ}$, one can similarly construct minimal surfaces $\Sigma_{{g}}^f$ in $(\widetilde{M}, {g})$ parametrized by small Dirichlet data $f \in W_{\delta}$. These minimal surfaces are still expressed in the Fermi coordinates for the background metric $ g^\circ$. This is the parametrization for minimal surfaces that we will use.

With the above parametrization, if $\Sigma$ is the fixed minimal surface and ${g}$ is a metric close to ${g}^{\circ}$, we define the measurement operator 
\begin{equation} \label{intro_f_def}
F({g}) = F_{\Sigma}({g}): W_{\delta} \to \mR
\end{equation}
that takes a Dirichlet data $f$ to the $g$-area of the minimal surface $\Sigma_{{g}}^f$.

Finally, we will need a condition ensuring that our background metric admits sufficiently many minimal surfaces.

\begin{Definition}\label{def:ampleness}
We say that $M$ satisfies the \emph{ampleness condition} if there is a neighborhood $U$ of $M$ in $\widetilde{M}$ such that for every $(y_0,\xi_0) \in T^\ast U$ there is a smooth embedded minimal surface
$$\Sigma=\Sigma_{(y_0,\xi_0)} \text{ in $(\widetilde{M},{g}^\circ)$}$$ 
such that $(y_0,\xi_0) \in N^\ast\Sigma$ and $\p \Sigma \subset \widetilde{M} \setminus M$.
\end{Definition}

If $\dim(\widetilde{M}) = 3$ and if there are no closed minimal surfaces inside $M$, the ampleness condition holds by \cite{mazet2022minimal} (see also the earlier work \cite{zbMATH06945887}). In any dimension, subsets of the Euclidean and hyperbolic space satisfy the ampleness condition.

We can now give an informal version of our main result on the stable determination of a small conformal perturbation from measurements of areas of minimal surfaces.

\begin{Theorem}[Informal statement; see Theorem \ref{thm:stability} for precise statement] \label{thm:rigidity_intro}
	Assume that $\widetilde M$ is an analytic manifold and ${g}^{\circ}$ is an analytic metric such that the ampleness condition in Definition~\ref{def:ampleness} holds. There are finitely many minimal surfaces $\Sigma_1, \ldots, \Sigma_K$ of $g^\circ$ with corresponding sets of small Dirichlet data $W_{\delta}^{(k)}$ such that if 
    \[
    F(\alpha_1 {g}^{\circ}) = F(\alpha_2 {g}^{\circ}) \text{ on $W_{\delta}^{(k)}$ for $k=1,\ldots, K$},
    \]
    and if $\alpha_1$, $\alpha_2$ are smooth functions sufficiently close to $1$ with $\alpha_1 = \alpha_2$ in $\widetilde{M} \setminus M$, then $\alpha_1 \equiv \alpha_2$.
    
    Moreover, there is a conditional H\"older stability estimate showing that $\alpha_1-\alpha_2$ must be small if $F(\alpha_1 {g}^{\circ}) - F(\alpha_2 {g}^{\circ})$ are small on each $W_{\delta}^{(k)}$.
\end{Theorem}

In fact, we will only use measurements corresponding to certain $(2n+2)$-dimensional subspaces of the infinite dimensional parameter manifolds $W_{\delta}^{(k)}$. These spaces depend on the background metric $g^{\circ}$ and on $\Sigma_k$, but not on the unknown conformal factors.

\subsection*{The minimal surface transform}

The proof of Theorem \ref{thm:rigidity_intro} is based on studying the Fr\'echet derivative, or first linearization, of the nonlinear measurement operator $F$. We will show that this linearization is the \emph{minimal surface transform}, which is a Radon type transform that integrates a function over a suitable family of minimal surfaces. Under certain conditions we will prove that this transform is invertible together with stability estimates. A version of the inverse function theorem, see \cite{stefanov2009linearizing}, will then imply local uniqueness and stability in the nonlinear inverse problem as stated in Theorem \ref{thm:rigidity_intro}.

As above, we fix a background metric ${g}^{\circ}$ in $\widetilde{M}$ and an embedded minimal surface $\Sigma$ with $\p \Sigma \subset \widetilde M \setminus M$. If $\alpha$ is a smooth positive function close to $1$, we define the measurement operator for the conformal metric $\alpha {g}^{\circ}$ as in \eqref{intro_f_def} by 
\[
F(\alpha): W_{\delta} \to \mR
\]
such that $F(\alpha)$ takes a Dirichlet data $f$ to the area of the minimal surface $\Sigma_{\alpha {g}^{\circ}}^f$ with respect to metric $\alpha {g}^{\circ}$.

It follows from Theorem \ref{thm: smooth dependence} and Proposition \ref{prop:diffF} that $F$ is a $C^2$ map acting on functions close to $1$ in a suitable Sobolev norm, and its linearization satisfies 
\[
(DF(1)\beta)(f) = \frac{n}{2} \int_{\Sigma^f} \beta \,{\rm dVol}_{\Sigma^f}
\]
where $\beta$ is a function on $\widetilde{M}$ (the infinitesimal variation of $\alpha$), $\Sigma^f$ is the minimal surface corresponding to Dirichlet data $f$, and the volume form ${\rm dVol}_{\Sigma^f}$ is induced on $\Sigma^f$ by the background metric ${g}^{\circ}$. Thus the linearization of $F$ indeed corresponds to integrating the function $\beta$ over minimal surfaces $\Sigma^f$. If one considers general perturbations of the metric (not just conformal ones), then one obtains integrals of $2$-tensor fields as stated in Proposition \ref{prop:diffF}. This is completely analogous to the classical boundary rigidity problem, where the linearization of the boundary distance function results in the geodesic X-ray transform \cite[Chapter 11]{zbMATH07625517}.

In the setting of a fixed minimal surface $\Sigma$ as above, we introduce the notation
\begin{equation} \label{rsigma_def}
\mathcal{R}_{\Sigma} \beta: W_{\delta} \to \mR, \qquad (\mathcal{R}_{\Sigma} \beta)(f) = \int_{\Sigma^f} \beta \,{\rm dVol}_{\Sigma^f}.
\end{equation}
Furthermore, if $\mathcal{S}$ is the set of all compact smooth embedded minimal surfaces in $\widetilde{M}$ that have smooth boundary contained in $\widetilde{M} \setminus M$, we consider the \emph{minimal surface transform}  
\[
\mathcal{R} \beta: \mathcal{S} \to \mR, \qquad (\mathcal{R} \beta)(\Sigma) = \int_{\Sigma} \beta \,{\rm dVol}_{\Sigma}.
\]
The crux of the present paper is to study the properties of the minimal surface transform and to show its invertibility under suitable conditions.

The operator $\mathcal{R}_{\Sigma}$ maps smooth functions in $\widetilde{M}$ to functions on the infinite dimensional manifold $W_{\delta}$. We will show that formally $\mathcal{R}_{\Sigma}$ is a \emph{double fibration transform}, which is a natural class of Radon type transforms studied in integral geometry, see e.g.\ \cite{MR516965, zbMATH02001570, zbMATH05819274, MST}. Such transforms are Fourier integral operators acting between finite dimensional manifolds. Since the manifold $W_{\delta}$ parametrizing the minimal surfaces is infinite dimensional, we will need to introduce a way of restricting to certain finite dimensional submanifolds in order to apply this theory.

A key property for the inversion of double fibration transforms is the so called Bolker condition, which corresponds to (microlocal) ellipticity of the Fourier integral operator and recovery of wave front sets \cite{MR812288}. If the transform integrates over one-dimensional curves, the Bolker condition is typically only satisfied under a no conjugate points condition (see \cite{MST}). Surprisingly, for our minimal surface transform the relevant Bolker condition always holds, with no restrictions on the metric. Here the fact that our minimal surfaces are generated as solutions of an elliptic PDE creates more flexibility than for standard Radon transforms.

The Bolker condition is verified in Section \ref{sec_bolker} by constructing suitable variations of the original minimal surface by controlling the values of solutions of the stability equation. In fact, we need finitely many solutions $v_1, \ldots, v_N$ of the stability equation such that $(v_1, \ldots, v_N)$ is an embedding of the minimal surface $\Sigma$ to some $\mR^N$. We construct such an embedding for some large $N$ by a Runge approximation argument, and following \cite{greene1975whitney, greene1975embedding, zbMATH07504299} we apply Whitney projections to obtain an embedding with $N = 2n+2$. This means that for each base minimal surface, we only use an $(2n+2)$-dimensional subspace of the infinite dimensional data set $W_{\delta}$.

Given the Bolker condition, we can recover smooth singularities of a function from the knowledge of its minimal surface transform. Moreover, if all the underlying structures are real-analytic, one can apply the methods of \cite{MST} to recover information of the analytic wave front set $\operatorname{WF}_a(\beta)$ as follows. The analytic microlocal approach to Radon type transforms was introduced in \cite{zbMATH04052264}, and further developments may be found in \cite{zbMATH02199899, zbMATH05257675, zbMATH05638442, zbMATH06756133, MST, bAFIO}.

\begin{Theorem}\label{thm:integral_transform_WF}
	Assume that $(\widetilde M,{{g}^\circ})$ is analytic, and let $\Sigma$ be an embedded minimal surface with $\p \Sigma \subset \widetilde{M} \setminus M$. If $\beta \in C(\widetilde{M})$ vanishes outside $M$ and satisfies 
	\[
	 (\mathcal{R}_{\Sigma} \beta)(f) = 0
	\] 
    for all $f \in W_{\delta}$, then  
\[
\operatorname{WF}_a(\beta) \cap N^* \Sigma = \emptyset.
\] 
\end{Theorem}

The benefit of recovering analytic singularities is that we can combine it with microlocal analytic continuation (H\"ormander-Kashiwara theorem) to obtain uniqueness results. Thus if $(\widetilde M,{{g}^\circ})$ is analytic we obtain uniqueness and stability for the minimal surface transform under the ampleness condition in Definition \ref{def:ampleness}, but uniqueness also holds under the following foliation condition (see e.g.\ \cite{zbMATH06606440, zbMATH07155016, MST} for related notions). In this condition the leaves $\Omega_s$ of the foliation are not required to be minimal surfaces, but there should be a suitable minimal surface normal to each point of $\Omega_s$.

\begin{Definition}\label{def:fol_intro}
	We say that $(M,g)$, where $(M,g) \subset \subset (\widetilde{M},g)$, satisfies the \emph{foliation condition} if the following holds: there is an open $M^\circ \subset \widetilde{M}$ with $M^\circ\supset M$ and $\rho\in C^1(M^\circ;[0,\infty))$ with $m_\rho \coloneqq \max_{M} \rho$ and level sets $\Omega_s = \rho^{-1}(\{s\}) \subset M^\circ, s\in [0,m_\rho]$ so that 
	\begin{enumerate}
	\item $\Omega_0 \subset M^\circ \setminus M$ and $\ol{\bigcup_{s\in [0,m_\rho)} \Omega_s} \supset M$,
	\item For every $s \in [0,m_\rho)$ and every $x_0 \in \Omega_s \cap M$ there is a minimal surface $\Sigma$ for $(\widetilde{M},{g})$ %
so that $d\rho(x_0)\in N^\ast_{x_0} \Sigma\setminus 0$ and $\p \Sigma \subset \widetilde{M} \setminus M$. %
	\item For every $s_0\in [0,m_\rho)$ and every open neighborhood $U$ of $\Omega_{s_0}\cap M$ in $M^\circ$, there is $\delta>0$ so that $\Omega_{s}\cap M \subset U$ for all $s\in (s_0-\delta,s_0+\delta)\cap [0,m_\rho]$.
	\end{enumerate}
\end{Definition}

\begin{Theorem}\label{thm:integral_transform_uniqueness}
	Assume that $(\widetilde M,{g^\circ})$ is analytic, and suppose that either the ampleness condition in Definition \ref{def:ampleness} or the foliation condition in Definition \ref{def:fol_intro} holds for $M$. If $\beta \in C(\widetilde{M})$ vanishes outside $M$ and satisfies 
	\[
	 (\mathcal{R} \beta)(\Sigma) = 0
	\] 
    for any smooth embedded minimal surface $\Sigma$ with $\p \Sigma \subset \widetilde{M} \setminus M$, then  
\[
\beta = 0.
\]
Moreover, if the ampleness condition holds for $M$, there is a stability estimate where the $L^2(M)$ norm of $\beta$ is bounded by a finite sum of Sobolev norms of $\mathcal{R}_{\Sigma_j} \beta$ in suitable balls (see Theorem \ref{thm:weakfol} for the precise statement).
\end{Theorem}

\subsection{Motivation}
The generalized boundary rigidity problem, formally posed in the mathematics literature in \cite{Tracey}, asks if the areas of minimal surfaces embedded
in a Riemannian manifold determine the Riemannian manifold. If
the minimal surfaces are 1-dimensional (i.e.\ geodesics), this reduces to the the classical boundary rigidity problem.  A key motivation for studying this
generalization arises from its deep connection to 
the AdS/CFT correspondence, a profound duality in theoretical physics proposed by Maldacena~\cite{maldacena1999large}. This conjecture establishes an equivalence between two distinct physical theories: conformal field theories (CFTs) and the geometry of Anti-de Sitter (AdS) spacetimes (the ``bulk''). %

A notable proposal within this correspondence, introduced by Ryu and Takayanagi~\cite{ryu2006holographic,ryu2006aspects}, connects the entanglement entropies of a CFT to the areas of minimal surfaces in an AdS spacetime.  In this framework, the subregion $A$ resides on the conformal boundary of the AdS spacetime,
 and determines a minimal surface $\Sigma$ in the bulk, anchored at infinity to the boundary $\partial A$ of $A$. The Ryu-Takayanagi conjecture posits that the entanglement entropy of $A$ in the CFT equals $(4G)^{-1} \text{Vol}(\Sigma)$, where $G$ is Newton's constant and $\text{Vol}(\Sigma)$ denotes the volume of $\Sigma$. Entanglement entropy measures the quantum correlations between $A$ and its complement. 
 
In the physics literature, the generalized boundary rigidity problem is called the \emph{bulk reconstruction problem}. It means constructing a mapping between the dual variables or degrees of freedom. Progress has been made in highly symmetric settings~\cite{bilson2008extracting,bilson2011extracting,fonda2015shape,jokela2021towards,jokela2019notes, Ji2025} with formal arguments for reconstruction outlined in~\cite{bao2019towards}. The work \cite{lin2015locality} computes the stress-energy tensor from the area data of minimal surfaces on a perturbative level (i.e.\ near the boundary) by inverting a Radon transform over hyperbolic spaces. We refer to \cite{jokela2025bulk, carstea2024calder} for further physics references on the matter.

Arguably, results in bulk reconstruction have been constrained by the limited mathematical tools traditionally used in physics. %
The work~\cite{jokela2025bulk}, building on~\cite{carstea2024calder}, marks a shift by leveraging modern methods from inverse problems (such as the higher order linearization method). The present study expands this direction, offering new state of the art tools for bulk reconstruction. 
We also study stability, which is important for bulk reconstruction and was previously considered for a special case in \cite{jokela2021towards}.

Another mathematical outcome of the current work is to provide further examples of double fibration transforms that can be analyzed by using the framework in \cite{MST}.

\subsection{Novelty of the results}

The previous works~\cite{Tracey, carstea2024calder, carstea2024inverse}  that study the inverse problem of recovering a Riemannian metric from the area data only consider $2$-dimensional minimal surfaces. The restriction to two dimensions arises because these works rely on linearization techniques that reduce the problem to the anisotropic Calder\'on problem for the stability operator, which is understood in two dimensions but remains widely open in dimensions $\geq 3$. Additionally, \cite{carstea2024calder, carstea2024inverse} require a solution to a density problem involving products of solutions to the stability equation, which is also unresolved in dimensions $\geq 3$. The known density results on Riemannian manifolds hold only for cylindrical (CTA) geometries under further assumptions~\cite{dos2009limiting}, even in the analytic category \cite{KLS22}. They are based on complex geometric optics solutions, which limits possible generalization of the method  \cite{LS12}. Other works, such as \cite{nurminen2023inverse, nurminen2024inverse} and \cite{marxkuo2024inverseproblemrenormalizedarea} (in the asymptotically hyperbolic case), consider the case of higher dimensional minimal surfaces, but they only recover the Taylor series of quantities related to the metric at a suitable boundary and thus do not really recover information in the interior. The work~\cite{jokela2025bulk} considers the problem in the special case of small symmetric perturbations of a pure AdS background, and argues a result similar to Theorem \ref{thm:rigidity_intro}.

The approach presented in this paper overcomes the aforementioned obstacles and restrictions. To our knowledge, this is the first work where minimal surfaces may be of any dimension $\geq 2$ when there are no additional symmetry assumptions, and one recovers information in the interior. (Note that Theorem~\ref{thm:rigidity_intro} requires analyticity of the background metric ${g}^\circ$, but not of the conformal factors.) 
Moreover, Theorem~\ref{thm:rigidity_intro} establishes H\"older stability for the recovery, marking a significant improvement over the (at best) logarithmic stability achievable by the methods of \cite{Tracey, carstea2024calder, carstea2024inverse}, which rely on solutions to Calderón type problems. A key reason for this stronger result lies in the nature of the data used: the prior methods rely on a reduction to the Dirichlet-to-Neumann map, for which the corresponding linearized problem is unstable, whereas our analysis uses a subset of the data for which the linearization is a generalized Radon transform type operator that happens to be stably invertible. 

Furthermore, we note that small perturbations of a fixed minimal surface are parametrized by an infinite dimensional Banach manifold, whereas the unknown metric depends on $n+1$ variables. This heuristic dimension count makes the generalized boundary rigidity problem extremely overdetermined (the unknown depends on finitely many variables, whereas the measurements depend on infinitely many). This strong overdetermination also suggests that one might be able to extract information from areas of minimal surfaces in different ways.

As discussed above in connection with the Bolker condition, we only use a finite dimensional subset of the data in our results. Under the ampleness condition, our data set can be parametrized by points $(y,\xi,z)$ where $(y,\xi) \in S^* M$ is a unit conormal of a minimal surface $\Sigma_{(y,\xi)}$ and $z \in \mR^{2n+2}$ parametrizes a subspace of Dirichlet data on $\p \Sigma_{(y,\xi)}$. Thus in Theorem \ref{thm:rigidity_intro} we determine an unknown depending on $n+1$ variables from a data set having  dimension $4n+3$. On the other hand, when proving injectivity of the minimal surface transform under the foliation condition in Theorem \ref{thm:integral_transform_uniqueness}, our data is parametrized by points $(y,z)$ where $y \in M$ (so there is a minimal surface through $y$ with normal $d\rho(y)$) and $z \in \mR^{2n+2}$. Thus in this case the data set is $(3n+3)$-dimensional.

\subsection{Organization}

This article is organized as follows. Section \ref{sec_introduction} is the introduction. In Section \ref{sec_linearization} we compute the linearization of the measurement operator and show that it is given by the minimal surface transform. Section \ref{sec_bolker} considers the minimal surface transform, which is shown to be a double fibration transform (at least after restricting to suitable finite dimensional submanifolds), and it is proved that the Bolker condition always holds. In Section \ref{sec_stability} we give the microlocal arguments proving injectivity of the minimal surface transform and prove the main theorems stated in the introduction.

Appendix \ref{subsec:implicitfunc} includes functional analysis results required for proving Theorem \ref{thm: smooth dependence}. Appendix \ref{sec_continuity_psdo} gives a boundedness result for pseudodifferential operators with finite regularity. Finally, Appendix \ref{sec:ev} includes the argument that one can always avoid the case where $0$ would be a Dirichlet eigenvalue of the stability equation by slightly modifying the boundary.

\subsection*{Acknowledgements}

T.L.\ and M.S.\ are partly supported by the Research Council of Finland (Centre of Excellence in Inverse Modelling and Imaging and FAME Flagship, grants 353091 and 359208). L.T. is partially supported by Australian Research Council Discovery Project DP190103451 and DP220101808.

\section{Area functional and its linearization} \label{sec_linearization}

Throughout this section let $g^\circ$ be a fixed smooth metric on $\widetilde{M}$. In order to generate new minimal surfaces near a given one, of central importance is an assumption on the operator in \eqref{eq: stability equation}. 
\begin{Definition}[Admissible minimal surface]\label{def:admis}
 	For a smooth Riemannian manifold $(N,g)$, we say that a minimal surface $\Sigma$ of $(N,g)$ with non-empty smooth boundary $\partial \Sigma\subset N$ is \emph{admissible} if it is an embedded smooth submanifold of $N$ and its stability operator does not have Dirichlet eigenvalue $0$. 
\end{Definition}
It will turn out that given a minimal surface $\Sigma$ with boundary $\partial \Sigma \subset \tilde M \setminus M$, by modifying the boundary $\partial \Sigma$ slightly, we can always guarantee that any individual smooth embedded minimal surfaces is made admissible, see Proposition~\ref{prop:makeadmissible}.

Assume now that $\Sigma^\circ$ a fixed admissible minimal surface for $(\widetilde{M},g^\circ)$. Let $g$ be any metric on $\widetilde{M}$, $\gamma \in (0,1)$ and $u\in C^{2,\gamma}(\Sigma^\circ)$. Introduce $g^\circ$ Fermi coordinates so that an open neighborhood of $\Sigma^\circ$ is
$$\{(x,t) \mid x\in \Sigma^\circ, t\in (-\epsilon, \epsilon)\}\subset \widetilde{M}.$$
The area (or volume) of the hypersurface $\Sigma_u = \{(x, u(x))\colon x \in \Sigma^\circ\}$ is given by
\begin{eqnarray}\label{eq: area formula}
A_{g}(u) = \int_{\Sigma_u} %
{\rm dVol}^{g}_{\Sigma_u}\,,
\end{eqnarray}
where ${\rm dVol}^{g}_{\Sigma_u}$ is the volume form induced on $\Sigma_u$ by the metric $g$.

For $\gamma \in (0,1)$ as above, define
\begin{equation}\label{eq: def G} G \coloneqq%
H^{3+(n+1)/2 + \gamma}\left(\widetilde{M}; \sigma\left(T^*\widetilde{M} \otimes T^*\widetilde{M}\right)\right) \subset C^{3,\gamma}\left(\widetilde{M}; \sigma\left(T^*\widetilde{M} \otimes T^*\widetilde{M}\right)\right)\,,
\end{equation}
where $\sigma$ is the symmetrization operator. Observe that when $G^\circ \subset G$ is a small enough open neighborhood of $g^\circ$, then each element of $G^\circ$ is a metric on $\widetilde{M}$.

We record a consequence of the implicit function theorem, the proof of which can be found in Appendix~\ref{subsec:implicitfunc}. For any $\delta>0$ we will use the notation 
\[
U_\delta \coloneqq \{f\in C^{2,\gamma}(\partial\Sigma^\circ)\colon \|f\|_{C^{2,\gamma}(\partial\Sigma^\circ)}\leq \delta\}.
\]

\begin{Theorem}\label{thm: smooth dependence}
There is an open neighborhood $G^\circ$ of $g^\circ$ in $G$ and a $\delta>0$ so that for all $(g,f) \in G^\circ\times U_\delta$ there is $u_{g}^f \in C^{2,\gamma}(\Sigma^\circ)$ with $u_{g}^f\vert_{\partial\Sigma^\circ}=f$ so that the graph defined using $g^\circ$-Fermi coordinates
\begin{eqnarray}\label{eq: def of Sigmafg}
\Sigma^\circ(f,g)\coloneqq \{(x,u_{g}^f(x))\mid x\in \Sigma^\circ\} \subset \widetilde{M}
\end{eqnarray}
is a minimal surface for the metric $g$. 

Furthermore, the map
\begin{equation}\label{eq: def of volume operator}
F\colon G^\circ \to C(U_{\delta}; \mR)\,, \quad F(g)(f) = A_{g}(u_{g}^f)
\end{equation}
is $C^{2}$-Fr\'echet differentiable from $G^\circ$ to $C(U_{\delta}; \mR)$. In particular, for each $g \in G^\circ$, there are operators $DF(g)$ and $R_{g}$ depending continuously on $g$ so that for all $H\in G$ small enough we have 
\begin{equation}\label{eq:tayloronF}
F(g +H) = F(g) + DF(g)H + R_{g}(H)\,,\quad\text{and}\quad \norm{R_{g}(H)}_{C(U_\delta; \mathbb R)} \leq C \norm{H}_{G}^2\,,
\end{equation}
where the constant $C>0$ is uniform over $G^\circ$. Finally, for every $g\in G^\circ, f\in U_\delta, H\in G$, 
	\begin{equation}\label{eq:dF}
		DF(g)H(f) = (D_{g}A_{g})(u_{g}^f)H\,.
	\end{equation}
\end{Theorem}
For appropriate $\Sigma,f,g$, the notation $\Sigma(f,g)\coloneqq \{(x,u_{g}^f(x))\mid x\in \Sigma\}$ used in Theorem~\ref{thm: smooth dependence} will be used throughout.%

In particular we will show
\begin{Proposition}\label{prop:diffF}
	There is an open neighborhood $G^\circ$ of $g^\circ$ in $G$ and $\delta>0$ small so that for any $g\in G^\circ$, any $f\in U_\delta$ and any $H\in G$ we have
	\begin{equation}\label{eq: differential of area functional}
		(DF(g)H)(f) =  \frac{1}{2} \int_{\Sigma^\circ(f,g)} \mathrm{tr}_{g^f_\Sigma}( \iota_{\Sigma^\circ(f,g)}^* H)
{\rm dVol}^{g}_{\Sigma^\circ(f,g)}\,,
	\end{equation}
	where $g^f_\Sigma$ is notation for the induced metric of $g$ on $\Sigma^\circ(f,g)$. If the deformation is within a conformal class, meaning that $H(x,t) = \beta(x,t) g(x,t)$ for some function $\beta\in C(\widetilde{M})$, this reduces to 
\begin{eqnarray}
\label{eq: minimal surface transform}
(DF(g)(\beta g))(f) = \frac{n}{2} \int_{\Sigma^\circ(f,g)} \beta(x,u^f_{g}(x)) {\rm dVol}^{g}_{\Sigma^\circ(f,g)}\,.
\end{eqnarray}
\end{Proposition}

Before moving on to the proof, we state
\begin{Lemma}
\label{lem: differential of volume form}
For $G^\circ \subset G$ a small neighborhood of $g^\circ$, let $\{g_s\mid s\in \mathbb R\} \subset G^\circ$ be a one-parameter family of metrics on $\widetilde{M}$ and let $\Sigma$ be an embedded smooth hypersurface. 
We have %
\begin{eqnarray}\nonumber
\partial_s {\rm dVol}^{g_s}_{\Sigma} \vert_{s= 0} %
=  \frac{1}{2} \mathrm{tr}_{\iota_{\Sigma}^*g_0}( \iota_{\Sigma}^*\dot {g}_0) {\rm dVol}^{g_0}_{\Sigma}\,,
\end{eqnarray}
\end{Lemma}
\begin{proof}
Suppose $(x,t)\in \Sigma\times (-\epsilon, \epsilon), \epsilon>0$ is a Fermi coordinate system for the metric $g_0 = g_s\vert_{s=0}$. %
Suppose that in this coordinate system, the metric $g_s$ has the expression
$$g_s = g_{jk;s}(x,t) dx^k dx^j + g_{tt;s}(x,t) dt^2 + \omega_s(x,t) \otimes dt  + dt\otimes \omega_s(x,t)\,,$$
where $\omega_s(\cdot, t)$ is a smooth family of one-forms on $\Sigma$ parametrized by $t\in (-\epsilon,\epsilon)$ and $g_{t,t;s}(x,t) > 0$.
Notice that the volume form ${\rm dVol}^{g_s}_{\Sigma} :={\rm dVol}^{\iota_{\Sigma}^*g_s}_{\Sigma}$ has the local coordinate expression
\begin{equation}
\label{eq: coordinate expression dvol3}
{\rm dVol}^{g_s}_{\Sigma} =|{\rm det}\left(g_{jk;s}(x,0)\right)|^{1/2}\,.
\end{equation}

By the choice of $g_0$-Fermi coordinate system, we know that 
$$g_0  = g_{jk;0}(x,t) dx^jdx^k + dt^2\,.$$
Combine \eqref{eq: coordinate expression dvol3} with Jacobi's formula to get
$$\partial_s{\rm dVol}^{g_s}_{\Sigma} \mid_{s=0}=\partial_s|{\rm det}\left(g_{jk;s}(x,0)\right)|^{1/2}\mid_{s=0} =\frac{1}{2} |{\rm det}\left(g_{jk;0}(x,0)\right)|^{1/2}{\rm tr} \left(\left(g_0(x,0)\right)^{-1} \dot g_{0}(x,0)\right) \,.$$
The right side is precisely $\frac{1}{2} \mathrm{tr}_{\iota_{\Sigma}^*g_0}(\iota_{\Sigma}^*\dot {g}_0) {\rm dVol}^{g_0}_{\Sigma}$ in coordinates.
\end{proof}

We are now in position to compute the linearization $DF$ of $F$. 
\begin{proof}[Proof of  Proposition~\ref{prop:diffF}]
Let $f\in U_\delta$, $g\in G^\circ$ with $G^\circ$ a small enough neighborhood of $g^\circ$ in $G$ and $\delta>0$ small enough. Let $g_s$ for $s \in\mathbb R$ be a smooth curve in $G^\circ$ with $g_ 0 = g$, and $H= \dot{g}_0 = \partial_s g_s\vert_{s=0}\in G$. Using \eqref{eq:dF}, we have
\begin{eqnarray}\nonumber
(DF(g)H)(f) = (D_{g}A_{g})(u_{g}^f)H  = \p_s \int_{\Sigma^\circ(f,g)} %
{\rm dVol}^{g_s}_{\Sigma^\circ(f,g)} \Big|_{s=0}\,.
\end{eqnarray}
Applying Lemma~\ref{lem: differential of volume form} we obtain \eqref{eq: differential of area functional}. When $H$ is conformal to $g$, this formula reduces to \eqref{eq: minimal surface transform}, which completes the proof of Proposition~\ref{prop:diffF}.
\end{proof}

\section{Canonical relation and Bolker condition} \label{sec_bolker}

In this section we will show that the minimal surface transform $\mathcal{R}_{\Sigma}$, restricted to a suitable finite dimensional space of Dirichlet data on $\p \Sigma$, is a double fibration transform. We also use the Runge approximation property for the stability equation to show that the Bolker condition is satisfied at every point of the canonical relation. These results are true in the $C^{\infty}$ category, which is the setting of Section \ref{subsect: heuristic}. If all underlying structures are analytic, the results are also true in the analytic category. We will work in the analytic setting from Section \ref{subsec_finite_dim_ms} on.

\subsection{Heuristic calculation}\label{subsect: heuristic}
As before, let $g^\circ$ be a fixed smooth metric on $\widetilde{M}$ and $\Sigma^\circ$ a fixed admissible minimal surface for $(\widetilde{M},g^\circ)$, from which we obtain a Fermi coordinate system for $g^\circ$ around $\Sigma^\circ$. We consider a subset of nearby minimal surfaces which can be represented as graphs $\Sigma^\circ(f,g^\circ)$ where $f\in U_\delta$, see Theorem~\ref{thm: smooth dependence}. %
In fact, throughout this subsection we replace $U_\delta$ by $W_\delta \coloneqq U_{\delta}\cap C^\infty$.

According to Theorem~\ref{thm: smooth dependence}, we have a well-defined transform on $2$-tensor fields, 
\[
	C^\infty(\widetilde{M}; \sigma(T^*\widetilde{M} \times T^*\widetilde{M}))\ni H\mapsto DF(g^\circ) H \in C(W_\delta;\mathbb R)
\]
defined via \eqref{eq: differential of area functional}. To simplify the analysis we will consider conformal deformations of metrics, in which case our transform will act on scalar functions as in \cite{MR812288,MST} (see also \cite[\S~8]{zbMATH03698077} and \cite{MR516965}). The relevant transform is the map $\beta \mapsto DF(g^\circ)\beta g^\circ$ given in \eqref{eq: minimal surface transform}, which can be rewritten as in \eqref{rsigma_def} (omitting the constant factor $\frac{n}{2}$) as the minimal surface transform  
\begin{equation}\label{eq:astrafo}
	\mathcal{R}_{\Sigma^{\circ}}: C_c^\infty(\widetilde{M}) \to C(W_{\delta}; \mR), \ \ (\mathcal{R}_{\Sigma^{\circ}} \beta)(f) = \int_{\Sigma^\circ(f,g^\circ)} \beta \,\mathrm{dVol}^{g^\circ}_{\Sigma^\circ(f,g^\circ)}.
\end{equation}
A further technical issue persists: for each $\beta\in C_c^\infty(\widetilde{M})$, $\mathcal{R}_{\Sigma^{\circ}} \beta$ takes its input from an infinite dimensional parameter space, $W_\delta$, rather than, say, some finite dimensional manifold. While this ``overdetermination'' suggests that the analysis of the transform $\mathcal{R}_{\Sigma^{\circ}}$ might become easier, infinite dimensionality prohibits us from directly applying the Fourier integral operator (FIO) machinery of \cite{MR812288,MST}.

Nevertheless, for the purpose of motivation, let us consider the (infinite dimensional) double fibration:
\begin{equation} \label{df_diagrams_heu}
\begin{tabular}{cc}
\begin{tikzcd}
& Z \arrow[dl,swap,"\pi_{W_\delta}"] \arrow[dr,"\pi_{\widetilde{M}}"] \\
W_\delta && \widetilde{M}
\end{tikzcd} \qquad & \qquad
\begin{tikzcd}
& N^*Z\setminus 0 \arrow[dl,swap,"\pi_L"] \arrow[dr,"\pi_R"] \\
T^* W_\delta \setminus 0 && T^* \widetilde{M} \setminus 0
\end{tikzcd}
\end{tabular}
\end{equation}
where the submanifold $Z\subset W_\delta \times \widetilde{M}$ is defined by
\begin{eqnarray}\label{eq: manifold Z}
Z \coloneqq \left\{(f; x,t)\in W_\delta \times \widetilde{M}\mid (x,t) \in \Sigma^\circ(f,g^\circ) \subset \widetilde{M}\right\}\,.
\end{eqnarray}
We will adopt notation from Appendix~\ref{subsec:implicitfunc} that characterizes solutions $u$ of the minimal surface equation \eqref{eq: MSE} for $(\widetilde{M},g^\circ)$ as solutions $u$ of $L_{g^\circ}(u)=0$, where $L_{g^\circ}$ is the elliptic second-order differential operator from Lemma~\ref{lem:linofmin}. For the minimal surface $\Sigma^\circ$, $u=0$ satisfies $L_{g^\circ}(0)=0$, and the stability operator of $\Sigma^\circ$ is given by $D_uL_{g^\circ}(0)$, so that solutions of \eqref{eq: stability equation} are precisely those functions $v$ satisfying $D_uL_{g^\circ}(0)v=0$. For brevity, we will denote the stability operator at $u=0$ by 
\[
Q := D_uL_{g^\circ}(0) = \Delta_{\iota^\ast_{\Sigma^\circ}g^\circ}+q.
\]

In \eqref{eq: manifold Z}, we may thus reformulate 
\begin{eqnarray}\label{eq: min surf as graph}
\Sigma^\circ(f,g^\circ) = \{(x,u^{f}_{g^\circ}(x))\colon x\in \Sigma^\circ,\ (L_{g^\circ}(u^{f}_{g^\circ}),u^{f}_{g^\circ}\vert_{\p\Sigma^\circ})=(0,f)\}\,.
\end{eqnarray}
Formally, the tangent and cotangent bundle of $W_\delta$ are identified with $C^\infty(\partial \Sigma^\circ)$ and $\mathcal{D}'(\partial \Sigma^\circ)$ respectively. %
We wish to compute the normal bundle of $Z$ near the point $(0; x,0)$ for arbitrary $x\in \Sigma^\circ$. A direct computation and an application of Lemma~\ref{lem:linofmin} show that at $(0;x, 0)\in Z$, by looking at tangent vectors of $Z$ given by curves $(f_s;x_s,t_s)$ with $t_s = u_{g^{\circ}}^{f_s}(x_s)$ where $L_{g^{\circ}}(u^f_{g^{\circ}}) = 0$, we have 
\begin{equation}\label{eq: tangent space of Z}
T_{(0; x,0)}Z = \{(\dot f_0; \dot x_0  + v(x) \partial_t) \colon \dot{f}_0 \in C^\infty(\partial\Sigma^\circ), \ \dot x_0 \in T_x\Sigma^\circ,\ (Qv,v\vert_{\p\Sigma^\circ})=(0,\dot f_0)\}\,.%
\end{equation}
By our admissibility assumption on $\Sigma^\circ$ that translates to the well-posedness of \eqref{eq: stability equation}, which is equivalent to the bijectivity of $(Q\cdot,\cdot\vert_{\p\Sigma^\circ})$ over certain function spaces (see Lemma~\ref{lem:linofmin}), the function $v$ in \eqref{eq: tangent space of Z} is uniquely determined. 

More generally, according to Lemma~\ref{lem:linofmin}, for every $f\in W_\delta$ sufficiently small, there is a linear elliptic second order differential operator $E(u_{g^\circ}^f) \colon C^{2,\gamma}(\Sigma^\circ)\to C^{0,\gamma}(\Sigma^\circ)$
so that the map 
\[
C^{2,\gamma}(\Sigma^\circ) \ni v\mapsto (E(u_{g^\circ}^f)v,v\vert_{\partial\Sigma^\circ})\in C^{0,\gamma}(\Sigma^\circ)\times C^{2,\gamma}(\partial\Sigma^\circ)
\]
is bijective. (The operator $E(u^f_{g^\circ})$ is merely $D_u L_{g^\circ}(u^f_{g^\circ})$.) Therefore, we may introduce $P(\cdot,\cdot; u^f_{g^\circ}) \in \mathcal D'(\partial \Sigma^\circ, \Sigma^\circ)$ defined as the Poisson kernel of the Dirichlet problem for $(E(u^f_{g^\circ})\cdot,\cdot\vert_{\p\Sigma^\circ})$. %
More precisely, $P(\cdot,\cdot;u^{f}_{g^\circ})$ is uniquely determined by the property that for every $h\in C^{2,\gamma}(\partial\Sigma^\circ)$, $ v(x)\coloneqq\int_{\partial\Sigma^\circ} P(y,x;u^f_{g^\circ})h(y)\dd y$ satisfies
\[
	E(u^f_{g^\circ})v=0\quad\text{and}\quad v\vert_{\partial\Sigma^\circ}= h\,.
\]
These observations allow us to calculate that the tangent bundle of $Z$ over $(f; x, u_{g^\circ}^f(x))\in W_\delta\times \widetilde{M}$ is given by
\begin{eqnarray}\label{eq: tangent space of Z'}
T_{(f; x,u_{g^\circ}^f(x))}Z &= &\left\{(\dot f_0; \dot x +( du^f_{g^\circ}(x) \cdot \dot x  + v(x) )\partial_t) \in C^\infty(\partial\Sigma^\circ)\times T_{(x,0)}\widetilde{M}\colon\right.\\\nonumber && \ \ \left.\dot x \in T_x\Sigma^\circ, \ E(u^f_{g^\circ})v = 0,\ v\vert_{\partial\Sigma^\circ} = \dot f_0\right\}.
\end{eqnarray} 
Therefore, the conormal bundle of $Z$ at $(f;x,u^f_{g^\circ}(x)) \in Z$ can be expressed as 
\begin{multline}\label{eq: conormal of Z}
N^*_{(f;x,u^f_{g^\circ}(x))} Z=\Big\{ \ (b; \xi) \in \mathcal{D}'(\partial \Sigma^\circ) \times T^*_{(x,0)}\widetilde{M}\colon \xi =\lambda(-du^f_{g^\circ}(x)+ dt), \\
 b(\cdot) = -\lambda P(\cdot,x;u^f_{g^\circ}), \ \lambda \in \mR \setminus \{0\} \Big\}\,.
\end{multline}
Note that by Lemma~\ref{lem:linofmin}, at $(0; x, 0)\in Z$, $P(\cdot, x; 0) \eqqcolon P_q(\cdot, x)$ is precisely the Poisson kernel of the Schr\"odinger operator $Q = D_uL_{g^\circ}(0) = \Delta_{\iota^\ast_{\Sigma^\circ}g^\circ}+q$ from \eqref{eq: stability equation}, which is the stability operator of $\Sigma^\circ$.

Before we move on to consider the Bolker condition, we will state for the convenience of the reader a consequence of Runge approximation that follows from \cite[Appendix A]{LLS}. 

\begin{Lemma}\label{lem:runge}
    Let $\Sigma^\circ$ be an admissible minimal surface for $(\widetilde{M},g^\circ)$. %
    For every $f\in C^{\infty}(\partial\Sigma^\circ)$ let $v^f \in C^{\infty}(\Sigma^\circ)$ be the unique solution of $Q v^f = 0$ in $\Sigma^\circ$ with $v^f\vert_{\partial\Sigma^\circ} = f$. Then: %
    \begin{enumerate}
        \item For any $x,y\in \Sigma^\circ$, $x\neq y$, there is $f\in C^{\infty}(\partial\Sigma^\circ)$ so that $v^f(x)\neq v^f(y)$.
        \item For any $x\in (\Sigma^\circ)^{\mathrm{int}}$, $a \in \mR$ and $\xi \in T_x^\ast \Sigma^\circ$ there is $f\in C^{\infty}(\partial\Sigma^\circ)$ so that $v^f(x) = a$ and $d v^f(x)=\xi$.
    \end{enumerate}
    If all underlying structures are analytic, we obtain the same statement with $f\in C^\omega(\partial\Sigma^\circ)$ and $v^f \in C^\omega(\Sigma^\circ)$.
\end{Lemma}
\begin{proof}
The admissibility assumption for $\Sigma^\circ$ ensures that $0$ is not a Dirichlet eigenvalue of the second order elliptic operator $Q$ in $\Sigma^\circ$. This shows that for any $f$ there is a unique solution $v^f$. Upon extending $\Sigma^{\circ}$ smoothly and placing sources outside $\Sigma^{\circ}$, the statements (1) and (2) follow from \cite[Proposition A.7 and its proof]{LLS}. If the underlying structures are analytic and the sources are outside $\Sigma^{\circ}$, the analyticity of $f$ and $v^f$ follows from elliptic regularity since $Q$ has analytic coefficients.
\end{proof}

We will later use a Whitney projection argument to obtain an analogue of Lemma \ref{lem:runge} where $f$ lies in an $(2n+1)$-dimensional space.

Now let us return to considering $N^\ast Z$, whose points will be denoted by $(f, b; x,  u^f_{g^{\circ}}(x), \xi)$ where $(f; x, u^f_{g^{\circ}}(x)) \in Z$ and $(b,\xi) \in N^*_{(f;x,u^f_{g^\circ}(x))} Z$. Recall that the Bolker condition for a double fibration as in \eqref{df_diagrams_heu} is that $\pi_L$ is an injective immersion. Thus let us explore whether the left projection $\pi_L : N^*Z \to T^* W_\delta$ is injective at $(0, P_q(\cdot, x_0); x_0, 0,  dt) \in N^*Z$; that is, whether there exists a $y_0\in \Sigma^\circ$ with $P_q(\cdot, y_0) = P_q(\cdot, x_0)$ and $y_0\neq x_0$. The existence of such $y_0$ is excluded by Lemma~\ref{lem:runge}. Therefore we have that for all $(0, P_q(\cdot, x_0); x_0, 0, dt) \in N^*Z$, 
\begin{eqnarray}\label{eq: formally injective}
\pi_{L}^{-1}(0, P_q(\cdot, x_0)) =\{ (0, P_q(\cdot, x_0); x_0, 0, dt)\}.
\end{eqnarray}

We now explore whether the left projection $\pi_L : N^*Z \to T^* W_\delta$ is an immersion at a given point $(0, P_q(\cdot, x_0); x_0, 0,  dt) \in N^*Z$. Using the representation of $N^*_{(f; x, u_{g^\circ}^f(x))}Z$ given in \eqref{eq: conormal of Z}, let 
\[
	(f_s,  -\lambda_s P(\cdot, x_s; u^{f_{s}}_{g^\circ}); x_s, u^{f_s}_{g^\circ}(x_s), \lambda_s(-du^{f_s}_{g^\circ}(x_s)+ dt)\,,\qquad s\in \mathbb R
\]
be a curve in $N^*Z$ with $f_0 = 0$ (so that $u^{f_0}_{g^\circ}=0$). Applying $\pi_L$ to this curve and differentiating at $s=0$ yields that
\begin{multline*}
	\partial_s\vert_{s= 0}\pi_L(f_s,  -\lambda_s P(\cdot, x_s; u^{f_{s}}_{g^\circ}); x_s, u^{f_s}_{g^\circ}(x_s), \lambda_s(-du^{f_s}_{g^\circ}(x_s)+ dt)) \\
     = (\dot f_0, -\lambda_0 d_x P_q(\cdot, x_0) \cdot \dot x_0 -\dot{\lambda}_0 P_q(\cdot,x_0) + B(\cdot, x_0) \dot f_0 )\,,
\end{multline*}
where for each $(y,x)\in \partial \Sigma^\circ \times \Sigma^\circ$, $B(y,x)\colon C^\infty(\partial \Sigma^\circ) \to \mathbb R$ is a linear map. So if the right hand side is zero, then $\dot{f}_0 = 0$ and 
\[
	\lambda_0 d_xP_q(\cdot, x_0) \cdot \dot x_0 + \dot{\lambda}_0 P_q(\cdot,x_0) = 0\,.
\]
That is, for all $f\in C^\infty(\partial \Sigma^\circ)$, 
\begin{eqnarray} \label{eq: dv cdot dot x = 0}
\lambda_0 d v^f(x_0) \cdot \dot x_0 + \dot{\lambda}_0 v^f(x_0) =\int_{\partial \Sigma^\circ} (\lambda_0 d_xP_q(y, x)\vert_{x=x_0} \cdot \dot{x}_0 + \dot{\lambda}_0 P_q(y,x_0))f(y)dy = 0
\end{eqnarray}
 where $(Q v^f,v^f\vert_{\p\Sigma^\circ})=(0,f)$. %
 By Runge approximation again (Lemma~\ref{lem:runge}), we can prescribe $dv^f(x_0)$ and $v^f(x_0)$ freely at $x_0$. Combined with \eqref{eq: dv cdot dot x = 0}, this implies that 
 \[
    \lambda_0 \dot{x}_0 = \dot{\lambda}_0 = 0.
 \]
Thus $\dot{x}_0 = \dot{\lambda}_0 = 0$, so that $d\pi_L$ is injective at all points of the form $(0, -P_q(\cdot, x_0); x_0, 0, dt)\in N^*_{(0; x_0, 0)} Z$ with $x_0\in \Sigma^\circ$.

The non-compactness of $W_\delta$ prevents us from concluding that $\pi_L$ is an injective immersion on all of $N^*Z$. Furthermore, the infinite dimensionality of $W_\delta$ means that the double fibration transform given by \eqref{df_diagrams_heu} is not in the framework of \cite{MST}. To this end we will consider double-fibration transforms over a finite dimensional set of admissible minimal surfaces.

\subsection{Finite dimensional set of minimal surfaces} \label{subsec_finite_dim_ms}

We now assume that all underlying structures are analytic. Let $g^\circ$ be a fixed analytic metric on $\widetilde{M}$ and $\Sigma^\circ\subset \widetilde{M}$ a fixed admissible minimal surface and fix Fermi coordinates for $g^\circ$ around $\Sigma^\circ$ so that
$$\{d_{g^\circ}(\cdot, \Sigma^\circ)<\epsilon\} \cong \Sigma^\circ \times (-\epsilon, \epsilon).$$
Choose $\delta>0$ so that $(L_{g^\circ}(\cdot),\cdot\vert_{\p\Sigma^\circ})=(0,f)$ is solvable for all $f\in U_\delta$, which is possible according to Proposition~\ref{prop: smooth dependence of solution}. For each finite dimensional subspace $\BB$ of $C^\omega(\partial\Sigma^\circ)$ we define
\begin{equation}\label{eq: def of G}
\G_\delta(\BB) \coloneqq U_\delta\cap \BB 
\end{equation}
and 
\begin{equation}\label{eq: def of ZB}
Z_\delta(\BB) \coloneqq \left\{(f, x,t)\in \G_\delta(\BB)\times \widetilde M\mid t = u^f_{g^\circ}(x),\ (L_{g^\circ}(u^f_{g^\circ}),u_{g^\circ}^f\vert_{\partial\Sigma^\circ})=(0,f)\right\}\,.%
\end{equation}

Recall from \cite[Def.~1.5]{MST} or \cite[\S~8]{zbMATH03698077} the notion of double fibrations. The main result of this section is 
\begin{Proposition}\label{prop: bolker}
There exists a subspace $\BB \subset C^\omega(\partial\Sigma^\circ)$ of dimension $N=2n+2$ such that if $\delta >0$ is chosen sufficiently small, the double fibration defined by
\begin{equation} \label{df_diagrams}
\begin{tabular}{cc}
\begin{tikzcd}
& Z_\delta(\BB) \arrow[dl,swap,"\pi_{\G_\delta(\BB)}"] \arrow[dr,"\pi_{\widetilde M}"] \\
\G_\delta(\BB) && \widetilde M
\end{tikzcd} \qquad & \qquad
\begin{tikzcd}
& N^*Z_\delta(\BB)\setminus 0 \arrow[dl,swap,"\pi_L"] \arrow[dr,"\pi_R"] \\
T^*\G_\delta(\BB) \setminus 0 && T^* \widetilde M \setminus 0
\end{tikzcd}
\end{tabular}
\end{equation}
 satisfies the Bolker condition; that is, $\pi_L$ is an injective immersion. Furthermore, if $(\widetilde{M},g^\circ)$ is analytic and if we choose a basis $f_1,\dots, f_N\in \BB$ and identify $\mathcal G_\delta(\BB)$ with the unit ball $B^N_1\subset\mathbb R^N$, then the interior of $Z_\delta(\BB)$ is embedded analytically in $B_1^N \times \widetilde M$.
\end{Proposition}
After the possible introduction of a suitable weight function, double fibrations give rise to double fibration transforms defined in \cite[Def.~1.5]{MST}, which according to \cite[Thm.~2.2]{MST} are FIOs. %
In what follows, for any $N\in\mathbb{N}$ we use the notation 
\[
\bar B_1^N = \{ z \in \mR^N \mid |z_1| + \ldots + |z_N| \leq 1 \}
\]
to denote the closed unit ball in $(\mathbb R^{N}, \ell^1)$. We use the $\ell^1$ norm to ensure that if $\|f_j\| \leq \delta$ for $1 \leq j \leq N$, then for $z \in \bar B_1^N$ we still have $\| z_1 f_1 + \ldots z_N f_N \| \leq \delta$. Akin to the representation \eqref{eq:astrafo}, a straightforward application of Proposition~\ref{prop: bolker} gives 

\begin{Corollary}\label{cor: bolker}
Let $\Sigma^\circ\subset \widetilde{M}$ be an admissible minimal surface with respect to the metric $g^\circ$. Let also $N = 2n+2$. There are $\delta>0$ small 
and
\[
f_1,\dots, f_{N}\in C^\omega(\partial\Sigma^\circ) \text{ with } \|f_j\|_{C^{2,\gamma}}  = \delta
\]
so that the operator $DF(g^\circ) : C^\infty_c(\widetilde{M}) \to C(\bar B_1^N)$ defined as%
\[
(DF(g^\circ) h)(z) = (DF(g^\circ) h)(z_1 f_1+\dots + z_{N} f_{N}),
\]
with the right hand side being the minimal surface transform in \eqref{eq: minimal surface transform}, is a Fourier integral operator satisfying the Bolker condition at all points of its canonical relation. When $(\widetilde{M},g^\circ)$ is analytic, this operator is an analytic double fibration transform.
\end{Corollary}
Recall that $\Sigma^\circ = \ol{\Sigma}^\circ$ is closed and compact. In order to give the proof of Proposition~\ref{prop: bolker} we begin with 
\begin{Lemma}\label{lem: diagonal separation}
There exists an $\epsilon>0$ and a finite set $\SSS'\subset C^\omega(\p \Sigma^\circ)$ such that for all $x, y \in \Sigma^\circ$ with $x\neq y$ and $d_{\Sigma^\circ}(x,y) < \epsilon /4$, there exists a solution $v\in C^\omega(\Sigma^\circ)$ of 
\begin{equation} \label{eq_dulgv_zero}
Qv = D_uL_{g^\circ}(0)v=(\Delta_{\iota^\ast_{\Sigma^\circ}g^\circ}+q)v=0
\end{equation}
with boundary condition $v\vert_{\partial \Sigma^\circ} \in \SSS'$ such that $v(x) \neq v(y)$. Furthermore, for all points $x\in\Sigma^\circ$, 
\begin{eqnarray}\label{eq: full span}
{\rm span}\{dv(x) \colon v\ 
\text{{\rm solves} \eqref{eq_dulgv_zero}},\ 
v\vert_{\partial \Sigma^\circ} \in \SSS'\} = T^*_x\Sigma^\circ.
\end{eqnarray}
\end{Lemma}
\begin{proof}
By Lemma~\ref{lem:runge}, for all $x\in \Sigma^\circ$, there are $C^\omega$ solutions $(u_1(\ccdot; x), \dots u_n(\ccdot;x))$ of \eqref{eq_dulgv_zero} in $\Sigma^\circ$ such that $\{du_j(x;x)\}_{j=1}^n$ is linearly independent. As such, in a small ball $B_{\epsilon_x}(x)$ centered around $x$ of radius $\epsilon_x>0$, 
$$\big(u_1(\ccdot; x), \dots u_n(\ccdot;x)\big): B_{\epsilon_x}(x) \to \mathbb R^n$$
forms a coordinate system  of $\Sigma^\circ$ near $x$. %

Define the finite set $\SSS_x := \{u_1(\ccdot; x)\vert_{\partial \Sigma^\circ}, \dots, u_n(\ccdot;x)\vert_{\partial \Sigma^\circ}\}$. Now consider the covering of $\Sigma^\circ$ given by $\{B_{\frac{\epsilon_x}{2}}(x)\colon x\in \Sigma^\circ\}$. By compactness, there exist points $x_1,\dots, x_N\in \Sigma^\circ$ such that the set $\{B_{\frac{\epsilon_{x_1}}{2}}(x_1),\dots, B_{\frac{\epsilon_{x_N}}{2}}(x_N)\}$ is a finite cover of $\Sigma^\circ$. Set $\SSS' := \bigcup\limits_{j = 1}^N \SSS_{x_j}$.

By the fact that for each $x_j$, $j=1,\dots, N$, $$(u_1(\ccdot; x_j), \dots, u_n(\ccdot;x_j)): B_{\epsilon_x}(x) \to \mathbb R^n$$
is a coordinate system near $x_j$, \eqref{eq: full span} is automatically satisfied. Now choose $\epsilon \coloneqq {\rm min}\{\epsilon_{x_1},\dots, \epsilon_{x_N}\}$ so that any two distinct points $x,y\in\Sigma^\circ$ with $d_{\Sigma^\circ}(x,y)<\epsilon/4$ must both belong to $B_{\epsilon_{x_j}}(x_j)$ for some $j=1,\dots, N$. Using again the fact that 
$$\big(u_1(\ccdot; x_j), \dots, u_n(\ccdot;x_j)\big): B_{\epsilon_x}(x) \to \mathbb R^n$$
is a coordinate system, there must be $u_l(\ccdot; x_j)$, $l\in 1,\dots, N$ which separates these two points and satisfies $u_l(\ccdot; x_j)\vert_{\partial\Sigma^\circ} \in \SSS_{x_j}\subset \SSS'$. The proof is complete.
\end{proof}

\begin{Lemma}
\label{lem: off diagonal separation}
There exists a finite set $\SSS\subset C^\omega(\p\Sigma^\circ)$ such that the set 
$$\{v\in C^\omega(\Sigma^\circ)\colon %
v\ {\rm solves}\ \eqref{eq_dulgv_zero},\  v\vert_{\partial\Sigma^\circ}\in \SSS\}$$
separates points on $\Sigma^\circ$. Furthermore, for all points $x\in\Sigma^\circ$, 
\begin{eqnarray}\label{eq: full span'}
{\rm span}\{dv(x) \colon %
v\ {\rm solves}\ \eqref{eq_dulgv_zero}, \
v\vert_{\partial \Sigma^\circ} \in \SSS\} = T^*_x\Sigma^\circ.
\end{eqnarray}
\end{Lemma}
\begin{proof}
Consider the product manifold
$$\Sigma^\circ\times\Sigma^\circ \setminus \{(x,y)\in \Sigma^\circ\times\Sigma^\circ\mid d_{\Sigma^\circ}(x,y)<\epsilon/8\}$$
where $\epsilon>0$ is given by Lemma \ref{lem: diagonal separation}. For any $(x,y)$ in this set, we can use Proposition A.7 (a) of \cite{LLS} to find $v(\ccdot; x,y)\in C^{\infty}(\Sigma^\circ)$ solving \eqref{eq_dulgv_zero} on $\Sigma^\circ$ such that $v(x;x,y) \neq v(y; x,y)$. By perturbation, if necessary, we may assume that $v(\ccdot; x,y)|_{\partial \Sigma^\circ}\in C^\omega(\partial \Sigma^\circ)$ and $v(\ccdot; x,y)\in C^\omega(\Sigma^\circ)$. By continuity, there is $\epsilon_{x,y}>0$ such that \begin{eqnarray}\label{eq: empty intersection}
v(\overline{B_{\epsilon_{x,y}}(x)}; x,y)\cap v(\overline{B_{\epsilon_{x,y}}(y)}; x,y) = \emptyset.
\end{eqnarray}
The set
$$\{B_{\epsilon_{x,y}}(x) \times B_{\epsilon_{x,y}}(y)\colon d_{\Sigma^\circ}(x,y) \geq \epsilon/8\}$$
is clearly an open cover of the compact set 
$$\Sigma^\circ\times\Sigma^\circ \setminus \{(x,y)\in \Sigma^\circ\times \Sigma^\circ\colon d_{\Sigma^\circ}(x,y)<\epsilon/8\}$$
So we may find a finite cover $\{B_{\epsilon_{x_j,y_j}}(x_j) \times B_{\epsilon_{x_j,y_j}}(y_j)\colon j = 1,\dots, N\}$ with corresponding solutions
$\{v(\ccdot; x_j,y_j)\colon j = 1, \dots, N\}$ of \eqref{eq: stability equation}. Set $\SSS'' := \{v(\ccdot; x_1,y_1)\vert_{\partial\Sigma^\circ},\dots, v(\ccdot; x_N,y_N)\vert_{\partial\Sigma^\circ} \}$. Then for any two points $x,y\in \Sigma^\circ$, with $d_{\Sigma^\circ}(x,y) \geq \epsilon/8$, there exists $j\in\{1,\dots, N\}$ such that $(x,y) \in B_{\epsilon_{x_j, y_j}}(x_j) \times B_{\epsilon_{x_j, y_j}}(y_j)$ which by \eqref{eq: empty intersection} means that $v(x; x_j, y_j) \neq v(y;x_j, y_j)$ with $v(\ccdot ;x_j,y_j)\vert_{\partial \Sigma^\circ} \in \SSS''$. Now set $\SSS := \SSS'\cup \SSS''$ where $\SSS'$ is as in the statement of Lemma \ref{lem: diagonal separation}. The condition \eqref{eq: full span'} is satisfied by Lemma~\ref{lem: diagonal separation}. Furthermore, solutions of \eqref{eq_dulgv_zero} with boundary values in $\SSS$ separate distinct points $x,y\in \Sigma^\circ$ both when $d_{\Sigma^\circ}(x,y)\geq \epsilon/8$ and when $d_{\Sigma^\circ}(x,y) <\epsilon/8$. This completes the proof.
\end{proof}

We will next improve Lemma \ref{lem: off diagonal separation} and show that one can take $\SSS$ to have cardinality $2n+1$. By using Runge approximation and a Whitney projection method, it was proved in \cite{greene1975whitney} that an \(n\)-dimensional Riemannian manifold can be embedded into \(\mathbb{R}^{2n+1}\) by a mapping whose component functions are harmonic. In the subsequent work \cite{greene1975embedding}, they generalized this result to solutions of elliptic equations lacking a zeroth-order term.

Using Lemmas~\ref{lem: diagonal separation} and~\ref{lem: off diagonal separation}, we extend these results to include a zeroth-order term. An immediate consequence is that the set \(\SSS\) in Lemma \ref{lem: off diagonal separation} can be chosen to have (possibly) reduced cardinality \(2n+1\). In order to have a double fibration transform later, we will need that all of the solutions cannot vanish at a fixed point, and this can be arranged by adding one more solution.%

 \begin{Proposition}\label{prop:GreeneWu_generalization}
Let $Q = \Delta_{\iota^\ast_{\Sigma^\circ}g^\circ}+q$ be as above, where the coefficients are in $C^\omega(\Sigma)$. Then there is a $C^\omega$ embedding of $\Sigma^{\circ}$ in $\mathbb{R}^{2n+1}$, whose component functions are solutions of
\[
Qv = 0 \text{ on } \Sigma \text{ and } v|_{\p \Sigma} \in C^\omega\,.
\]
If the embedding is taken into $\mathbb R^{2n+2}$, then the embedding can be chosen to avoid the origin.

\end{Proposition}

\begin{proof}
We extend $\Sigma$ to a slightly larger, connected, non-compact $C^\omega$ manifold $\widetilde{\Sigma}$ and extend the coefficients of $Q$ to $\widetilde{\Sigma}$ with the same regularity. Let $\SSS = \{f_1, \dots, f_N\}$ be as in Lemma \ref{lem: off diagonal separation}, and let $\{v^{f_1}, \dots, v^{f_N}\}$ be the corresponding solutions to $Qv=0$ in $\Sigma$ with $v^{f_j}|_{\p \Sigma} = f_j$.

While Lemma 4 of \cite{greene1975embedding} constructs an embedding into $\mathbb{R}^l$ using harmonic functions,  we instead consider the map
\[
F = (v^{f_1}, \ldots, v^{f_N}): \Sigma \to \mathbb{R}^N
\]
and set $l=N$. To confirm that $F$ is an embedding, note that injectivity follows because $\SSS$ separates points, and the injectivity of its differential is guaranteed by condition \eqref{eq: full span'}. With this, the remainder of the proof then follows identically to that in \cite{greene1975embedding}.  (In the proof of Lemma 11 in \cite{greene1975embedding}, one applies Lemma \ref{lem:runge}.)  %

To show that we can arrange the embedding to avoid the origin by adding one more solution, suppose there exists a point $x_0\in \Sigma$ such that 
$$v^{f_1} (x_0) = \dots = v^{f_{2n+1}}(x_0) = 0\,.$$
If such a point exists, it would be the unique point mapping to the origin since the mapping is an embedding. Now choose $f_{2n+2}\in C^\omega(\partial \Sigma)$ such that $v^{f_{2n+2}}(x_0) \neq 0$ and we are done.
\end{proof}

\begin{Remark}
Proposition \ref{prop:GreeneWu_generalization} also holds for general second order elliptic operators with analytic coefficients, and if the coefficients are smooth then the corresponding result holds in the smooth category. This is true since Runge approximation and consequently Lemmas \ref{lem: diagonal separation} and \ref{lem: off diagonal separation} hold for such operators (see \cite[Proposition A.7]{LLS}).
\end{Remark}

Using Proposition \ref{prop:GreeneWu_generalization}, we immediately obtain:

\begin{Corollary}\label{cor:Greene_wu_cor}
There exists a set $\SSS = \{ f_1, \ldots, f_{2n+2} \} \subset C^\omega(\partial\Sigma)$ such that the set 
$$\{v\in C^\omega(\Sigma) \colon %
Qv=0,\  v\vert_{\partial\Sigma}\in \SSS\}$$
separates points on $\Sigma$, the map 
\begin{eqnarray}\label{eq: embedding avoiding zero}
x\mapsto (v^{f_1}(x),\dots , v^{f_{2n+2}}(x))
\end{eqnarray}
avoids the origin, and for all $x\in\Sigma$, 
\begin{eqnarray}\label{eq: full span''}
{\rm span}\{dv(x) \colon %
Qv=0, \  
v\vert_{\partial \Sigma} \in \SSS\} = T^*_x\Sigma.
\end{eqnarray}
\end{Corollary}
\begin{proof}
Let $(v_1, \ldots, v_{2n+2})$ be the embedding from Proposition \ref{prop:GreeneWu_generalization} that avoids the origin, let $f_j = v_j|_{\p \Sigma}$, and define 
\[
\SSS = \{ f_1, \ldots, f_{2n+2} \}.
\]
Then $v_j = v^{f_j}$, and the result follows since the embedding $(v_1, \ldots, v_{2n+2})$ is injective, avoids the origin, and its differential is injective.
\end{proof}

We now proceed toward the proof of Proposition \ref{prop: bolker}. Let $\Sigma = \Sigma^\circ$ and $N=2n+2$, and let $v^{f_1}, \ldots, v^{f_{N}}$ be the solutions in Corollary \ref{cor:Greene_wu_cor}. We define the finite dimensional space $\BB$ as 
\[
\BB = \operatorname{span}\{ f_1, \ldots, f_N \}.
\]
Let $\{f_1^*, \dots, f_N^*\}$ be the dual basis in $\mathcal{D}'(\partial\Sigma^\circ)$ to $\{f_1, \dots, f_N\}$, and define $\BB^* = \operatorname{span}(f_1^*,\dots, f_N^*)$ (in other words, $f^\ast_j$ is the distribution mapping $C^\infty(\partial\Sigma^\circ)\ni \phi\mapsto \langle f^\ast_j ,\phi\rangle =  \int_{\partial\Sigma^\circ} f_j\phi $).
         We consider now $\G_\delta(\BB)$ and $Z_\delta(\BB)$ as defined in \eqref{eq: def of G} and \eqref{eq: def of ZB} respectively for this choice of $\BB$. The cotangent bundle $T^*\G_\delta(\BB)$ can then be identified with $\G_\delta(\BB)\times \BB^*$. Let  $P(\cdot, x; u^f_{g^\circ}) \in \mathcal{D}'(\partial\Sigma^\circ)$ be the Poisson kernel which appeared in \eqref{eq: conormal of Z} and set 
\begin{equation}\label{eq: P as a covector}
	P(\cdot, x; u^f_{g^\circ})\vert_\BB \coloneqq \sum\limits_{j=1}^N P(f_j, x; u^f_{g^\circ}) \langle f_j^*,\cdot \rangle\,,\qquad P_q(\cdot,x)\vert_\BB \coloneqq P(\cdot, x; 0)\vert_\BB\,.
\end{equation}

We now prove a lemma which will be useful for ignoring the conic direction in $N^* Z_{\delta}(\BB)$.

\begin{Lemma}
\label{lem: conic no matter}
Let $X$ and $Y$ be smooth manifolds and let $\mathcal L \subset (T^*Y \times T^*X) \setminus 0$ be an embedded conic Lagrangian submanifold satisfying the following condition:
\begin{eqnarray}
    \label{eq: no flat spots}
    \text{there is no } (x,\xi)\in T^*X\setminus 0 \text{ such that } (y,0;x,\xi) \in\mathcal L.
\end{eqnarray}
Suppose the left projection $\bar \pi_L : \mathcal L/\mathbb R_+ \to T^*Y/\mathbb R_+$ is an injective immersion. Then the left projection $\pi_L : \mathcal L \to T^*Y$ is an injective immersion.
\end{Lemma}
\begin{proof}
We need to show that $\pi_L : \mathcal L \to T^*Y$ is an injective immersion. To this end, we note that due to \eqref{eq: no flat spots}, for any auxiliary metric $\|\cdot\|$ on $Y$,
$$\mathcal L/\mathbb R_+ \cong \overline{\mathcal L} := \{(y,\eta; x,\xi)\in \mathcal L \mid \|\eta\| = 1\}$$
so that $\bar \pi_L :  \overline{\mathcal L} \to S^*Y$ 
is an injective immersion by assumption. Now suppose that $(y_0,\eta_0; x_1,\xi_1)$, $(y_0,\eta_0; x_2,\xi_2)\in \mathcal L$ are such that their images under $\pi_L$ coincide. By \eqref{eq: no flat spots}, we can choose $\lambda>0$ so that $(y_0,\lambda\eta_0; x_1,\lambda\xi_1), (y_0,\lambda\eta_0; x_2,\lambda\xi_2) \in \overline{\mathcal L} $. Using that $\bar\pi_L$ is an injection, we see that $(x_1,\xi_1) = (x_2,\xi_2)$, which makes $\pi_L$ injective. 

To see that $d\pi_L$ is injective at all points on $\mathcal L$, it suffices to show that it is an injection on $T_{(y,\eta; x,\xi)}\mathcal L$
for $(y,\eta; x,\xi)\in\overline{\mathcal L}$. To this end, we write
$$ \mathcal L =  \overline{\mathcal L} \times \mathbb R_+$$
so that $T_{(y,\eta; x,\xi)}\mathcal L = T_{(y,\eta; x,\xi)}\overline{\mathcal L} \oplus \mathbb R$ for all $(y,\eta; x,\xi) \in \overline{\mathcal L}$. Here and below, $\mathbb R$ denotes the space of tangent vectors corresponding to the conic direction. Now suppose $V \in T_{(y,\eta; x,\xi)}\mathcal L$ is in the kernel of $d\pi_L$. We write $V = \overline V \oplus R$ for $\overline V \in T_{(y,\eta; x, \xi)}\overline{\mathcal L}$ and $R\in \mathbb R$. We then have that 
\begin{eqnarray}
\label{eq: dpiL is zero}
0 = d\pi_L\mid_{(y,\eta; x,\xi)} V = d\bar \pi_L\mid_{(y,\eta; x,\xi)} \overline{V} + d\pi_L\mid_{(y,\eta; x,\xi)} R \in T_{(y,\eta)} S^*Y \oplus \mathbb R.
\end{eqnarray}
By construction, if $\overline V \in T \overline{\mathcal L}$, then $d\pi_L \overline V \in TS^*Y$ and similarly $d\pi_L R \in \mathbb R$. Using \eqref{eq: dpiL is zero}, we see that  
\[
d\bar \pi_L\mid_{(y,\eta; x,\xi)} \overline{V} \in T_{(y,\eta)} S^*Y \cap \mathbb R  =\{0\}.
\]
This implies that $\overline V = 0$ by our assumption that $\bar\pi_L$ is an immersion. Going back to \eqref{eq: dpiL is zero}, we see that $d\pi_L R = 0$ for some vector $R\in T_{(y,\eta; x, \xi)} \mathcal L$ in the conic direction. By \eqref{eq: no flat spots}, this means that $R = 0$. \end{proof}

We now consider the double fibration given by \eqref{df_diagrams}.

\begin{Lemma} \label{lem: injectivity at f = 0}
The set $Z_\delta(\BB)$ has the following properties.
\begin{enumerate}
\item[(i)]
The conormal bundle $N^*Z_\delta(\BB)$ is given by 
\begin{eqnarray}\label{eq: N*ZB}\nonumber
N^* Z_\delta(\BB) = \left\{\left(f, \lambda P(\cdot, x; u^f_{g^\circ})\vert_\BB; x, u_{g^\circ}^f(x),\lambda( - du_{g^\circ}^f(x) + dt )\right)\in T^* \G_\delta(\BB)\times T^*\widetilde M\colon x\in \Sigma^\circ, \lambda >0 \right\}
\end{eqnarray}
where $u^f_{g^\circ}$ solves \eqref{eq: MSE}, ie. $L_{g^\circ}(u^f_{g^\circ})=0$, and $u^f_{g^\circ}\vert_{\p\Sigma^\circ}=f$.
\item[(ii)] 
For $\delta>0$ sufficiently small, the projections $\pi_{\mathcal G_\delta(\BB)} :Z_\delta (\BB) \to \mathcal G_\delta(\BB)$ and $\pi_{\widetilde M}: Z_\delta(\BB) \to \widetilde M$ are submersions.
\item[(iii)]
We have that
\begin{multline}
    \label{eq: modded out conormal}
N^*Z_\delta(\BB)/\mathbb R_+ \cong \overline{\mathcal L}_\delta := \\
 \left\{\left(f,  P(\cdot, x; u^f_{g^\circ})\vert_\BB; x, u_{g^\circ}^f(x), - du_{g^\circ}^f(x) + dt \right)\in T^* \G_\delta(\BB)\times T^*\widetilde M\colon x\in \Sigma^\circ \right\}.
\end{multline}
Denote by $\bar\pi_L$ the left projection acting on $\overline{\mathcal L}_\delta$. For each $x\in \Sigma^\circ$, 
\begin{eqnarray}\label{eq: finite dim inj}
    \bar \pi_L^{-1}(0, P_q(\cdot, x)\vert_\BB) = \{(0, P_q(\cdot, x)\vert_\BB; x, 0, dt)\}\,,
\end{eqnarray}
and 
\begin{eqnarray}\label{eq: finite dim immersion}
d \bar \pi_L(0, P_q(\cdot, x)\vert_\BB; x, 0, dt)\colon T_{(0, P_q(\cdot, x)\vert_\BB; x, 0, dt)}\overline{\mathcal L}_\delta \to T(\G_\delta(\BB) \times \BB^*)
\end{eqnarray}
is injective.
\end{enumerate}
\end{Lemma}
\begin{proof}
(i) The expression for $N^*Z_\delta(\BB)$ can be obtained following the heuristic computation in Section \ref{subsect: heuristic} so we will not repeat it here.

(ii) By Lemma 5.3 of \cite{MST} it suffices to check that 
\begin{eqnarray}
    \label{eq: left is ok}
\text{there are no } (x,\xi) \in T^*\widetilde M,\ y\in \mathcal G_\delta(\BB) \text{ such that } (y,0;x,\xi) \in N^*Z_\delta(\BB)\setminus 0
\end{eqnarray}
and \begin{eqnarray}
    \label{eq: right is ok}
\text{there are no } (y,\eta) \in T^*\mathcal G_\delta(\BB),\ x\in \widetilde M \text{ such that } (y,\eta;x,0) \in N^*Z_\delta(\BB)\setminus 0.
\end{eqnarray}
Condition \eqref{eq: right is ok} can be verified by observing that $-du_{g^\circ}^f(x) + dt$ never vanishes. To see \eqref{eq: left is ok}, we need to check that $P(\cdot, x; u_{g^\circ}^f)\vert_{\BB} \neq 0$ for all $x\in \Sigma^\circ$. In the particular case $f = 0$ we observe first that for all $\phi\in \BB$,
\[\langle P(\cdot, x; 0) \vert_{\BB}, \phi\rangle = \langle P_q(\cdot, x) \vert_{\BB}, \phi\rangle = v^\phi(x)\]
where $v^\phi$ is the solution of \eqref{eq_dulgv_zero} with boundary value $\phi\in \BB$. We have chosen $\BB = \operatorname{span}(\SSS)$ and by Corollary \ref{cor:Greene_wu_cor}, the embedding via \eqref{eq: embedding avoiding zero} avoids the origin. Therefore, for all $x\in \Sigma^\circ$ we can find an element $\phi\in \BB$ such that $\langle P(\cdot, x; 0) \vert_{\BB}, \phi\rangle \neq 0$. Therefore, $P(\cdot, x; 0)\vert_{\BB}\neq 0$ for all $x\in \Sigma^\circ$, which shows that \eqref{eq: left is ok} holds when $f= 0$.

For general $f\in \mathcal G_\delta(\BB)$ sufficiently small, we will apply a perturbation argument. For any fixed $x_0 \in \Sigma^\circ$, let $\phi_0 \in \BB$ so that $\langle P(\cdot,x_0;0)\vert_\BB, \phi_0 \rangle  \neq 0$. 
According to Corollary~\ref{cor:contDependenceOfStabilitySol} and Proposition~\ref{prop: smooth dependence of solution}, the map 
\[
    C^{2,\gamma}(\p\Sigma^\circ)\times C^{2,\gamma}(\p\Sigma^\circ) \ni (h,f)\mapsto \int_{\partial \Sigma^\circ}P(y, \cdot; u^f_{g^\circ}) h(y) dy \in C^{2,\gamma}(\Sigma^\circ)
\]
is continuous. In particular, $$(x,f)\mapsto \int_{\partial\Sigma^\circ}P(y,x; u^f_{g^\circ})\phi_0(y) dy \in \mR$$ 
as a map from $\Sigma^\circ \times C^{2,\gamma}(\p\Sigma^\circ)$ to $\mathbb R$ is continuous. Therefore, there are open neighborhoods $U_0$ of $x_0$ in $\Sigma^\circ$ and $W_0$ of $0$ in $C^{2,\gamma}(\p\Sigma^\circ)$ so that $\int_{\partial \Sigma^\circ}P(y,x;u^f_{g^\circ}) \phi_0(y) dy\neq 0$ for $(x,f)\in U_0\times W_0$. In particular, $P(\cdot, x; u^f_{g^\circ})\vert_\BB \neq 0$ for $(x,f)\in U_0\times W_0$. 

Using the compactness of $\Sigma^\circ$ and the fact that $x_0\in \Sigma^\circ$ was chosen arbitrarily above, we cover $\Sigma^\circ$ by open neighborhoods $U_1,\dots, U_m$ so that there exist open neighborhoods $W_1,\dots, W_m$ of $0$ in $C^{2,\gamma}(\p\Sigma^\circ)$ so that $P(\cdot, x; u^f_{g^\circ})\vert_\BB \neq 0$ for $(x,f)\in U_j\times W_j, j\in \{1,\dots, m\}$. In particular, for any $f \in \bigcap_{j=1}^m W_j$, we find that $P(\cdot, x; u^f_{g^\circ})\vert_\BB \neq 0$ for all $x\in \bigcup_{j=1}^m U_j \supset \Sigma^\circ$. 

In conclusion, by choosing $\delta>0$ sufficiently small, for every $f\in \mathcal{G}_\delta(\BB)$ and every $x\in\Sigma^\circ$, $P(\cdot,x;u^f_{g^\circ})\vert_\BB \neq 0$, completing the verification of \eqref{eq: right is ok}.

(iii) Equation \eqref{eq: modded out conormal} is obvious.
To check \eqref{eq: finite dim inj}, following the calculation in Section \ref{subsect: heuristic} shows that this is equivalent to showing that $x\mapsto P_q(\cdot, x)\vert_\BB$ from $\Sigma^\circ \to \BB^*$ is injective. But by Corollary \ref{cor:Greene_wu_cor} solutions $v$ to $Qv=0$ with boundary conditions in $\BB$ separate points in $\Sigma^\circ$. So \eqref{eq: finite dim inj} holds.

Next we follow the calculation in Section \ref{subsect: heuristic}, with the modification that \eqref{eq: dv cdot dot x = 0} holds with $\dot{\lambda}_0 = 0$ due to that fact that the conic direction was modded out in \eqref{eq: modded out conormal}. Thus \eqref{eq: finite dim immersion} is equivalent to showing that at all points $x\in \Sigma^\circ$, 
$${\rm span}\{dv(x) \colon Qv=0,\  v\vert_{\partial\Sigma^\circ} \in \BB \} = T^*_{x}\Sigma^\circ\,.$$
This is precisely what $\BB$ is constructed to do by Corollary \ref{cor:Greene_wu_cor}. The proof is complete.
\end{proof}

We need to extend the injectivity part in Lemma \ref{lem: injectivity at f = 0} (iii) from the case $f=0$ to small $f$.
\begin{Lemma}
\label{lem: nbhds of injectivity}
Let $\delta_0>0$ be sufficiently small so that $Z_{\delta_0}(\BB)$ is well-defined. There is a finite collection of open subsets $U_j\subset \overline{\mathcal L}_{\delta_0}, j\in\{1,\dots,J\}$ with $\{(0, P_q(\cdot, x)\vert_\BB; x, 0, dt)\colon x\in \Sigma^\circ\}\subset \bigcup_{j=1}^J U_j$ such that if for some $j\in \{1,\dots,J\}$, 
\[
	(f, P(\cdot, y; u_{g^\circ}^{f})\vert_{\BB} ;y, u_{g^\circ}^{f} (y),  -du_{g^\circ}^{f}(y) + dt )\in U_j\,,\  (f, P(\cdot, x; u_{g^\circ}^{f})\vert_{\BB} ;x, u_{g^\circ}^{f} (x),  -du_{g^\circ}^{f}(x) + dt )\in U_j
\]
and $ (f, P(\cdot, y; u_{g^\circ}^{f})\vert_{\BB}) = (f, P(\cdot, x; u_{g^\circ}^{f})\vert_{\BB})$, then $x=y$.
\end{Lemma}
\begin{proof}
As a consequence of Lemma \ref{lem: injectivity at f = 0} (iii), the map $d\bar \pi_L: T\overline{\mathcal L}_{\delta_0} \to TT^*\mathcal G_{\delta_0}(\BB)$ is injective on $\{(0, P_q(\cdot, x)\vert_\BB; x, 0, dt)\colon x\in \Sigma^\circ\}$. This implies that for each $x\in \Sigma^\circ$, $\bar\pi_L : \overline{\mathcal L}_{\delta_0} \to T^*\mathcal G_{\delta_0}(\BB)$ is injective in a small neighborhood of $(0, P_q(\cdot, x)\vert_\BB; x, 0, dt)$ in $ N^*Z_{\delta_0}(\BB)$. By compactness, we can find a finite open cover $\{U_j \subset \overline{\mathcal L}_{\delta_0}\}_{j=1}^J$ of 
$\{(0, P_q(\cdot, x)\vert_\BB; x, 0, dt)\colon x\in \Sigma^\circ\}$
such that the map $\bar\pi_L: U_j \to T^*\mathcal G_{\delta_0}(\BB)$ is injective for each $j= 1,\dots,J$. 
\end{proof}

Finally, we verify that under our assumptions the minimal surfaces and the solutions are analytic.

\begin{Lemma}
    \label{lem: min surfs are analytic}
    Let $(\widetilde{M}, g^\circ)$ be an analytic Riemannian manifold with boundary $\partial \widetilde{M}$ and $M\Subset \widetilde{M}$ be an open subset. Suppose $\Sigma \subset \widetilde{M}$ is a smooth embedded minimal surface with boundary $\partial{\Sigma}\subset \widetilde{M}\setminus M$. Then $\Sigma^{\rm int}$ is an analytic hypersurface.  

    Furthermore, if $w$ solves \eqref{eq: MSE} on $\Sigma^\circ$ with $w\vert_{\p\Sigma^\circ}\in C^\omega(\p\Sigma^\circ)$, then $w\in C^\omega(\Sigma^\circ)$. 
\end{Lemma}
\begin{proof}
It suffices to show that $\Sigma$ is an analytic embedding in a neighborhood of every point $p\in \Sigma^{\rm int}$. To this end, suppose $p_0\in \Sigma^{\rm int}$. Since $\Sigma\subset \widetilde{M}$ is a smooth embedding, there exists an open subset $\Omega \subset \Sigma$ containing $p_0$ and an analytic hypersurface $\mathcal H$ intersecting $\Omega$ tangentially at $p_0$. In (analytic) Fermi coordinates normal to $\mathcal H$,
$$\Omega = \{(x', x_n) \mid x_n = u(x')\}\,,$$
with $u(x')$ solving the nonlinear elliptic PDE \eqref{eq: MSE}. According to Lemma~\ref{lem:linofmin}, the linearization of \eqref{eq: MSE} at $u$ is an elliptic differential operator so that $u$ is an elliptic solution of \eqref{eq: MSE} so that we have  $u\in C^\omega(\Omega)$ according to \cite[Appdx.~Thm.~41]{MR867684} (see also \cite[Cor.~1.4]{arXiv:1506.06224}). The same references also prove the second statement of this lemma.
\end{proof}

We can now give the 
\begin{proof}[Proof of Proposition~\ref{prop: bolker}]

Let $\delta_0>0$ be sufficiently small. Then by Lemma \ref{lem: injectivity at f = 0} (ii), diagram \eqref{df_diagrams} gives a double fibration transform in the sense of \cite{MST}.

To verify the Bolker condition, it suffices by Lemma \ref{lem: conic no matter} to show that $\bar\pi_L : \overline{\mathcal L}_{\delta} \to T^* \mathcal G_{\delta}(\BB)$ is an injective immersion for some $\delta>0$. Let $\{U_k\}_{k=1}^J$ be the collection of open sets in $\overline{\mathcal L}_{\delta_0}$ constructed in Lemma \ref{lem: nbhds of injectivity} such that
\begin{eqnarray}
\label{eq: in the union}
\{(0, P_q(\cdot, x)\vert_\BB; x, 0, dt)\colon x\in \Sigma^\circ\} \subset \bigcup_{j=1}^J U_j\,.
\end{eqnarray}
By Lemma \ref{lem: injectivity at f = 0} we have injectivity of $\bar \pi_L$ and $d\bar \pi_L$ on the submanifold
\[
	\{(0, P_q(\cdot, x)\vert_\BB; x, 0, dt)\colon x\in \Sigma^\circ\}\subset N^*Z_\delta(\BB)\,.
\]
We now need to choose $\delta>0$ small enough so that this property extends to all of $\overline{\mathcal L}_{\delta_0}$ via a perturbation argument. First, we look at injectivity of $\bar\pi_L$. Suppose to the contrary that, for any $\delta>0$, $\bar\pi_L$ is not injective on $\overline{\mathcal L}_{\delta_0}$. This implies that for any $j \geq 1$, there exist distinct $x_j,y_j\in\Sigma^\circ$ and $f_j\in \G_{1/j}(\BB)$ for which 
\begin{multline}\label{eq: non-injective}
\bar \pi_L\left((f_j, P(\cdot, x_j; u_{g^\circ}^{f_j})\vert_{\BB} ;x_j, u_{g^\circ}^{f_j} (x_j), -du_{g^\circ}^{f_j}(x_j) + dt )\right) \\
 = \bar\pi_L\left((f_j, P(\cdot, y_j; u_{g^\circ}^{f_j})\vert_{\BB} ;y_j, u_{g^\circ}^{f_j} (y_j), -du_{g^\circ}^{f_j}(y_j) + dt ) \right).
\end{multline}
By the definition of $\G_\delta(\BB)$, $f_j\to 0$ in $C^{2,\gamma}(\partial\Sigma^\circ)$, which implies that $u^{f_j}_{g^\circ} \to 0$ in $C^{2,\gamma}(\Sigma^\circ)$ by Proposition~\ref{prop: smooth dependence of solution}. By compactness we may assume, after passing to a subsequence, that $x_j\to x_0\in \Sigma^\circ$ and $y_j\to y_0\in \Sigma^\circ$. Taking this limit in \eqref{eq: non-injective} gives 
\[
	\bar\pi_L\left((0, P(\cdot,x_0; 0)\vert_{\BB} ;x_0, 0,  dt )\right) = \bar\pi_L\left((0, P(\cdot, y_0; 0)\vert_{\BB} ;y_0, 0,  dt ) \right)\,.
\]
Lemma~\ref{lem: injectivity at f = 0} then states that $x_0 = y_0$ which means that 
\begin{eqnarray}\label{eq: points are near}
d(x_0,x_j)\to 0,\ d(x_0, y_j)\to 0
\end{eqnarray}
as $j\to\infty$. Now $(0, P(\cdot,x_0; 0)\vert_{\BB} ;x_0, 0,  dt )\in U_k$ for some $k \in\{ 1,\dots, J\}$ so 
$$(f_j, P(\cdot, x_j; u_{g^\circ}^{f_j})\vert_{\BB} ;x_j, u_{g^\circ}^{f_j} (x_j), -du_{g^\circ}^{f_j}(x_j) + dt ), (f_j, P(\cdot, x_j; u_{g^\circ}^{f_j})\vert_{\BB} ;y_j, u_{g^\circ}^{f_j} (y_j), -du_{g^\circ}^{f_j}(x_j) + dt ) \in U_k$$
for $j\in \mathbb N$ sufficiently large ($1/j\ll \delta_0$). The construction of $U_k$ guarantees that $\bar\pi_L$ is injective in $U_k$. So \eqref{eq: non-injective} implies that $x_j = y_j$ which is a contradiction to our assumption that $y_j$ and $x_j$ are distinct. So there is indeed a $j_0\in \mathbb N$ such that $\bar\pi_L$ is injective when acting on $\overline{\mathcal L}_{1/j_{0}}$, and we set $\delta_1 = 1/j_0$.

We now need to verify the injectivity of $d\bar\pi_L$ over the tangent space of each point in $\overline{\mathcal L}_{\delta_1}$. Lemma~\ref{lem: injectivity at f = 0} states that this property is satisfied at points in the set 
$$\{(0, P_q(\cdot, x); x, 0, dt)\colon x\in \Sigma^\circ\}\subset \overline{\mathcal L}_{\delta_1}\,.$$  
Injectivity of $d\bar \pi_L$ is an open condition, so there is a neighborhood $U$ of $\{(0, P_q(\cdot, x); x, 0, dt)\colon x\in \Sigma^\circ\}$ in $\overline{\mathcal L}_{\delta_1}$ on which $d\bar\pi_L$ is injective.

By Proposition~\ref{prop: smooth dependence of solution} we know that $u^{f}_{g^\circ} \to 0$ in $C^{2,\gamma}(\Sigma^\circ)$ if $f\to 0$ in $C^{2,\gamma}(\partial\Sigma^\circ)$. Therefore, for some $\delta\in (0,\delta_1)$ sufficiently small, we have 
\[
    \overline{\mathcal L}_{\delta} \subset U\,,
\]
so that we have guaranteed that $\pi_L$ is an injective immersion on $\overline{\mathcal L}_{\delta}$.

Finally, that the interior of $Z_\delta(\BB)$ is embedded analytically in $B_1(0)\times \widetilde{M}$ when $g^{\circ}$ is analytic is guaranteed by Lemma~\ref{lem: min surfs are analytic} and the consequence of Proposition~\ref{prop: smooth dependence of solution} that a solution to \eqref{eq: MSE} depends analytically on its boundary value. 
\end{proof}

\section{Stability in the case of conformal deformations} \label{sec_stability}

The purpose of this section is to show that one can identify a small conformal deformation of an analytic metric on any compact subset strictly within the manifold $\widetilde{M}$ from knowledge of area data of minimal surfaces with prescribed boundary. From this section onward we will have need of multiple minimal surfaces so we replace the notation $\Sigma^\circ$ by $\Sigma_0$.

Let $g^\circ$ be a fixed, real analytic metric. Let $\Sigma_0$ be an arbitrary admissible minimal surface for $(\widetilde{M},g^\circ)$ with $\p\Sigma_0\subset \widetilde{M}\setminus M$. As a consequence of Theorem~\ref{thm: smooth dependence}, there is an open neighborhood $\mathcal{N}\subset H^{3+\frac{n+1}{2}+\gamma}(\widetilde{M})$ containing $1$ and a $\delta >0$ so that for all $\alpha \in \mathcal{N}$ the operator $F(\alpha g^\circ)(f)$ given by \eqref{eq: def of volume operator} is well-defined for all $f\in U_\delta$. In particular, let $f_1,\dots,f_{N} \in U_\delta$ be a basis of $\BB_0\subset C^{\omega}(\partial \Sigma_0)$ as constructed by Corollary~\ref{cor: bolker}, where $N=2n+2$ is its dimension. We consider the map
\begin{equation}\label{eq:boldF0}
	\mathbf{F}_0 \colon  \mathcal{N} \to C(\bar{B}_1^{N})\,,\quad \mathbf{F}_0(\alpha)(z) = F(\alpha g^\circ)(z_1f_1 + \dots + z_{N}f_{N})\,,
\end{equation}
Recall that $\bar{B}_1^{N}$ is the closed unit ball in the $\ell^1$ norm, which ensures that $z_1f_1 + \dots + z_{N}f_{N}\in U_\delta$ when $z \in \bar{B}_1^{N}$.

    Our first main result considers the case where we use areas of minimal surfaces to distinguish an analytic metric satisfying the ampleness condition from its small conformal perturbations (i.e.\ one of the conformal factors is assumed to be $1$). %
\begin{Theorem}\label{thm:rigidity}
	Assume that $g^\circ$ is analytic and that the ampleness condition in Definition~\ref{def:ampleness} holds for $M$. There is a finite set of admissible minimal surfaces $\Sigma_1,\dots,\Sigma_K$ for $(\widetilde{M},g^\circ)$ with $\p\Sigma_j \subset \widetilde{M}\setminus M$ so that on each $\p\Sigma_j$ there is an $N=(2n+2)$-dimensional subspace $\mathbf{B}_j \subset C^\omega(\partial\Sigma_j)$ spanned by elements $\{f_1,\dots, f_{N}\} \subset C^\omega(\partial\Sigma_j)$ each with $C^{2,\gamma}$-norm equal $\delta_j$, and a map $\mathbf{F}_j: \mathcal N \to C(\bar B_1^{N})$ defined as in \eqref{eq:boldF0} so that the following holds.

	There are $\mu \in (1/2,1)$ and $s \gg 1$ so that for any $L>0$ there are $C>0, \eps >0$ so that for any $g \in G$ metric on $\widetilde{M}$ conformal to $g^\circ$, $g = \alpha g^\circ$, satisfying $\alpha=1$ on $\widetilde{M}\setminus M$, $\norm{\alpha-1}_{H^{3+\frac{n+1}{2}+\gamma}(M)} \leq \eps$, and $\norm{\alpha}_{H^s(M)} \leq L$, we have
	\begin{equation}\label{eq:alphabound}
		\norm{\alpha-1}_{H^{3+\frac{n+1}{2}+\gamma}(M)} \leq C\sum_{j=1}^K\norm{\mathbf{F}_j(\alpha)-\mathbf{F}_j(1)}_{L^2(B_1^{N})}^\mu\,.
	\end{equation}
\end{Theorem}

The previous theorem will pave the way to the following more general local uniqueness and stability result. Its proof is similar but somewhat more involved than that of Theorem \ref{thm:rigidity}, and it involves an explicit calculation of the amplitude of the normal operator of a double fibration transform.

\begin{Theorem}\label{thm:stability}
	Assume that $g^\circ$ is analytic and that the ampleness condition in Definition~\ref{def:ampleness} holds for $M$. There is a finite set of admissible minimal surfaces $\Sigma_1,\dots,\Sigma_K$ for $(\widetilde{M},g^\circ)$ with $\p\Sigma_j\subset \widetilde{M}\setminus M$ so that on each $\p\Sigma_j$ there is an $N=(2n+2)$-dimensional subspace $\mathbf{B}_j \subset C^\omega(\partial\Sigma_j)$ spanned by elements $\{f_1,\dots, f_{N}\} \subset C^\omega(\partial\Sigma_j)$ each with $C^{2,\gamma}$-norm equal to $\delta_j$, and a map $\mathbf{F}_j\colon \mathcal N \to C(\bar B_1^{N})$ defined as in \eqref{eq:boldF0} so that the following holds.

	There are $\mu \in (1/2,1)$ and $s \gg 1$ so that for any $L>0$ there are $C,\eps >0$ so that for any  $g_1,g_2 \in G\cap C^{16n^2+72n+39}$ metrics on $\widetilde{M}$ conformal to $g^\circ$, $g_i = \alpha_i g^\circ, i\in\{1,2\}$, satisfying $\alpha_1=\alpha_2$ on $\widetilde{M}\setminus M$, $\norm{\alpha_i - 1}_{C^{16n^2+72n+39}(M)} \leq \eps$, %
	and $\norm{\alpha_i}_{H^s(M)} \leq L\,,i\in\{1,2\}$, we have
	\[
		\norm{\alpha_1-\alpha_2}_{H^{3+\frac{n+1}{2}+\gamma}(M)}\leq C\sum_{j=1}^K\norm{\mathbf{F}_j(\alpha_1)-\mathbf{F}_j(\alpha_2)}^\mu_{L^2(B_1^{N})}\,.
	\]
\end{Theorem}

Note that the operator in \eqref{eq:boldF0} is only well-defined when the base minimal surface $\Sigma_0$ is admissible, i.e.\ $0$ is not a Dirichlet eigenvalue of the stability operator. In Theorems~\ref{thm:rigidity} and~\ref{thm:stability}, however, the minimal surfaces coming from the ampleness assumption may not be admissible. In order for these to become admissible, we will replace them by some slightly smaller minimal surfaces. This is facilitated by 
\begin{Proposition}\label{prop:makeadmissible}
    Let $g$ be a smooth metric on $\widetilde{M}$ and let $\Sigma^\circ$ be a smooth embedded minimal surface of $(\widetilde{M},g)$ 
    such that $\partial \Sigma^\circ \subset \widetilde{M}\setminus M$. Then there exists an open subset $\Sigma' \subset\subset \Sigma^\circ$ with smooth boundary which is admissible and satisfies $\partial \Sigma' \subset \widetilde{M}\setminus M$. If $g$ and $\Sigma^\circ$ are analytic, then the boundary of $\Sigma'$ can be chosen analytic.

\end{Proposition}
The proof is found in Appendix~\ref{sec:ev}. 

\subsection{The linearization of $\mathbf{F}_0$}\label{subsec:linF0}

Fix $\Sigma_0\subset \widetilde{M}$ an admissible minimal surface for the metric $g^\circ$. %
We may assume that $(\widetilde{M},g^\circ)$ is a compact subset of some $(n+1)$-dimensional closed Riemannian manifold $(S,g^\circ)$ (see e.g.\ \cite[Thm.~3.1.8]{zbMATH07625517}). 
Let $\chi_0 \in C_c^\infty(\mR), 0\leq \chi_0\leq 1$ be compactly supported in $(-1,1)$ and identically $1$ in the set $[-1/2,1/2]$, and denote 
\begin{eqnarray}\label{eq: cut off def}
\chi(z) \coloneqq \chi_0(\sqrt{N} \abs{z})
\end{eqnarray}
where $z\in \mathbb R^{N}$. With this choice $\chi$ is supported in the Euclidean ball of radius $\frac{1}{2\sqrt{N}}$, which is strictly contained in the unit $\ell^1$-ball $\bar{B}_1^{N}$ where $z$ lives.

We will investigate the Fr\'echet derivative $D\mathbf{F}_0$ evaluated at $\alpha =1$. Here we employ terminology related to double fibration transforms as in \cite{MST}.

\begin{Proposition}\label{prop:ellonset}
	i) The operator $D\mathbf{F}_0(1)\colon H^{3+\frac{n+1}{2}+\gamma}(\widetilde{M}) \to C(\bar{B}_1^{N})$ is an analytic double fibration transform. Considered as an FIO, its canonical relation $C_0 = (N^*Z_\delta({\bf B}_0)\setminus 0)'\subset T^* \bar{B}_1^{N}\times T^* \widetilde{M}\setminus 0$ satisfies the Bolker condition at every point of $C_0$ with
    \begin{eqnarray}
        \label{eq: projection contains}
        N^\ast \Sigma_0 \subset \pi_R(C_0).
    \end{eqnarray}

    ii) For any $\tilde{\chi} \in C_c^\infty(\widetilde{M})$ with $0\leq \tilde{\chi}\leq 1$ 
	the following holds:
For every $(y_0,\xi_0) \in N^\ast \Sigma_0$, there exists an open conic neighborhood $U_0 \subset T^\ast \widetilde{M}$ of $(y_0,\xi_0)$ and a map $\pi_R^{+}\colon U_0 \to C_0$ such that $\pi_R\circ \pi_R^+ = \mathrm{id}$ on $U_0$. The conic neighborhood $U_0 \subset T^*\widetilde{M}$ can be chosen sufficiently small so that, for $\chi$ from \eqref{eq: cut off def}, $\chi \equiv 1$ on $\pi_z \circ\pi_L\circ \pi_R^+(U_0)$, in which case the operator
	\begin{equation}\label{eq:definepsido}
		\tilde{\chi}(D\mathbf{F}_0(1))^\ast \chi^2 D\mathbf{F}_0(1)\tilde{\chi} \colon C_c^\infty(S)\to C^\infty(S)
	\end{equation}
	is a $\Psi$DO of order $-n$. Its principal symbol has a non-negative representative and is 
    elliptic in $U_0\cap \pi^{-1}(\supp(\tilde \chi)^{\rm int})$.
\end{Proposition}

\begin{proof}
	We have
	\begin{equation}\label{eq:DF0geo}
		\frac{2}{n} D\mathbf{F}_0(1)\beta(z) = (\pi_{\mathcal{G}_\delta(\BB_0)})_\ast(\pi_{\widetilde{M}}^\ast \beta {\rm d Vol}_{\Sigma_0(f^z,g^\circ)}^{g^\circ})\,,
	\end{equation}
	which is an FIO that satisfies the Bolker condition on all points of its canonical relation by Corollary~\ref{cor: bolker}. According to the same result it is in fact an analytic double fibration transform.

	Using Lemma~\ref{lem: injectivity at f = 0} to calculate the canonical relation of $D\mathbf{F}_0(1)$, namely $C_0 = (N^\ast Z_\delta(\BB_0)\setminus\{0\})'$, leads to 
	\begin{equation}\label{eq:whereelliptic}
		\pi_R \left(C_0\right) = \left\{(x,t,\xi) \colon x\in{\Sigma}_0,  \exists z \in \bar{B}_1^{N}\colon t=u_{g^\circ}^{f^z}(x), \xi \in N^\ast_{(x,t)}\Sigma_0(f^z,g^\circ), u_{g^\circ}^{f^z}\ \text{solves}\ \eqref{eq: MSE} \right\}\,.
	\end{equation}
	In particular, if on the RHS of \eqref{eq:whereelliptic} $z=0$, the RHS is equal to $N^\ast \Sigma_0$ in Fermi coordinates around $\Sigma_0$. Thus, for any $((x_0,t_0),\xi_0)= (y_0,\xi_0) \in N^\ast \Sigma_0$, there is some $\zeta_0$ so that 
	\[
		(0,\zeta_0,y_0,\xi_0) \in C_0\,.
	\]
	The Bolker condition and \cite[Lem.~4.3]{zbMATH01860565} guarantee that $\pi_R \colon C_0 \to T^\ast \widetilde{M}$ is a submersion, which guarantees the existence of an open neighborhood $U_0$ of $(y_0,\xi_0)$ in $T^\ast \widetilde{M}$ so that $U_0 \subset \pi_R\left( C_0\right)$.
	Furthermore, the constant rank theorem (or Ehresmann's theorem, see \cite[Lem.~9.2]{zbMATH00205892}) guarantees the existence of a continuous map $\pi_R^+$ so that $\pi_R \circ \pi_R^+  = \mathrm{id}$ on an open neighborhood of $(y_0,\xi_0)$. In particular, due to the continuity of $\pi_R^+$, perhaps after contracting the neighborhood $U_0$ about $(y_0,\xi_0)$, for all $z\in \pi_z\circ \pi_L\circ\pi_R^+(U_0)$ we have $\abs{z} <1/(2\sqrt{N})$, which guarantees that $\chi \equiv 1$ on $\pi_z\circ \pi_L\circ\pi_R^+(U_0)$. 

 Following \cite{MR812288}, using \eqref{eq:DF0geo} and the clean intersection calculus, see e.g.\ \cite[Thm.~25.2.3]{zbMATH05528184}, we find that $\tilde{\chi}(D\mathbf{F}_0(1))^\ast \chi^2 D\mathbf{F}_0(1)\tilde{\chi}$ is a $\Psi$DO on $S$ of order $\mathrm{dim}(G_\delta(\BB_0)) - \dim(Z_\delta(\BB_0)) = N - (N+n) = -n$. Because $\chi$ is identically $1$ on $\pi_z\circ\pi_L \circ \pi_R^+(U_0)$, by direct computation of the principal symbol (see also \cite[Cor.~B.4.9]{plamenGunther}), it has a non-negative representative on $S$ and %
    this $\Psi$DO must be elliptic in the conic neighborhood $U_0 \cap \pi^{-1}(\supp(\tilde{\chi})^{\rm int})$ of $(y_0, \xi_0)$. This concludes the proof.
\end{proof}

In fact, in Proposition~\ref{prop:amplitudesketch}, we shall sketch the proof of the computation of the full amplitude of \eqref{eq:definepsido}. In either case, we now have all tools to give the 

\begin{proof}[Proof of Theorem~\ref{thm:integral_transform_WF}]
    Let $\Sigma$ be a smooth embedded minimal surface for $(\widetilde{M},g^\circ)$ with $g^\circ$ analytic and $\p\Sigma\subset \widetilde{M}\setminus M$. According to Proposition~\ref{prop:makeadmissible}, there is an admissible minimal surface agreeing with $\Sigma$ in an open neighborhood of $M$, which we shall call $\Sigma_0$ in order to apply results from this section.

Inclusion \eqref{eq: projection contains}  of Proposition~\ref{prop:ellonset} implies that 
    \begin{eqnarray}\label{eq:WF relation containment}
            \mathrm{WF}_a(\beta)\cap N^\ast \Sigma = \mathrm{WF}_a(\beta)\cap N^\ast \Sigma_0\subset \mathrm{WF}_a(\beta)\cap\pi_R(C_0).
    \end{eqnarray}

    Now let $\beta\in C(\widetilde{M})$ vanishing outside of $M$ satisfy $\mathcal{R}_\Sigma(f) = 0$ for all $f\in W_\delta$ This means that $\mathcal{R}_{\Sigma_0}(\beta)=0$, which implies that $D\mathbf{F}_0(1)\beta = 0$. Since by Proposition \ref{prop:ellonset} i), $D\mathbf{F}_0(1)$ is an analytic double fibration transform satisfying the Bolker condition, we may apply \cite[Thm.~1.2]{MST} to conclude that $\mathrm{WF}_a(\beta)\cap\pi_R(C_0) =\emptyset$. Combining this with \eqref{eq:WF relation containment} completes the proof.
\end{proof}

\subsection{Sufficiently many minimal surfaces}\label{sec:suffsurf}

We assume the ampleness condition (see Definition \ref{def:ampleness}) throughout this section.
To achieve global ellipticity, we need to patch together many admissible minimal surfaces $\{\Sigma_j\}_{j\in \mathcal J}$ to cover all microlocal directions in $S^\ast M$.

\begin{Lemma}\label{lem:findfiniteminsurfs}
    Let $M$ satisfy the ampleness condition. 
    There exists a finite collection of analytic admissible minimal surfaces $\Sigma_1,\dots,\Sigma_K$ in $(\widetilde{M},g^\circ)$ with $\p\Sigma_j\subset \widetilde{M}\setminus M$ so that the following is true. Defining $\mathbf{F}_j$ according to \eqref{eq:boldF0} with the underlying minimal surface $\Sigma_0$ replaced by $\Sigma_j$ and $\BB_0$ replaced by some finite-dimensional subspace $\BB_j \subset C^\omega(\p\Sigma_j)$, and defining $\chi$ as in \eqref{eq: cut off def}, for any $\tilde{\chi}\in C_c^\infty(\widetilde{M}), 0\leq \tilde{\chi}\leq 1$, the normal operator 
    \begin{equation}\label{eq:jthpsido}
        \tilde{\chi}D\mathbf{F}_j(1)^\ast \chi^2 D\mathbf{F}_j(1)\tilde{\chi}\colon C_c^\infty(S) \to C^\infty(S)
    \end{equation}
    is a $\Psi$DO of order $-n$. Its principal symbol has a non-negative representative, and is elliptic in $U_j\cap \pi^{-1}(\supp(\tilde{\chi})^{\rm int})$ for some $U_j\subset T^\ast S$ satisfying
    \begin{equation}
    \label{eq: covers everything}
     T^*{M} \subset\bigcup\limits_{j=1}^K U_j\,,\qquad M\Subset \bigcup\limits_{j=1}^K \pi(U_j)\,.
    \end{equation}
\end{Lemma}
\begin{proof}
    According to the ampleness condition for $M$, there is some open $W\subset \widetilde{M}$ with $M\subset W$ so that we can find a minimal surface $\Sigma_{y,\xi}'$ for each $(y,\xi)\in T^\ast W$ so that $(y,\xi)\in N^\ast \Sigma_{y,\xi}'$ and $\p\Sigma_{y,\xi}'\subset \widetilde{M}\setminus M$. According to Proposition~\ref{prop:makeadmissible} there is some admissible minimal surface $\Sigma_{y,\xi}\subset \Sigma_{y,\xi}'$ with $\p \Sigma_{y,\xi}\subset \widetilde{M}\setminus M$. By Lemma \ref{lem: min surfs are analytic}, $\Sigma_{y,\xi}$ is an analytic hypersurface.
    
    We introduce the operator $\mathbf{F}_{y,\xi}$ as in \eqref{eq:boldF0} with underlying admissible minimal surface $\Sigma_{y,\xi}$. The operator in \eqref{eq:jthpsido} with $\mathbf{F}_j$ replaced by $\mathbf{F}_{y,\xi}$ will be microlocally elliptic in some conic neighborhood $U_{y,\xi} \cap \pi^{-1}(\supp(\tilde{\chi})^{\rm int}) \subset T^\ast W$ of $(y,\xi)$ according to Proposition~\ref{prop:ellonset}. 

    Denoting by $\hat U_{y,\xi}$ the restriction of $U_{y,\xi}$ to $S^\ast W$, we see that 
    \[
        S^\ast M \Subset \bigcup_{(y,\xi)\in T^\ast W} \hat U_{y,\xi}\,.
    \]
    The proof is completed by exploiting the compactness of $S^\ast M$.
\end{proof}

From now on we let $\Sigma_1,\dots,\Sigma_K$ and $\BB_1,\dots,\BB_K$ be defined as in Lemma~\ref{lem:findfiniteminsurfs}. In particular, due to \eqref{eq: covers everything} we may define $\tilde{\chi} \in C_c^\infty(\widetilde{M}), 0\leq \tilde{\chi}\leq 1$ so that $\tilde{\chi} \equiv 1$ on some open neighborhood of $M$ and
\begin{equation}\label{eq:supptdchi}
		\supp (\tilde{\chi}) \Subset \bigcup_{j=1}^K \pi(U_j)\,,
\end{equation}
which fits the assumptions of Proposition~\ref{prop:ellonset}.

Denoting $z \coloneqq (z^1,\dots,z^K) \in \mR^{KN}$ with $z^j \in \mR^{N}, N=2n+2$, we  %
define the vector valued nonlinear area functional by
\begin{equation}\label{eq:boldF}
	\mathbf{F} \colon  C_c^\infty(\widetilde{M}) \to \bigtimes_{j=1}^K C(\mR^{N})\,,\quad \alpha\mapsto \mathbf{F}(\alpha)(z) = \begin{pmatrix} \chi({z^1})\mathbf{F}_1(\alpha)(z^1) \\ \vdots \\  \chi({z^K})\mathbf{F}_K(\alpha)(z^K)\end{pmatrix}.%
\end{equation}
The choice of $\Sigma_1,\dots,\Sigma_K, \BB_1,\dots,\BB_K$ as well as the definition of $\tilde\chi$ and $\mathbf{F}$ remain fixed in this section.
\begin{Proposition}\label{prop:injellop}
	For $\mathbf{F}$ as defined in \eqref{eq:boldF}, set %
	\[
		Q(1) \coloneqq  \begin{pmatrix} D\mathbf{F}(1)\tilde{\chi} \\ \langle D_{g^\circ}\rangle^{-n/2}(1-\tilde{\chi}) \end{pmatrix}.
	\]
	The operator
	\begin{equation}\label{eq:ellop}
		Q(1)^\ast Q(1) = \tilde{\chi}(D\mathbf{F}(1))^\ast D\mathbf{F}(1)\tilde{\chi} + (1-\tilde{\chi})\langle D_{g^\circ}\rangle^{-n} (1-\tilde{\chi})  \colon C_c^\infty(S) \to C^\infty(S)
	\end{equation}
	is an elliptic $\Psi$DO of order $-n$. Furthermore, if $(\widetilde{M}, g^\circ)$ is analytic this operator is injective on the space $\mathcal{E}'(S)$.
\end{Proposition}
\begin{proof}

	To prove that $Q(1)^\ast Q(1)$ is elliptic, we proceed as follows: the operator $\tilde{\chi}(D\mathbf{F}(1))^\ast D\mathbf{F}(1)\tilde{\chi}$ is given by
	\[
	\tilde{\chi}(D\mathbf{F}(1))^\ast D\mathbf{F}(1) \tilde{\chi}= \sum_{j=1}^K \tilde{\chi}(D\mathbf{F}_j(1))^\ast \chi^2 D\mathbf{F}_j(1) \tilde{\chi}\,,
	\]
	whose summands are all $\Psi$DOs of order $-n$ with principal symbols having a non-negative representative by Proposition~\ref{prop:ellonset}. Because each $\tilde{\chi}(D\mathbf{F}_j(1))^\ast \chi^2 D\mathbf{F}_j(1) \tilde{\chi}$ is elliptic on $U_j\cap \pi^{-1}(\supp(\tilde{\chi})^{\rm int})\subset T^\ast S$, and because there is no cancellation occurring in the principal symbols when summing up these operators, the operator $\tilde{\chi}(D\mathbf{F}(1))^\ast D\mathbf{F}(1)\tilde{\chi}$ is elliptic on $\bigcup_{j=1}^K \pi(U_j) \cap \supp(\tilde{\chi})^{\rm int} \subset S$. However, using \eqref{eq:supptdchi}, this means that $\tilde{\chi}(D\mathbf{F}(1))^\ast D\mathbf{F}(1)\tilde{\chi}$ is elliptic on $S\cap \supp(\tilde{\chi})^{\rm int}$, and in particular, $Q(1)^\ast Q(1)$ is elliptic on all of $S$. This guarantees that $Q(1)^\ast Q(1)$ is an elliptic $\Psi$DO on $S$.

	Finally, to show that the $\Psi$DO $Q(1)^\ast Q(1)$ is injective when $g^\circ$ is analytic, we proceed as follows: since elliptic $\Psi$DOs have parametrices, for any $\beta\in \mathcal{E}'(S)$, if $Q(1)^\ast Q(1)\beta = 0$ we must have $\beta \in C^\infty(S)$. We find that
	\[
		\norm{Q(1) \beta}_{L^2(\mR^{N})\times\dots\times L^2(\mR^{N})\times L^2(S)}^2 =  \langle Q(1)^\ast Q(1)\beta,\beta \rangle_{H^{n/2}(S), H^{-n/2}(S)}\,,
	\]
	so that $Q(1) \beta = 0 \iff Q(1)^\ast Q(1) \beta = 0$. Thus, the problem reduces to the injectivity of $Q(1)$ acting on $C^\infty(S)$.%

To this end, let $Q(1)\beta =0$ for some $\beta\in C^\infty(S)$. First, we find by the definition of $Q(1)$ that $(1-\tilde{\chi})\beta = 0$. Thus let us complete the proof by showing that $\tilde{\chi}\beta = 0$.
	
	We assume for contradiction that $\supp(\tilde{\chi}\beta )\neq \emptyset$. According to \cite[Prop.~8.5.8]{hoermander1}, the exterior normal set $N_e$ defined in \cite[Def.~8.5.7]{hoermander1} satisfies $\pi \left( N_e(\supp(\tilde{\chi}\beta))\right)  \neq \emptyset$, and thus we can find $(\hat y,\hat \xi)\in N_e(\supp(\tilde{\chi}\beta))$ with $\hat y\in \bigcup_{j=1}^K \pi\left( U_j\right)$. By microlocal analytic continuation, \cite[Thm.~8.5.6']{hoermander1}, we conclude that $(\hat y,\pm\hat\xi) \in \mathrm{WF}_a(\tilde{\chi}\beta)$.

	By \eqref{eq: covers everything}, there is some $j \in \{1,\dots, K\}$ so that $(\hat y,\hat \xi) \in U_j$, so that for some $(\hat z,\hat \zeta)\in T^* \bar{B}_1^{N}$ with $\abs{\hat z}<1/(2\sqrt{N})$ we have $(\hat z,\hat \zeta, \hat y,\hat \xi) \in C_j$, with $C_j$ the canonical relation of $D\mathbf{F}_j(1)$. By assumption of $Q(1)\beta =0$, we have that $\chi D\mathbf{F}_j(1)\tilde{\chi} \beta=0$, and because $\chi$ is identically $1$ in a neighborhood of $\hat z$, we find $(\hat z,\hat \zeta,\hat y,\hat\xi)\not\in \mathrm{WF}_a(D\mathbf{F}_j(1)\tilde{\chi}\beta)$. Recalling that $g^\circ$ is analytic, by Proposition~\ref{prop:ellonset}, $D\mathbf{F}_j(1)$ is an analytic double fibration transform and it satisfies the Bolker condition at $(\hat z,\hat \zeta,\hat y,\hat\xi)$. Thus, by \cite[Thm.~1.2]{MST}, we know that $(\hat y,\hat \xi) \not\in\mathrm{WF}_a(\tilde{\chi}\beta)$, which is a contradiction. We therefore conclude that $\supp(\tilde{\chi}\beta)$ is empty, which finishes the proof. 
\end{proof}

Recall that $\tilde{\chi}$ was chosen so that $\tilde{\chi}\equiv 1$ on an open neighborhood of $M$. We are now in position to prove
\begin{Proposition}\label{prop:lowerbd}
	Assume that $(\widetilde{M},g^\circ)$ is analytic and the ampleness condition in Definition ~\ref{def:ampleness} holds. %
	There is a constant $C>0$ so that for all $\beta \in L^2(M)$,%
	\[
		\norm{D\mathbf{F}(1)\beta}_{H^{\frac{n}{2}}(\mR^{N})\times\dots\times H^{\frac{n}{2}}(\mR^{N})} \geq C\norm{\beta}_{L^2(M)}\,.
	\]
\end{Proposition}
\begin{proof}
	By Proposition~\ref{prop:injellop}, $Q(1)^\ast Q(1)$ is an elliptic $\Psi$DO of order $-n$ on the closed manifold $S$, so that it is Fredholm and thus has closed range as an operator from $L^2(S)$ to $H^{n}(S)$, see \cite[Thm.~8.1]{MR1852334}. Thus, $Q(1)^\ast Q(1) \colon L^2(S) \to H^{n}(S)$ is a closed, self-adjoint, injective operator and thus invertible. We denote this inverse by $T\colon H^{n}(S)\to L^2(S)$. 

	Let $\beta \in L^2(M)$ be arbitrary, and write $\beta$ also for its extension by $0$ to $L^2(S)$. Using the observation that $\tilde{\chi} \equiv 1$ on an open neighborhood of $M$ and thus
	\[
		Q(1)^\ast Q(1)\beta =  \tilde{\chi}(D\mathbf{F}(1))^\ast D\mathbf{F}(1)\tilde{\chi}\beta =\tilde{\chi}(D\mathbf{F}(1))^\ast D\mathbf{F}(1)\beta\,,
	\]
	we calculate that
	\[
		\norm{\beta}_{L^2(M)} = \norm{\beta}_{L^2(S)} = \norm{T \tilde{\chi}(D\mathbf{F}(1))^\ast D\mathbf{F}(1)\beta}_{L^2(S)} \leq C\norm{\tilde{\chi}(D\mathbf{F}(1))^\ast D\mathbf{F}(1)\beta}_{H^{n}(S)}\,.
	\]
	The proof would thus be complete by showing that there is some $C<\infty$ so that
	\begin{equation}\label{eq:onebd}
	\norm{\tilde{\chi}(D\mathbf{F}(1))^\ast}_{H^{\frac{n}{2}}(\mR^{N})\times\dots\times H^{\frac{n}{2}}(\mR^{N})\to H^n(S)} \leq C\,,
	\end{equation}
	which we prove now. Notice that for any $u\in H^{-n}(S)$, 
	\[
		\norm{\chi D\mathbf{F}_j(1) \tilde{\chi} u}_{H^{-n/2}(\mR^{N})}^2 = \langle \langle D_{g^\circ}\rangle^{-n}u, \underbrace{\langle D_{g^\circ}\rangle^{n} \tilde{\chi}(D\mathbf{F}_j(1))^\ast \chi \langle D\rangle^{-n} \chi D\mathbf{F}_j(1) \tilde{\chi}\langle D_{g^\circ}\rangle^n}_{\in \Psi^0} \langle D_{g^\circ}\rangle^{-n} u\rangle_{L^2(S)} %
	\]
	The operator on the right side of the inner product applied to $\langle D_{g^\circ}\rangle^{-n} u$ is a $\Psi$DO of order $0$ on $S$ by the clean intersection calculus, as in the proof of Proposition \ref{prop:ellonset}. Thus the Cauchy-Schwartz inequality gives, for some $C>0$, that 
	\[
		\norm{\chi D\mathbf{F}_j(1) \tilde{\chi} u}_{H^{-n/2}(\mR^{N})}^2\leq C\norm{u}_{H^{-n}(S)}^2\,.
	\]
	We conclude that %
	$D\mathbf{F}(1)\tilde{\chi} \colon H^{-n}(S) \to H^{-\frac{n}{2}}(\mR^{N})\times\dots\times H^{-\frac{n}{2}}(\mR^{N})$ is bounded, which directly implies \eqref{eq:onebd}.
\end{proof}	

\subsection{Proof of Theorem~\ref{thm:rigidity}}

\begin{proof}[Proof of Theorem~\ref{thm:rigidity}]
	Following \cite[Thm.~2]{stefanov2009linearizing}, we define the Banach spaces
	\[
	\mathcal{X}_1 \coloneqq H^{3+\frac{n+1}{2}+\gamma}(M)\,,\quad \mathcal{X}_1' \coloneqq L^2(M)\,,
	\]
	for $\gamma>0$ from \eqref{eq: def G}, as well as
	\[
	\mathcal{X}_2 \coloneqq L^2(\mR^{N})\times\dots\times L^2(\mR^{N}) = (L^2(\mR^{N}))^K\,,\quad \mathcal{X}_2' \coloneqq H^{\frac{n}{2}}(\mR^{N})\times\dots\times H^{\frac{n}{2}}(\mR^{N}) = (H^{\frac{n}{2}}(\mR^{N}))^K\,.
	\]
	We see that $\mathcal{X}_1 \subset \mathcal{X}_1'$ and $\mathcal{X}_2'\subset \mathcal{X}_2$. 
	As a consequence of \eqref{eq:tayloronF} in Theorem~\ref{thm: smooth dependence}, we have that for all $\tilde\alpha \in {\mathcal X}_1$,
	\[\mathbf{F}(\tilde{\alpha}+1) - \mathbf{F}(1) = D\mathbf{F}(1)\tilde{\alpha} +R_1(\tilde{\alpha}g^\circ) %
	\]
	with
	\[
		\norm{R_1(\tilde{\alpha}g^\circ)}_{\mathcal{X}_2} \leq C\norm{R_1(\tilde{\alpha}g^\circ)}_{C(\bar{B}_1^{N})\times\dots\times C(\bar{B}_1^{N})} \leq C'\norm{\tilde{\alpha}}_{\mathcal{X}_1}^2\,.
	\]
	This is because there are cut-offs $\chi$ present on the output side of $R_1$ as they are present in the definition of $\mathbf{F}$. Furthermore, according to Proposition~\ref{prop:lowerbd}, for all $\beta \in \mathcal{X}_1\subset L^2(M)$ we have
	\[
		\norm{D\mathbf{F}(1)\beta}_{\mathcal{X}_2'} \geq C\norm{\beta}_{\mathcal{X}_1'}\,.
	\]

    We now set ${\mathcal X}_1'' \coloneqq  H^{s_1}(M)$ and ${\mathcal X}''_2  \coloneqq H^{s_2}(\mR^{N})\times\dots\times H^{s_2}(\mR^{N})$ for $s_1, s_2\gg 1$. By \cite{MR0500580}, we have $\|\beta\|_{{\mathcal X}_1} \lesssim \|\beta\|_{{\mathcal X}'_1}^{\mu_1} \|\beta\|_{{\mathcal X}_1''}^{1-\mu_1}$ and $\|u\|_{{\mathcal X}'_2} \lesssim \|u\|_{{\mathcal X}_2}^{\mu_2} \|u\|_{{\mathcal X}_2''}^{1-\mu_2}$ with $\mu_1, \mu_2 \in (0,1]$ and $\mu_1\mu_2>1/2$. Furthermore, if $s_1\gg s_2$, Proposition~\ref{prop:ellonset} shows that that $D{\bf F}(1)\colon {\mathcal X}''_1 \to {\mathcal X}''_2$ is bounded by standard norm estimates for FIOs (or recall that $D\mathbf{F}(1)^\ast D\mathbf{F}(1)$ is a $\Psi$DO). %
We are thus in position to apply \cite[Thm.~2]{stefanov2009linearizing}
to conclude the following: for any $L>0$ there are $C,C',\eps >0$ so that if $\norm{\tilde{\alpha}}_{{\mathcal X}''_1} \leq L$ and $\norm{\tilde{\alpha}}_{\mathcal{X}_1}<\eps$, we have
	\[
		\norm{\tilde{\alpha}}_{\mathcal{X}_1} \leq  CL^{2-\mu_1-\mu_2}\sum_{j=1}^K\norm{\chi\mathbf{F}_j(\tilde{\alpha}+1)-\chi\mathbf{F}_j(1)}_{L^2(\mR^{N})}^{\mu_1\mu_2} \leq C' L^{2-\mu_1-\mu_2}\sum_{j=1}^K\norm{\mathbf{F}_j(\tilde{\alpha}+1)-\mathbf{F}_j(1)}_{L^2(B_1^{N})}^{\mu_1\mu_2}\,.
	\]
	As a consequence of this estimate, for any $\alpha \in H^{3+\frac{n+1}{2}+\gamma}(\widetilde{M})$ with $\alpha = 1$ on $\widetilde{M} \setminus M$, we can apply the above estimate to $\tilde{\alpha}=\alpha-1$ which gives the estimate \eqref{eq:alphabound}.%
\end{proof}

\subsection{Injectivity and stability for the minimal surface transform}\label{sec:foliation_inj}

Recall the foliation condition from Definition \ref{def:fol_intro}.

\begin{Theorem}\label{thm:weakfol}
    	Let $(\widetilde{M},g^\circ)$ be analytic and $\beta \in C(\widetilde{M})$ with $\supp(\beta)\subset M$. Assume that $M$ satisfies either the foliation condition in Definition~\ref{def:fol_intro} or the  ampleness condition in Definition~\ref{def:ampleness}.
    
    If for all smooth embedded minimal surfaces $\Sigma$ in $(\widetilde{M},g^\circ)$ with $\p\Sigma\subset \widetilde{M}\setminus M$, we have
	\begin{equation}\label{eq:full02}
		(\mathcal{R}\beta)(\Sigma)=0\,,%
	\end{equation}
	then $\beta = 0$.

    Furthermore, in case the ampleness condition holds, there is a finite set of admissible minimal surfaces $\Sigma_1,\dots,\Sigma_K$ for $(\widetilde{M},g^\circ)$ with $\p\Sigma_j\subset \widetilde{M}\setminus M$ so that for some $C>0$ independent of $\beta$,
    \begin{equation}\label{eq:normboundlinearized2}
        \norm{\beta}_{L^2(M)} \leq C\sum_{j=1}^K \left\lVert (\mathcal{R}\beta)(\Sigma_j(f^z,g^\circ))\right\rVert_{H^{\frac{n}{2}}(B_{1}^{N})}\,,
    \end{equation}
    where $B_{1}^{N}$ is the $\ell^1$-unit ball in $\mR^{N}$ and $N=2n+2$%
    , and $\Sigma(f, g^\circ)$ is given by \eqref{eq: min surf as graph}. The expression in the norm brackets on the RHS is a function of the variable $z\in B_1^{N}$, of which the norm is taken.
\end{Theorem}
\begin{proof}
In the case of the ampleness condition, this result and \eqref{eq:normboundlinearized2} are a consequence of Proposition~\ref{prop:lowerbd}. We thus give the proof in the case of the foliation condition, for which we let $\rho$ be the function guaranteed to us by Definition~\ref{def:fol_intro}. 

	We begin a layer stripping procedure, for which we introduce the set 
	\[
		J \coloneqq \{s\in [0,m_\rho)\colon \beta = 0\ \text{on}\ \rho^{-1}([0,s])\}\,.%
	\]
	which is a closed set by the continuity of $\beta$ and $\rho$. We wish to show that $J$ is open, which will allow us to conclude by connectedness that $J=[0,m_\rho)$. By conditions (1) and (3) from the foliation, choosing $s_0 = 0, U =\emptyset$ in (3), there is $\delta>0$ so that $\beta$ vanishes on $\Omega_s$ for $s\in [0,\delta)$. So let $s_0\in J$, and we may assume that $s_0>0$. 

	Due to (2) in Definition~\ref{def:fol_intro}, for arbitrary $x_0 \in \Omega_{s_0}$, there is some minimal surface $\Sigma_0$ so that $(x_0,\xi_0)\coloneqq (x_0,d\rho(x_0))\in N^\ast \Sigma_0$ and $\p\Sigma_0\subset \widetilde{M}\setminus M$. Using Proposition~\ref{prop:makeadmissible}, we find some admissible minimal surface $\Sigma_0'\subset \Sigma_0$ so that $\p\Sigma_0'\subset \widetilde{M}\setminus M$.
    
    We may define the operator $\mathbf{F}_0$ according to \eqref{eq:boldF0} using this minimal surface $\Sigma_0'$, and we let $C_0$ denote the canonical relation of the FIO $D\mathbf{F}_0(1)$, which we know to be an analytic double fibration transform (see Proposition~\ref{prop:ellonset}). By the same proposition, we know that $(0,\zeta_0;x_0,\xi_0)\in C_0$ for some $\zeta_0$ and $C_0$ satisfies the Bolker condition there. Let $\BB_0 \subset C^\omega(\partial\Sigma_0)$ be the set constructed in Corollary~\ref{cor: bolker} with $2n+2= N = \mathrm{dim}(\BB_0)$. %

    By \eqref{eq:full02} and the fact that $\beta$ vanishes outside of $M$, we thus, in fact, have that for all $z\in B_1^{N}(0)$,
	\begin{align*}
		D\mathbf{F}_0(1)\beta(z) &= DF(g^\circ)\beta g^\circ(f^z) = \frac{n}{2}\int_{\Sigma'_0(f^z,g^\circ)}\beta {\rm dVol}^{g^\circ}_{\Sigma'_0(f^z,g^\circ)} = \frac{n}{2}(\mathcal{R}\beta)(\Sigma_0'(f^z,g^\circ)) = 0\,.%
    	\end{align*}
    An application of \cite[Thm.~1.2]{MST} thus shows that $(x_0,\xi_0)\not\in \WF_a(\beta)$. 

	By the fact that $s_0\in J$, we have $\beta = 0$ on $\Omega_s$ for $s\in [0,s_0]$. Implicitly, part of the assumption of (2) in the definition of the foliation is that $d\rho(x)\neq 0$ for all $x\in M$. Assume for contradiction that $x_0\in \supp(\beta)$. Using the $C^1$-function $-\rho$ and the fact that $\supp(\beta)\subset \{\rho \geq s_0\}$, by verifying the conditions in \cite[Prop.~8.5.8]{hoermander1} one finds that $(x_0,-\xi_0)\in \ol{N_e(\supp(\beta))}$, where $N_e$ is the exterior normal set defined in \cite[Def.~8.5.7]{hoermander1}. In particular, \cite[Thm.~8.5.6']{hoermander1} implies that $(x_0,\xi_0) \in \WF_a(\beta)$, a contradiction. Thus, there is an open neighborhood of $x_0$ on which $\beta$ vanishes. 

	Since $x_0\in \Omega_{s_0}$ was arbitrary, we conclude that there is an open neighborhood of $\Omega_{s_0}$ in which $\beta$ vanishes, which by (3) means that there is $\delta>0$ so that $[0,s_0+\delta) \cap [0,m_\rho)\subset J$. 

	We have concluded that $J$ is non-empty, open and closed and connected in $[0,m_\rho)$ and thus $J=[0,m_\rho)$. This implies that $\beta$ vanishes on $\bigcup_{s\in [0,m_\rho)}\Omega_s$, and because $\beta$ is continuous it must vanish on the closure of that set, which according to (1) is all of $M$. Because $\supp(\beta) \subset M$, we deduce that $\beta=0$.
\end{proof}

\subsection{Stability of the Bolker condition under perturbations of the metric}

We now turn to proving Theorem~\ref{thm:stability}, which will require us to investigate the Fr\'echet derivative $D\mathbf{F}(\alpha)$ of $\mathbf{F}$ for any $\alpha$ near $1$. Let $\Sigma_0$ as before be a fixed admissible minimal surface for some fixed metric $g^\circ$ (the stability operator of which does not have Dirichlet eigenvalue $0$), and recall the definition of $\mathbf{F}_0 \colon  \mathcal{N} \to C(\bar{B}_1^{N})$ from \eqref{eq:boldF0}. The goal is to show that the operator $D\mathbf{F}_0(\alpha)^\ast \chi^2 D\mathbf{F}_0(\alpha)$ remains a $\Psi$DO for $\alpha\approx 1$ in the appropriate norm, and to do this we will prove that the Bolker condition holds when $\alpha$ is close to $1$. There is a slightly subtle issue here: in the definition of $\mathbf{F}_0(\alpha)$ we are using a fixed parametrization of minimal surfaces and a fixed set of basis elements $f_1, \ldots, f_{N}$ (i.e.\ we are not allowed to change the parametrization or basis elements when $\alpha$ changes), and we need to show that the Bolker condition for $D\mathbf{F}_0(\alpha)$ holds with respect to this fixed parametrization and these basis elements as $\alpha$ varies among small perturbations of $1$.

More precisely, given a smooth, admissible minimal surface $\Sigma_0 \subset \widetilde{M}$ for the metric $g^\circ$ whose boundary $\partial\Sigma_0\subset \widetilde{M}\setminus M$ is non-empty, Proposition~\ref{prop:diffF} gives an explicit formula for $\left(DF(\alpha g^\circ)\beta\right)(f)$ when $f\in U_\delta\subset C^{2,\gamma}(\partial \Sigma_0)$. When $\alpha =1$, Corollary \ref{cor: bolker} gives basis elements $f_1,\dots, f_{N}\in U_\delta$ such that the operator 
$$ D\mathbf{F}_0(1):  H^{3+\frac{n+1}{2}+\gamma}(\widetilde{M}) \to C(\bar{B}_1^{N})$$
is an FIO satisfying the Bolker condition. In what follows, for these fixed $f_1,\dots,f_{N}$ and any $z\in \bar{B}_1^{N}$, we write
$$f^z := z_1 f_1 + \dots+ z_{N} f_{N},$$
which satisfies $f^z \in U_\delta$.

Throughout this section, when we say $\alpha \in H^{3+\frac{n+1}{2}+\gamma}(\widetilde{M})$ is near (the constant function) $1$, we will always mean in the $H^{3+\frac{n+1}{2}+\gamma}(\widetilde{M})$-topology. Following \eqref{eq: manifold Z}, for $\alpha \in H^{3+\frac{n+1}{2}+\gamma}(\widetilde{M})$ sufficiently near $1$, we define
\begin{eqnarray}\label{def: Z alpha}
	Z(\alpha) \coloneqq \{(z,x,t)\in \bar{B}_1^{N}\times\widetilde{M}\colon t=u^{f^z}_{\alpha g^\circ}(x)\}\,,
\end{eqnarray}
where we as usual have denoted by $u^{f}_{\alpha g^\circ} \in C^{2,\gamma}(\Sigma_0)$ the function with boundary value $f$ so that in $g^\circ$-Fermi coordinates around $\Sigma_0$, $\Sigma_0(f,\alpha g^\circ)=\{(x,u^{f}_{\alpha g^\circ}(x)); x\in \Sigma_0\}$ is a minimal surface for the metric $\alpha g^\circ$ in $\widetilde{M}$. (This $u^{f^z}_{\alpha g^\circ}$ exists due to Lemma \ref{lem: smooth dependence of sol}.) 
Throughout this section we shall also use the notation 
\[
	\pi_L^\alpha \colon N^\ast Z(\alpha)\to T^\ast \bar{B}_1^{N}\,,\quad (z,\zeta;y,\xi)\mapsto (z,\zeta)\,,
\]
for the left-projection on $N^\ast Z(\alpha)$.%
\begin{Lemma}\label{lem:canonical converge}
For any $\alpha \in H^{3+\frac{n+1}{2}+\gamma}(\widetilde{M})$ sufficiently near $1$ %
there exists an open conic neighborhood $U\subset T^*\bar{B}_1^{N} \times T^*\widetilde{M}$ containing $N^*Z(1)$ and a map $\Psi_\alpha \colon U \to T^*\bar{B}_1^{N} \times T^*\widetilde{M}$ homogeneous of degree $1$ in the fiber which is a diffeomorphism from $N^*Z(1)$ to $N^*Z(\alpha)$ when restricted to $N^*Z(1)$. 
As $\alpha \to 1$ in $H^{3+\frac{n+1}{2}+\gamma}(\widetilde{M})$ the map $\pi_L^\alpha\circ\Psi_\alpha\vert_{N^*Z(1)}$ converges in $C^1(K; T^*\bar B_1^{N})$ to $\pi_L^1 \colon N^*Z(1) \to T^\ast \bar{B}_1^{N}$ on every compact subset $K \subset N^*Z(1)$.%
\end{Lemma}
\begin{proof}
For each $\alpha$ near $1$ in $H^{3+\frac{n+1}{2}+\gamma}(\widetilde{M})$, Lemma \ref{lem: smooth dependence of sol} shows that $Z(\alpha)\subset \bar{B}_{1}^{N}(0) \times \widetilde{M}$ is a $C^2$-hypersurface. In addition, at a point $(z;x, u_{\alpha g^\circ}^{f^z})\in Z(\alpha)$, the unique conormal direction is given by $\rho\left(-dt + d\phi_\alpha(z,x)\right)$, where $\rho>0$ and $\phi_\alpha(z,x) \coloneqq u_{\alpha g^\circ}^{f^z}(x)$ is $C^2$ with respect to $(z,x)\in \bar{B}_1^{N}\times \widetilde{M}$. %

Let $\mathcal O \subset  \widetilde{M}$ be a $g^\circ$-Fermi neighborhood of $\Sigma_0$ and consider the diffeomorphism 
\[
	\psi_\alpha\colon B^{N}_1\times \mathcal O\to  B^{N}_1 \times \widetilde{M}\,, \quad (z,x,t) \mapsto (z, x, t+\phi_{\alpha}(z,x) - \phi_{1}(z,x))\,,
\] 
for $(x,t) \in \widetilde{M}$ in $g^\circ$-Fermi coordinates around $\Sigma_0$. By Lemma \ref{lem: smooth dependence of sol} the restriction of $\psi_\alpha$ to $Z(1) = \{(z,x,t)\mid t=\phi_1(z,x)\}$ is a $C^2$ diffeomorphism from $Z(1)$ to $Z(\alpha)$.

Denote by $\Psi_\alpha$ the lift of $\psi_\alpha$ to a $C^1$ diffeomorphism on the cotangent bundle $U:=T^*\bar{B}^{N}_1 \times T^*\mathcal O \to T^*\bar{B}^{N}_1 \times T^*\widetilde{M}$ defined by
$$\Psi_\alpha \colon (z,y, \zeta,\xi) \mapsto \left(\psi_\alpha(z,y), \left(\psi_\alpha^{-1}\right)^*\vert_{\psi_\alpha(z,y)}(\zeta,\xi)\right)$$
that is homogeneous of degree $1$ in the fiber $(\zeta,\xi)$. The restriction of $\Psi_\alpha$ to $N^*Z(1)$ is a diffeomorphism onto $N^*Z(\alpha)$ as desired. %

As a consequence of Lemma \ref{lem: smooth dependence of sol} and the differentiability of $\phi_\alpha$ discussed above, we know that $\phi_\alpha\to \phi_1$ in $C^2$ as $\alpha\to 1$ in $H^{3+\frac{n+1}{2}+\gamma}(\widetilde{M})$. This fact and an explicit computation of $\pi_L^\alpha \circ \Psi_\alpha$ complete the proof.
\end{proof}
The following general fact related to stability of embeddings \cite[Prop.~4.33]{MR4328926} will be useful later:
\begin{Lemma}
    \label{lem: general openness of inj imm}
Let $X, Y$ be smooth manifolds with $X$ compact. Suppose $f:  X \to Y$ is an injective immersion and $f_j \in C^1(X ; Y)$ is a sequence of functions converging to $f$ in $C^1( X; Y)$. Then $f_j$ is an injective immersion for $j\in \mathbb N$ sufficiently large.
\end{Lemma}
\begin{proof}
The fact that $f_j$'s are immersions is a direct consequence of $C^1$ convergence. To see injectivity, suppose by contradiction that there is a sequence of points $x^1_j\neq x_j^2$ in $X$ such that $f_j(x_j^1) = f_j(x_j^2)$. By compactness, after passing to a subsequence we can take the limit to arrive at $x_j^1\to x^1$, $x_j^2\to x^2$ with $f(x^1) = f(x^2)$ for $x_1, x_2 \in X$. By the injectivity assumption, $x^1 = x^2 =: \hat x$. By the fact that $f$ is an injective immersion, there is a coordinate neighbourhood $U\subset X$ containing $\hat x$, $x_j^1$, and $x_j^2$ such that $f\mid_U: U \to Y$ is an embedding onto its image.

We can now estimate, in coordinates, that
\begin{align*}
0=\|f_j(x^1_j) - f_j(x^2_j)\|& \geq \| f(x^1_j) - f(x^2_j)\| - \|(f_j(x_j^1) - f(x_j^1)) - (f_j(x_j^2) - f(x_j^2) \|\\ &\geq \| f(x^1_j) - f(x^2_j)\|- \|df - df_j\|_\infty \|x_j^1 - x_j^2\|.
\end{align*}
By the fact that $f\mid_U$ is an embedding onto the image,
\begin{eqnarray*}
0=\|f_j(x^1_j) - f_j(x^2_j)\|\geq c\|x_j^1 - x_j^2\|- \|df - df_j\|_\infty \|x_j^1 - x_j^2\|.
\end{eqnarray*}
Convergence in $C^1$ means that $\|df - df_j\|_\infty\leq c/2$ for $j\in \mathbb N$ sufficiently large. This means that $0\geq\| x_j^1-x_j^2\|$, contradicting our previous assertion that $x_j^1\neq x_j^2$. So the proof is complete.
\end{proof}

Below,  ${\rm dVol}^{\alpha g^\circ}$ denotes the volume form induced by ${\rm dVol}^{\alpha g^\circ}$  to $\Sigma_0(f^z,\alpha g^\circ)$.

\begin{Lemma}\label{lem:bolkerstable}
	For $\alpha \in H^{3+\frac{n+1}{2}+\gamma}(\widetilde{M})$ sufficiently close to $1$ (in the $H^{3+\frac{n+1}{2}+\gamma}(\widetilde{M})$ topology), the Fr\'echet derivative evaluated at $\alpha$, 
    	$$D\mathbf{F}_0(\alpha): C_c^\infty(\widetilde{M}) \to C(\bar{B}_1^{N}),$$
where $\left(D\mathbf{F}_0(\alpha)\beta\right)\in C(\bar{B}_1^{N})$ is given by %
\begin{eqnarray}\label{eq: DFalpha}
(D\mathbf{F}_0(\alpha)\beta)(z) = \frac{n}{2} \int_{\Sigma_0(f^z,\alpha g^\circ)} \beta \s {\rm dVol}^{\alpha g^\circ}.
\end{eqnarray}
This is a double fibration transform satisfying the Bolker condition at every point in its canonical relation.
\end{Lemma}
\begin{proof}
According to \eqref{eq: minimal surface transform}, we have for $\beta \in C_c^\infty(\widetilde{M})$ and $\alpha$ near $1$ in $H^{3+\frac{n+1}{2}+\gamma}(\widetilde{M})$,
	\[
		(D\mathbf{F}_0(\alpha)\beta)(z) = \frac{n}{2} \int_{\Sigma_0(f^z,\alpha g^\circ)} \beta\s  {\rm dVol}^{\alpha g^\circ}\,,
	\]
which gives identity \eqref{eq: DFalpha}. This operator is a double fibration transform with canonical relation %
	$N^\ast Z(\alpha)\setminus 0$ due to \cite[Thm.~2.2]{MST}.%

	To verify the Bolker condition for $N^\ast Z(\alpha)\setminus 0$, we need to show that the left projection 
	$$\pi^\alpha_L \colon N^*Z(\alpha)\setminus 0 \to T^*\bar{B}_1^{N}$$
	is an injective immersion. Denoting by $\Psi_\alpha$ the diffeomorphism constructed in Lemma \ref{lem:canonical converge} (restricted to $N^*Z(1)$), we instead show that for $\alpha\in H^{3+\frac{n+1}{2}+\gamma}(\widetilde{M})$ sufficiently near $1$, $\pi^\alpha_L \circ \Psi_\alpha \colon N^*Z(1)/\mR_+ \to T^*\bar{B}_1^{N}$ is an injective immersion, since $\Psi_\alpha$ being a diffeomorphism that is homogeneous of degree $1$ in the fiber will then lead to the desired conclusion.

	To this end, we identify 
    $$N^*Z(1)/ \mathbb R_+ = \{(z,\zeta; x,t, \xi)\in N^*Z(1) \mid \|(\zeta,\xi)\|= 1\}$$
where $\|\cdot\|$ denotes the norm in any auxiliary metric. The map $\pi_L^1= \pi_L\colon N^*Z(1)/\mR_+ \to T^*\bar{B}_1^{N}$ is an injective immersion by Corollary \ref{cor: bolker}, and by Lemma~\ref{lem:canonical converge}, $\pi_L^\alpha\circ \Psi_\alpha \mid_{N^*Z(1)/ \mathbb R_+} \to \pi_L$ in $C^1$ as $\alpha\to 1$ in $H^{3+\frac{n+1}{2}+\gamma}(\widetilde{M})$. We now evoke the general principle of Lemma \ref{lem: general openness of inj imm} to obtain that $\pi_L^\alpha\circ \Psi_\alpha \mid_{N^*Z(1)/ \mathbb R_+}$ is an injective immersion for $\alpha$ sufficiently close to $1$, and the proof is complete upon applying Lemma~\ref{lem: conic no matter}.
\end{proof}

\subsection{Amplitude of the normal operator of a double fibration transform}\label{subsec:amplitude}

Recall the setting of Section~\ref{subsec:linF0}; in particular, $\Sigma_0$ is a fixed admissible minimal surface for the smooth Riemannian manifold $(\widetilde{M},g^\circ)$, and let $\chi\in C^\infty_c(B_1^{N})$ as in \eqref{eq: cut off def} where $N=2n+2$ is the dimension of the space $\BB_0$ constructed in Corollary~\ref{cor: bolker}. 

By Lemma~\ref{lem:bolkerstable}, for $\alpha\in H^{3+\frac{n+1}{2}+\gamma}(\widetilde{M})$ sufficiently close to $1$, the FIO $D\mathbf{F}(\alpha)$ satisfies the Bolker condition. Now for any $\tilde{\chi}\in C_c^\infty(\widetilde{M}), 0\leq \tilde{\chi}\leq 1$, arguing as in the proof of Proposition~\ref{prop:ellonset}, applying \cite[Thm.~1]{MR812288}, we find that, as an operator on $C^\infty(S)$ for a closed compact manifold $S$ containing $\widetilde{M}$,
\begin{eqnarray}\label{eq: order of Q}
\tilde{\chi}(D\mathbf{F}_0(\alpha))^\ast \chi^2 D\mathbf{F}_0(\alpha)\tilde{\chi}\in \Psi^{-n}(S).%
\end{eqnarray}
In other words, the operator on the LHS is a pseudo-differential operator of order $\mathrm{dim} (G_\delta(\BB_0)) - \dim (Z(\alpha)) = N - (N+n) = -n$ where $Z(\alpha)$ is given by \eqref{def: Z alpha}.

We mention that in \cite{zbMATH03738575}, it is calculated how the principal symbol of the normal operator transforms under the condition that the output and input dimensions are the same. We are not in that case and also need access to the full amplitude rather than just the principal symbol. We point the reader also to \cite[Lem.~4]{zbMATH06756133}, \cite{zbMATH02199899}, \cite[Lem.~4.2]{zbMATH05206066}, \cite{zbMATH02103763}, where results of the kind we aim to attain are proved by explicit calculation.%

\begin{Proposition}\label{prop:amplitudesketch}
For any $\alpha \in H^{3+\frac{n+1}{2}+\gamma}(\widetilde{M})$ sufficiently near $1$, the operator 
	\[
		P(\alpha) \coloneqq \tilde{\chi}D\mathbf{F}_0(\alpha)^\ast \chi^2 D\mathbf{F}_0(\alpha)\tilde{\chi} \in \Psi^{-n}(S)%
	\]
is a pseudodifferential operator of order $-n$ on $S$. 

Furthermore, for all $\alpha\in C^{\ell}(\widetilde{M})$ with $\norm{\alpha-1}_{C^{\ell}(\widetilde{M})}$ small enough, $P(\alpha)$, as an operator $L^2(S)\to H^n(S)$, depends continuously on $\alpha$ in the $C^{\ell}(\widetilde{M})$-norm. Here $\ell = \ell_n = 16n^2+72n+39$.
\end{Proposition}
\begin{proof}
The fact that $P(\alpha)$ belongs to $\Psi^{-n}(S)$ was already shown in \eqref{eq: order of Q}. We thus only need to show the continuous dependence on the conformal factor $\alpha \in C^{\ell}(\widetilde M)$. To prove this, we will follow basic references related to Fourier integral operators and clean composition calculus (e.g.\ \cite{zbMATH03481135, zbMATH05528184, zbMATH05817029}).

In this proof we will explicitly calculate the full amplitude of the $\Psi$DO $P(\alpha)$ and will apply Proposition~\ref{prop:hbounded} to find the desired continuity statement, which requires that we have control of $4n+1$ derivatives of the amplitude. Working backwards from this number throughout this proof will lead to requiring control of $\ell = \ell_n = 16n^2+72n+39$ derivatives of $\alpha$.
	
Using the fact from Lemma~\ref{lem:bolkerstable} that $D\mathbf{F}_0(\alpha)$ given by \eqref{eq: DFalpha} is a double fibration transform satisfying the Bolker condition, from the proof of \cite[Thm.~2.2]{MST}, in local coordinates, the kernel of $D\mathbf{F}_0(\alpha)$ is given by
	\begin{equation}\label{eq:myFIO}
		D\mathbf{F}_0(\alpha)(z,x,t)= \int_{\mR} e^{i(\phi_\alpha(z,x)-t)\cdot\eta} a_\alpha(z,x)\dd \eta\,.
	\end{equation}
It is an easy exercise to show that if we use $g^\circ$-Fermi coordinates around $\Sigma^\circ$, we may choose $\phi_\alpha(z,x) = u_{\alpha g^\circ}^{f^z}(x)$. 

By Lemmas~\ref{lem: smooth dependence of sol} and~\ref{lem: volume form}, the map $\alpha \mapsto e^{i(\phi_\alpha(z,x)-t)\cdot\eta} a_\alpha(z,x)$ is a continuous map 
\[
\text{from } C^{16n^2+72n+39} = C^{3+2(8n^2+36n+18)} \text{ into } C^{2+ 8n^2+36n+18} = C^{8n^2+36n+20}
\]
for each fixed $\eta$. We will write $\ell' = 8n^2+36n+20$, so that $\alpha \mapsto e^{i(\phi_\alpha(z,x)-t)\cdot\eta} a_\alpha(z,x)$ is a continuous map $C^{\ell} \to C^{\ell'}$.
	
	Let $\omega_1 \in C_c^\infty(\mR)$ be identically $1$ near the origin. Inserting $\omega_1(\eta)$ into the formula \eqref{eq:myFIO}, we can compute the composition $P(\alpha)$ explicitly to have the kernel
	\[
		P(\alpha)(y,s,x,t) = \int e^{i(\phi_\alpha(z,x)-t)\cdot\eta+i(s-\phi_\alpha(z,y))\cdot\theta} p_{1,\alpha}(y,s,x,t,z,\eta,\theta)\dd \eta\dd\theta\dd z + R_{1,\alpha}(y,s,x,t)\,,
	\]
	where
	\[
		p_{1,\alpha}(y,s,x,t,z,\eta,\theta) = \overline{a_\alpha(z,y)}a_\alpha(z,x)(1-\omega_1(\eta)\omega_1(\theta))\chi^2(\abs{z}^2)\tilde{\chi}((x,t))\tilde{\chi}((y,s))\,.
	\]	
The remainder Schwartz kernel $R_{1,\alpha}(y,s,x,t)$ is an integral in the $(\eta, \theta, z)$ variable over compactly supported regions and therefore $\alpha \mapsto R_{1,\alpha}$ is a continuous map from $C^{\ell}$ to $C^{\ell'}$.

We perform the change of variable $z \mapsto \frac{z}{\sqrt{\abs{\theta}^2+\abs{\eta}^2}}$, and introduce
	\[
		\varphi_\alpha((y,s),(x,t),(z,\eta,\theta)) \coloneqq \left(\phi_\alpha\left(\frac{z}{\sqrt{\abs{\eta}^2+\abs{\theta}^2}},x\right)-t\right)\cdot\eta+\left(s-\phi_\alpha\left(\frac{z}{\sqrt{\abs{\theta}^2+\abs{\eta}^2}},y\right)\right)\cdot\theta\,,
	\]
	which is homogeneous of degree $1$ with respect to the variable $\sigma\coloneqq (\eta,\theta,z)$ so that it is in fact a phase. We shall also introduce the shorthand $\tilde{y} = (y,s)$ and $\tilde{x}=(x,t)$ and thus have
	\begin{equation}\label{eq:Pbeforemicrolocalize}
		P(\alpha)(\tilde{y},\tilde{x}) = \int e^{i\varphi_\alpha(\tilde{y},\tilde{x},\sigma)} p_{2,\alpha}(\tilde{y},\tilde{x},\sigma)\dd\sigma + R_{1,\alpha}\,,
	\end{equation}
	with 
	\[
			 p_{2,\alpha}(\tilde{y},\tilde{x},\sigma)= p_{2,\alpha}((y,s),(x,t),(z,\eta,\theta))=p_{1,\alpha}\left(y,s,x,t,\frac{z}{\sqrt{\abs{\eta}^2+\abs{\theta}^2}},\eta,\theta\right)(\abs{\eta}^2+\abs{\theta}^2)^{-N/2}\,,%
	\]
    which is homogeneous of degree $-N$ in $\sigma$ for $\sigma$ away from the origin. We are still in the position that $\alpha\mapsto p_{2,\alpha}$ is a continuous map $C^{\ell}\to C^{\ell'}$. This is also true for $\alpha\mapsto e^{i\varphi_\alpha}$ for each fixed $\sigma$. 

	Because $D\mathbf{F}_0(\alpha)$ satisfies the Bolker condition, we know from \cite{MR812288} that $P(\alpha)$ is the clean composition of (omitting the cut-off functions) $D\mathbf{F}_0(\alpha)^\ast$ with $D\mathbf{F}_0(\alpha)$, which is to say that the phase $\varphi_\alpha$ is clean. Furthermore, one can explicitly calculate that the excess is given by $e= \mathrm{dim}(\mathcal{G}(\BB))-\mathrm{dim}(\widetilde{M}) = N-(n+1)$. (Or by calculating that $\dim \{\nabla_\sigma \varphi_\alpha = 0\}=N+n+1$ and using the definition of the excess from \cite[Def.~21.2.15]{zbMATH05129478}.)

	Following the proof of \cite[Lem.~7.1]{zbMATH03481135} (or see the proof of \cite[Prop.~25.1.5']{zbMATH05528184}) this means that we can (after rearranging the components of $\sigma$) split $\sigma = (\bar \sigma,\sigma''')$ with $\bar \sigma = (\sigma_{1},\dots,\sigma_{n+3})$ and $\sigma''' = (\sigma_{n+4},\dots,\sigma_{N+2})$ so that the manifold $\{\nabla_\sigma \varphi_\alpha= 0\}$ is locally defined by 	
	\[
		\nabla_{\bar\sigma} \varphi_\alpha = 0
	\]
	and that $\{\nabla_\sigma \varphi_\alpha =0\}$ intersects $\varphi_{n+4,\dots,N+2}=\mathrm{const}$ transversally. Furthermore, for each fixed $\sigma'''$, $\varphi_\alpha(\tilde{y},\tilde{x},\sigma)=\varphi_\alpha(\tilde{y},\tilde{x},\bar\sigma,\sigma''')$ is a non-degenerate phase function as a function of $(\tilde{y},\tilde{x},\bar \sigma)$.

    At this point, by using a microlocal partition of unity, we shall work in a microlocal neighborhood of some point $(\tilde{y}_0,\tilde{x}_0,\sigma_0)$ so that $\nabla_\sigma \varphi_\alpha(\tilde{y}_0,\tilde{x}_0,\sigma_0)$ vanishes. In particular, we may assume that $p_{2,\alpha}$ has slim conic support with respect to $\sigma$ near $\sigma_0$.

	Introducing a cut-off $\omega_2\in C_c^\infty(\mR^{n+3})$ identically $1$ near the origin, performing the change of variables $\sigma'''\mapsto \abs{\bar\sigma}\sigma'''$ we find
	\[
		P(\alpha)(\tilde{y},\tilde{x}) = \iint e^{i\varphi_{\alpha}(\tilde{y},\tilde{x},\bar \sigma,\abs{\bar\sigma}\sigma''')} \abs{\bar\sigma}^{N-n-1}(1-\omega_2(\bar\sigma))p_{2,\alpha}(\tilde{y},\tilde{x},\bar\sigma,\abs{\bar\sigma}\sigma''')\dd\bar \sigma\dd\sigma'''+R_{2,\alpha}\,,
	\]
	where in the inner integral we interpret $\sigma'''$ as a fixed parameter and $\alpha\mapsto R_{2,\alpha}$ is continuous $C^{\ell}\to C^{\ell'}$. 
    
    Because $p_{2,\alpha}$ has slim conic support (in $\sigma$), we see that there is some $C>0$ (independent of $\tilde{y},\tilde{x},\bar\sigma$) so that $p_{2,\alpha}(\tilde{y},\tilde{x},\bar\sigma,\abs{\bar\sigma}\sigma''') = 0$ for $\abs{\sigma'''} > C$. Let us thus focus our attention on the inner integral for which we introduce new notation
	\begin{equation}\label{eq:introPQ}
		P(\alpha) = \int Q_{\sigma'''}(\alpha) \dd \sigma''' +R_{2,\alpha}\,,\quad Q_{\sigma'''}(\alpha) = \int e^{i\varphi_{\alpha,\sigma'''}(\tilde{y},\tilde{x},\bar \sigma)} p_{3,\alpha;\sigma'''}(\tilde{y},\tilde{x},\bar\sigma)\dd\bar \sigma\,,
	\end{equation}
	where
	\[
		\varphi_{\alpha,\sigma'''}(\tilde{y},\tilde{x},\bar \sigma) = \varphi_{\alpha}(\tilde{y},\tilde{x},\bar \sigma,\abs{\bar\sigma}\sigma''')\,,
	\]
	and 
	\[
		p_{3,\alpha;\sigma'''}(\tilde{y},\tilde{x},\bar\sigma) = \abs{\bar\sigma}^{N-n-1}p_{2,\alpha}(\tilde{y},\tilde{x},\bar\sigma,\abs{\bar\sigma}\sigma''')\,,
	\]
    which is homogeneous of degree $-n-1$ in $\bar\sigma$ away from the origin. We may again assume that $p_{3,\alpha;\sigma'''}$ has support in some microlocal neighborhood of some $(\tilde{y}_0,\tilde{x}_0,\bar\sigma_0)$.
    
	Following the proof of \cite[Lem.~7.1]{zbMATH03481135}, we know that $\varphi_{\alpha,\sigma'''}$ as a function of $(\tilde{y},\tilde{x},\bar \sigma)$ is a non-degenerate phase function and 
    \[
    \nabla_{\bar\sigma}\varphi_{\alpha,\sigma'''}(\tilde{y},\tilde{x},\bar\sigma) = 0\iff \nabla_{\bar\sigma} \varphi_\alpha(\tilde{y},\tilde{x},\bar\sigma,\abs{\bar\sigma}\sigma''') = 0 \iff \nabla_{\sigma}\varphi_\alpha(\tilde{y},\tilde{x},\bar\sigma,\sigma''') = 0\,.
    \]

    Because $D\mathbf{F}_0(\alpha)$ satisfies the Bolker condition, we know that $\nabla_{\sigma}\varphi_\alpha = 0 \implies \tilde{y}=\tilde{x}, \nabla_{\tilde{y}} \varphi_\alpha(\tilde{y},\tilde{x},\sigma)=-\nabla_{\tilde{x}}\varphi_\alpha(\tilde{y},\tilde{x},\sigma)$. This can be verified either directly or using the fact that we know the canonical relation related to the phase $\varphi_\alpha$ to be a subset of the graph of the diagonal. In any case, we have deduced that 
\begin{equation}\label{eq:samemanifold}
\nabla_{\bar\sigma}\varphi_{\alpha,\sigma'''}(\tilde{y},\tilde{x},\bar \sigma) = 0\implies \tilde{y}=\tilde{x}\,,\,\nabla_{\tilde{y}}\varphi_{\alpha,\sigma'''}(\tilde{y},\tilde{x},\bar \sigma) = -\nabla_{\tilde{x}}\varphi_{\alpha,\sigma'''}(\tilde{y},\tilde{x},\bar \sigma)\,,
\end{equation}
and $\varphi_{\alpha,\sigma'''}$ is a non-degenerate phase function. In particular, this reaffirms that $P(\alpha)$ and $Q_{\sigma'''}(\alpha)$ must be pseudodifferential operators (see for example \cite[\S~2.5]{zbMATH05817029}). 

We now reduce the dimension of the frequency variables to the minimum possible, which we know to be $n+1$ by \cite[Lem.~2.3.5]{zbMATH05817029}. %
Here one may follow the beginning of the proof of \cite[Thm.~2.3.4]{zbMATH05817029} (see also \cite[Prop.~18.4.6]{MR4436039}), which will lead to $ \bar\sigma = (\sigma',\sigma'')\in \mR^{n+3} = \mR^{n+1}\times\mR^{2}$,
\begin{equation}\label{eq:decomp}
		\varphi_{\alpha,\sigma'''}(\tilde{y},\tilde{x},\bar\sigma) = \varphi_{\alpha,\sigma'''}^\flat(\tilde{y},\tilde{x},\sigma') + \varphi_{\alpha,\sigma'''}^\sharp(\tilde{y},\tilde{x},\bar\sigma)\,,
\end{equation}
	where $\varphi_{\alpha,\sigma'''}^\flat$ is a non-degenerate phase function with
	\[
		\nabla_{\sigma'} \varphi_{\alpha,\sigma'''}^\flat=0\iff \nabla_{\bar\sigma} \varphi_{\alpha,\sigma'''} = 0 \iff \bar\sigma = (\sigma',\sigma''(\tilde{y},\tilde{x},\sigma'))\implies \nabla_{(\tilde{y},\tilde{x})} \varphi_{\alpha,\sigma'''}^\flat = \nabla_{(\tilde{y},\tilde{x})} \varphi_{\alpha,\sigma'''}\,,
	\]
	and $\partial_{\sigma''}^2 \varphi_{\alpha,\sigma'''}^\sharp$ is non-degenerate and $\partial_{\sigma'}^2\varphi_{\alpha,\sigma'''}^\flat = 0$. Here, $\sigma''(\cdot,\cdot,\cdot)$ is a function defined by the relation above near some $\tilde{y}_0,\tilde{x}_0,\sigma'_0$ so that $\bar\sigma_0 = (\sigma'_0,\sigma''(\tilde{y}_0,\tilde{x}_0,\sigma'_0))$.

    Together with \eqref{eq:samemanifold} we then have
    \[
        \nabla_{\sigma'} \varphi_{\alpha,\sigma'''}^\flat=0\implies \tilde{y}=\tilde{x}\,, \nabla_{\tilde{y}}\varphi_{\alpha,\sigma'''}^\flat(\tilde{y},\tilde{x},\sigma') = -\nabla_{\tilde{x}}\varphi_{\alpha,\sigma'''}(\tilde{y},\tilde{x},\sigma')\,,
    \]
    so that \cite[Thm.~3.1.6]{FIO1} implies that there is a diffeomorphism $(\tilde{y},\tilde{x},\sigma')\mapsto(\tilde{y},\tilde{x},\psi_{\alpha;\sigma'''}(\tilde{y},\tilde{x},\sigma'))$ (defined from a conic neighborhood of some point $(\tilde{y}_0,\tilde{x}_0,\tilde{\sigma}'_0)$ to a conic neighborhood of the point $(\tilde{y}_0,\tilde{x}_0,\sigma_0')$) so that 
	\begin{equation}\label{eq:equivphase}
		\varphi_{\alpha;\sigma'''}^\flat(\tilde{y},\tilde{x},\psi_{\alpha;\sigma'''}(\tilde{y},\tilde{x},\sigma')) = (\tilde{y}-\tilde{x})\cdot \sigma'\,.
	\end{equation}
    (See also \cite[Def.~18.4.17]{MR4436039}.) We shall shelve this result for a moment.
    
	Recalling \eqref{eq:decomp}, together with the introduction of a smooth cut-off $\omega_3\in C_c^\infty(\mR^{n+1})$ identically $1$ near the origin, leads to
	\[
		Q_{\sigma'''}(\alpha)(\tilde{y},\tilde{x}) = \int_{\mR^{n+1}} e^{i\varphi_\alpha^\flat(\tilde{y},\tilde{x},\sigma')}  p_{4,\alpha;\sigma'''}(\tilde{y},\tilde{x},\sigma') \dd\sigma' + R_{3,\alpha;\sigma'''}
	\]
	with $\alpha\mapsto R_{3,\alpha;\sigma'''}$ a continuous map $C^{\ell}\to C^{\ell'}$, and $R_{3,\alpha;\sigma'''}$ has compact support with respect to $\sigma'''$, and
	\[
		p_{4,\alpha;\sigma'''}(\tilde{y},\tilde{x},\sigma') = (1-\omega_2(\sigma'))\int e^{i\varphi_\alpha^\sharp(\tilde{y},\tilde{x},\sigma', \sigma'')} p_{3,\alpha;\sigma'''}(\tilde{y},\tilde{x},\sigma', \sigma'')\dd\sigma''\,.
	\]
	Using substitution $\sigma'' \mapsto \abs{\sigma'}\sigma''$ and the homogeneity of the phase, we find that
	\begin{equation}\label{eq:p4}
		p_{4,\alpha;\sigma'''}(\tilde{y},\tilde{x},\sigma')=(1-\omega_2(\sigma'))\abs{\sigma'}^{2}\int e^{i\abs{\sigma'}\varphi_\alpha^\sharp(\tilde{y},\tilde{x},\abs{\sigma'}^{-1}\sigma', \sigma'')} p_{3,\alpha;\sigma'''}(\tilde{y},\tilde{x},\sigma', \abs{\sigma'}\sigma'')\dd\sigma''\,.
	\end{equation}
	Recall that we are performing our calculations on a small microlocal neighborhood, so that $p_{3,\alpha;\sigma'''}$ has small conic support with respect to $\bar\sigma$. In particular, there is $C>0$ so that $p_{3,\alpha;\sigma'''}(\tilde{y},\tilde{x},\sigma',\abs{\sigma'}\sigma'')$ vanishes when $\abs{\sigma''} > C$ (see also the remarks near \cite[Eq.~(3.2.4)]{FIO1}). We will employ the method of stationary phase to show that $p_{4,\alpha;\sigma'''}$ (and a finite amount of its derivatives) behaves like $\mathcal{O}(\abs{\sigma'}^{-n})$ for large $\sigma'$. 
    
    For any $k\in\mathbb{N}$ we may apply the method of stationary phase (\cite[Thm.~7.7.6]{hoermander1}) to expand the integral expression in $p_{4,\alpha;\sigma'''}$ up to $k$ terms the sum of which we denote by $\tilde{p}^k_{4,\alpha;\sigma'''}$, which is a finite sum of terms homogeneous of degree $\leq -n$ in $\sigma'$ for $\sigma'$ away from the origin. Also, for some $C>0$
    \begin{equation}\label{eq:O1}
        \abs{p_{4,\alpha;\sigma'''}-\tilde{p}^k_{4,\alpha;\sigma'''}} \leq C\abs{\sigma'}^{2-k}\sum_{\abs{\beta}\leq 2k}\sup_{\sigma''}\abs{\p^\beta_{\sigma''} p_{3,\alpha;\sigma'''}(\tilde{y},\tilde{x},\sigma',\abs{\sigma'}\sigma'')} = \mathcal{O}(\abs{\sigma'}^{-n+1-k})\,,
    \end{equation}
    since $p_{3,\alpha;\sigma'''}$ was homogeneous of degree $-n-1$ in $\bar\sigma$ away from the origin. 
    Because $\tilde{p}^k_{4,\alpha;\sigma'''}$ is the sum of terms homogeneous of degree $\leq -n$ in $\sigma'$ away from the origin, and because we have the explicit formula \eqref{eq:p4}, we also deduce that for all multi-indices $\rho \in \mathbb{N}_0^{3(n+1)}$
    \begin{equation}\label{eq:O2}
        \abs{\partial^\rho(p_{4,\alpha;\sigma'''}-\tilde{p}^k_{4,\alpha;\sigma'''})} \leq \abs{\partial^\rho p_{4,\alpha;\sigma'''}}+\abs{\partial^\rho \tilde{p}^k_{4,\alpha;\sigma'''}} = \mathcal{O}(\abs{\sigma'}^{-n+1+\abs{\rho}}+\abs{\sigma}^{-n}) = \mathcal{O}(\abs{\sigma'}^{-n+1+\abs{\rho}})\,,
    \end{equation}
    and the quantity on the LHS depends on the first $2k+2+\abs{\rho}$ derivatives of $\varphi_\alpha^\sharp$ and $p_{3,\alpha;\sigma'''}$ (by the explicit formulas for $p_{4,\alpha;\sigma'''}$ and $\tilde{p}^k_{4,\alpha;\sigma'''}$, the latter of which given by \cite[Thm.~7.7.6]{hoermander1}).

    Following the proof of \cite[Thm.~2.9]{zbMATH03267579} we see that for any $j\geq 0$ and all compact sets $K_1,K_2 \subset \mR^{3(n+1)}$ with $K_1 \Subset K_2^{\mathrm{int}}$, for some $C>0$,
    \begin{gather*}
        \sup_{K_1,\abs{\rho}=j+1}\abs{\p^\rho (p_{4,\alpha;\sigma'''}-\tilde{p}^k_{4,\alpha;\sigma'''})} \\
        \leq C\sup_{K_2,\abs{\rho_1}=j,\abs{\rho_2}=j+2}\abs{\p^{\rho_1}(p_{4,\alpha;\sigma'''}-\tilde{p}^k_{4,\alpha;\sigma'''})}\left(\abs{\p^{\rho_1}(p_{4,\alpha;\sigma'''}-\tilde{p}^k_{4,\alpha;\sigma'''})}+\abs{\p^{\rho_2} (p_{4,\alpha;\sigma'''}-\tilde{p}^k_{4,\alpha;\sigma'''})}\right)\,.
    \end{gather*}
    Iteratively plugging in $j=0, j=1,\dots,j=4n+2$ above, and using \eqref{eq:O1} and \eqref{eq:O2} we find that
    \[
        \sup_{\abs{\rho}\leq 4n+2}\abs{\partial^\rho(p_{4,\alpha;\sigma'''}-\tilde{p}^k_{4,\alpha;\sigma'''})} = \mathcal{O}(\abs{\sigma'}^{(4n+4)(4n+5)/2-2-(4n+3)n-k}) = \mathcal{O}(\abs{\sigma'}^{-n})
    \]
    if $k\geq 4n^2+16n+8$. Taking $k = 4n^2+16n+8$, we find that%
    \begin{equation}\label{eq:boundedinsigma}
        \norm{\langle \sigma'\rangle^n(p_{4,\alpha;\sigma'''}-\tilde{p}^k_{4,\alpha;\sigma'''})}_{C^{4n+2}_{\tilde{y},\tilde{x},\sigma'}} < \infty\,,\qquad\text{and}\qquad \norm{\langle \sigma'\rangle^n\tilde{p}^k_{4,\alpha;\sigma'''}}_{C^{4n+2}_{\tilde{y},\tilde{x},\sigma'}} < \infty\,,
    \end{equation}
    where $\langle \sigma'\rangle = (1+\abs{\sigma'}^2)^{1/2}$ is the Japanese bracket. The second part in \eqref{eq:boundedinsigma} follows from the explicit formula for $\tilde{p}_{4,\alpha;\sigma'''}^k$ and that this expression is the sum of terms homogeneous of degree $\leq -n$ in $\sigma'$ away from the origin.

    Writing $p_{4,\alpha;\sigma'''}=(p_{4,\alpha;\sigma'''}-\tilde{p}^k_{4,\alpha;\sigma'''})+\tilde{p}^k_{4,\alpha;\sigma'''}$ we know from the remarks after \eqref{eq:O2} and the explicit formula for $\tilde{p}_{4,\alpha;\sigma'''}^k$ that 
    $\sup_{\abs{\rho}\leq 4n+2}\abs{\p^\rho p_{4,\alpha;\sigma'''}}$ depends on $2k+2+4n+2 = 8n^2+36n+20=\ell'$ derivatives of $\varphi_\alpha^\sharp$ and $p_{3,\alpha;\sigma'''}$. Combined with \eqref{eq:boundedinsigma}, we find that $\alpha\mapsto p_{4,\alpha;\sigma'''}$ is a continuous map from $C^{\ell}$ to the set of functions $p$ satisfying
	\[
		\norm{\langle \sigma'\rangle^n p(\tilde{y},\tilde{x},\sigma')}_{C^{4n+2}_{\tilde{y},\tilde{x},\sigma'}} < \infty\,.
	\]

	Recalling \eqref{eq:equivphase}, if the microlocal support, say $\Gamma_{\tilde{y},\tilde{x}}$, of the amplitude $p_{4,\alpha;\sigma'''}$ is small enough, it is compactly contained in $\psi_{\alpha;\sigma'''}(\tilde{y},\tilde{x},\tilde{\Gamma}_{\tilde{y},\tilde{x}})$ where $\tilde{\Gamma}_{\tilde{y},\tilde{x}}$ is some conic neighborhood of the point $(\tilde{y}_0,\tilde{x}_0,\tilde{\sigma}'_0)$. Therefore, 
    \[
        Q_{\sigma'''}(\alpha)-R_{3,\alpha;\sigma'''} = \int_{\mR^{n+1}} e^{i\varphi_\alpha^\flat(\tilde{y},\tilde{x},\sigma')}  p_{4,\alpha;\sigma'''}(\tilde{y},\tilde{x},\sigma') \dd\sigma' = \int_{\psi_{\alpha;\sigma'''}(\tilde{\Gamma}_{\tilde{y},\tilde{x}})} e^{i\varphi_\alpha^\flat(\tilde{y},\tilde{x},\sigma')}  p_{4,\alpha;\sigma'''}(\tilde{y},\tilde{x},\sigma') \dd\sigma'\,,
    \]
    and using substitution,
	\[
		Q_{\sigma'''}(\alpha)(\tilde{y},\tilde{x}) = \int_{\tilde{\Gamma}_{\tilde{y},\tilde{x}}} e^{i(\tilde{y}-\tilde{x})\cdot \sigma'} p_{5,\alpha;\sigma'''}(\tilde{y},\tilde{x},
    \sigma')\dd\sigma'+R_{3,\alpha;\sigma'''}\,,%
	\]
    where
    \[
        p_{5,\alpha;\sigma'''}(\tilde{y},\tilde{x},\sigma') \coloneqq p_{4,\alpha;\sigma'''}(\tilde{y},\tilde{x},\psi_{\alpha;\sigma'''}(\tilde{y},\tilde{x},\sigma'))%
        \abs{\mathrm{det}\partial_{\sigma'}\psi_{\alpha;\sigma'''}(\tilde{y},\tilde{x},\sigma')}\,.
    \]
    Since we had assumed that $p_{4,\alpha;\sigma'''}$ vanishes outside of a compact subset of $\tilde{\Gamma}_{\tilde{y},\tilde{x}}$, we extend (the definition of) $p_{5,\alpha;\sigma'''}$ as $0$ outside of $\tilde{\Gamma}_{\tilde{y},\tilde{x}}$ without modifying $Q_{\sigma'''}$, so that
    \[
        Q_{\sigma'''}(\alpha)(\tilde{y},\tilde{x}) = \int_{\mR^{n+1}} e^{i(\tilde{y}-\tilde{x})\cdot \sigma'} p_{5,\alpha;\sigma'''}(\tilde{y},\tilde{x},
    \sigma')\dd\sigma'+R_{3,\alpha;\sigma'''}\,.
    \]
    We note that because $\psi_{\alpha;\sigma'''}$ is homogeneous of degree $1$ in $\sigma'$, $p_{5,\alpha;\sigma'''}(\tilde{y},\tilde{x},\sigma')$ is indeed still an amplitude of order $-n$ and $\alpha \mapsto p_{5,\alpha;\sigma'''}$ is continuous $C^{\ell} \to S^n_{4n+1}$, where $S^n_{j}$ is the space of measurable functions $p$ so that
	\begin{equation}\label{eq:defineSj}
		\norm{p}_{S^n_{j}} = \sup_{\sigma'\in \mR^{n+1}}\norm{\langle \sigma'\rangle^n p(\tilde{y},\tilde{x},\sigma')}_{C^{j}_{\tilde{y},\tilde{x}}} < \infty\,.
	\end{equation}
	In particular, recalling the definition of $P$ and $Q_{\sigma'''}$ we have that
	\[
		P(\alpha)(\tilde{y},\tilde{x}) = \int Q_{\sigma'''}(\alpha)(\tilde{y},\tilde{x}) \dd\sigma'''+R_{2,\alpha} = \int_{\mR^{n+1}} e^{i(\tilde{y}-\tilde{x})\cdot \sigma'} \int %
        p_{5,\alpha;\sigma'''}(\tilde{y},\tilde{x},\sigma')\dd\sigma''' \dd \sigma'+R_{4,\alpha}\,,%
	\]
    where
    \[
        \alpha \mapsto R_{4,\alpha} \coloneqq R_{2,\alpha}+\int R_{3,\alpha;\sigma'''} \dd \sigma'''
    \]
    is continuous $C^{\ell}\to C^{\ell'}$ (recall that $R_{3,\alpha;\sigma'''}$ has compact support with respect to $\sigma'''$). 

    Thus, introducing $\omega_4 \in C_c^\infty(\mR^{n+1})$ with $\int \omega_4 \dd\sigma' = 1$, we have
    \begin{equation}\label{eq:explicitp}
        P(\alpha)(\tilde{y},\tilde{x}) = \int_{\mR^{n+1}} e^{i(\tilde{y}-\tilde{x})\cdot\sigma'}p_{6,\alpha}(\tilde{y},\tilde{x},\sigma') \dd\sigma'\,,
    \end{equation}
    where
	\[
		p_{6,\alpha}(\tilde{y},\tilde{x},\sigma') = \int %
        p_{5,\alpha;\sigma'''}(\tilde{y},\tilde{x},\sigma')\dd\sigma'''+ \omega_4(\sigma')e^{-i(\tilde{y}-\tilde{x})\cdot\sigma'}R_{4,\alpha}\,,
	\]
	and we recall (from remarks prior to \eqref{eq:introPQ}) that the integrand in the first term above is compactly supported with respect to $\sigma'''$. Thus, we know that $\alpha\mapsto p_{6,\alpha}$ is continuous $C^{\ell} \to S^n_{4n+1}$ because this was true for $\alpha\mapsto p_{5,\alpha;\sigma'''}$ for each fixed $\sigma'''$.

    Note that \eqref{eq:explicitp} is true in some neighborhood $\Gamma \subset \mR^{2(n+1)}$ of the diagonal $\tilde{y}=\tilde{x}$ by using a microlocal partition of unity and the knowledge that \eqref{eq:explicitp} holds microlocally near any $(\tilde{y},\tilde{x},\sigma')$ so that $\tilde{y}=\tilde{x}$. Finally, introducing a cut-off $\omega_5 \in C^\infty(\mR^{2(n+1)})$ that is identically $1$ near the diagonal and vanishes outside of $\Gamma$, we have 
    \begin{equation}\label{eq:finalformulaP}
        P(\alpha)(\tilde{y},\tilde{x}) = \int_{\mR^{n+1}}e^{i(\tilde{y}-\tilde{x})\cdot\sigma'} p_\alpha(\tilde{y},\tilde{x},\sigma')\dd\sigma'\,,
    \end{equation}
    where
    \[
        p_\alpha(\tilde{y},\tilde{x},\sigma') = \omega_5(\tilde{y},\tilde{x})p_{6,\alpha}(\tilde{y},\tilde{x},\sigma') + (1-\omega_5(\tilde{y},\tilde{x}))e^{-i(\tilde{y}-\tilde{x})\cdot\sigma'}\omega_4(\sigma')P(\alpha)(\tilde{y},\tilde{x})\,,
    \]
    and we use the formula \eqref{eq:Pbeforemicrolocalize} for $P(\alpha)(\tilde{y},\tilde{x})$ in the definition of $p_\alpha$.%

    Using the representation of $P(\alpha)(\tilde{y},\tilde{x})$ from \eqref{eq:Pbeforemicrolocalize}, we verify that for $\tilde{y}$ away from $\tilde{x}$, using partial integration $4n+1+3$ times, we can write \[
        P(\alpha)(\tilde{y},\tilde{x}) = \int e^{i\varphi_\alpha(\tilde{y},\tilde{x},\sigma)} \tilde{p}_{2,\alpha}(\tilde{y},\tilde{x},\sigma)\dd\sigma\,,
    \] 
    where $\tilde{p}_{2,\alpha}(\tilde{y},\tilde{x},\sigma)$ is a sum of terms homogeneous of degree $\leq -N-3-(4n+1)$ away from the origin in $\sigma$ and $\alpha\mapsto \tilde{p}_{2,\alpha}$ is a continuous map $C^{\ell}\to C^{\ell'-(4n+4)}=C^{8(n^2+3n+2)}$. (Akin to the definition of $p_{6,\alpha}$ above, we have absorbed the remainder $R_{1,\alpha}$ from \eqref{eq:Pbeforemicrolocalize} into $\tilde{p}_{2,\alpha}$.) We find that $\partial^\beta_{\tilde{y},\tilde{x}}P(\alpha)(\tilde{y},\tilde{x})$ is given by the integral of an absolutely integrable function for all $\abs{\beta}\leq 4n+1$, so that $\alpha\mapsto (1-\omega_5(\tilde{y},\tilde{x}))e^{-i(\tilde{y}-\tilde{x})\cdot\sigma'}\omega_4(\sigma')P(\alpha)(\tilde{y},\tilde{x})$ is a continuous map $C^{\ell} \to S^n_{4n+1}$ (the latter defined in \eqref{eq:defineSj}).
    
    In total, we conclude that $\alpha\mapsto p_\alpha$ is continuous $C^{\ell} \to S^n_{4n+1}$. Thus, since the explicit formula \eqref{eq:finalformulaP} is true for all $\tilde{y},\tilde{x}$ (in the coordinated neighborhood we are working in), using Proposition~\ref{prop:hbounded} completes the proof.
\end{proof}

\subsection{Proof of Theorem~\ref{thm:stability}}

\begin{proof}[Proof of Theorem~\ref{thm:stability}]
Recall the setting of Section~\ref{sec:suffsurf}; in particular, we choose $\Sigma_1,\dots,\Sigma_K$ to be the finite set of admissible minimal surfaces for the analytic metric $g^\circ$ constructed in Lemma~\ref{lem:findfiniteminsurfs} and take $\BB_1,\dots,\BB_K$ from this result too. %
Define $\chi\in C^\infty_c(B_1^{N})$ as in \eqref{eq: cut off def} and $\tilde\chi\in C^\infty_c(\widetilde M), 0\leq \tilde{\chi}\leq 1$ satisfying \eqref{eq:supptdchi} and $\tilde{\chi}\equiv 1$ on an open neighborhood of $M$. As a consequence of \eqref{eq: order of Q}, letting $Q(\alpha)^\ast Q(\alpha)$ be defined analogously as in \eqref{eq:ellop} where every occurrence of $D\mathbf{F}(1)$ is replaced with $D\mathbf{F}(\alpha)$,
\begin{align}\label{eq: def Qalpha}
	Q(\alpha)^\ast Q(\alpha) &= \tilde{\chi}(D\mathbf{F}(\alpha))^\ast D\mathbf{F}(\alpha)\tilde{\chi} + (1-\tilde{\chi})\langle D_{g^\circ}\rangle^{-n} (1-\tilde{\chi}) \\
    &= \sum_{j=1}^K \tilde{\chi}(D\mathbf{F}_j(\alpha))^\ast \chi^2 D\mathbf{F}_j(\alpha)\tilde{\chi} + (1-\tilde{\chi})\langle D_{g^\circ}\rangle^{-n} (1-\tilde{\chi}) \in \Psi^{-n}(S)\,.\notag
\end{align}

As a consequence of Lemma~\ref{lem:bolkerstable}, it is not difficult to see that Propositions~\ref{prop:ellonset} and ~\ref{prop:injellop} (aside from the injectivity statement or the statement about being an analytic double fibration transform) hold true verbatim when every occurrence of $D\mathbf{F}(1)$ is replaced by $D\mathbf{F}(\alpha)$ for any $\alpha$ near enough to $1$ in $H^{3+\frac{n+1}{2}+\gamma}(\widetilde{M})$. 

We will show here that, in fact, for $\alpha$ near $1$ in $C^{16n^2+72n+39}(\widetilde{M})$, the operator $Q(\alpha)^\ast Q(\alpha)$ is also invertible. Recall that we showed at the beginning of the proof of Proposition~\ref{prop:lowerbd} that
	\[
		Q(1)^\ast Q(1) \colon L^2(S) \to H^n(S)
	\]
	was invertible with an inverse we denoted by $T$. Consider now
	\[
		T Q(\alpha)^\ast Q(\alpha) = \mathrm{id}+ T(Q(\alpha)^\ast Q(\alpha)- Q(1)^\ast Q(1))\,,
	\]
	where if $\norm{\alpha-1}_{C^{16n^2+72n+39}}$ is small enough, the boundedness of $T$ and Proposition~\ref{prop:amplitudesketch} imply that we may invert using a Neumann series to find that indeed $Q(\alpha)^\ast Q(\alpha)$ is also invertible as a map $L^2(S) \to H^n(S)$. This is because $Q(\alpha)^\ast Q(\alpha)-Q(1)^\ast Q(1)$ is the finite sum of operators of the type $P(\alpha)-P(1)$ treated in Proposition~\ref{prop:amplitudesketch}, with the underlying minimal surface $\Sigma_0$ replaced by some $\Sigma_j, j\in\{1,\dots,K\}$.

	Proceeding as in the proof of Proposition~\ref{prop:lowerbd}, we can then show that there is a constant $C>0$ so that for all $\alpha$ near $1$ in $C^{16n^2+72n+39}$, we have
	\[
		\norm{D\mathbf{F}(\alpha)\beta}_{H^{\frac{n}{2}}(\mR^{N})\times\dots\times H^{\frac{n}{2}}(\mR^{N})} \geq C\norm{\beta}_{L^2(M)}\,.
	\]
	Finally, following the proof of Theorem~\ref{thm:rigidity}, an application of \cite[Thm.~2(b)]{stefanov2009linearizing} completes this proof. Here one will have to use Sobolev embedding theorems since we are working with spaces of continuous functions as well as Sobolev spaces.
\end{proof}

\appendix

\section{Dependence of the areas of minimal surfaces on the metric}\label{subsec:implicitfunc} 

The purpose of this section is to prove Theorem~\ref{thm: smooth dependence} and similar results required in other sections. For background on Fr\'echet differentiability we refer to \cite{MR960687}. Recall the definition 
\begin{equation*}%
G =%
H^{3+(n+1)/2 + \gamma}\left(\widetilde{M}; \sigma\big(T^*\widetilde{M} \otimes T^*\widetilde{M}\big)\right) 
\end{equation*}
of $G$ from \eqref{eq: def G}. Let $g^\circ \in G$ be fixed, and assume as usual that $\Sigma^\circ$ is an admissible minimal surface with boundary $\partial \Sigma^\circ = \partial \widetilde{M} \cap \Sigma^\circ$. %
Furthermore, throughout this section $G^\circ \subset G$ will be a small open neighborhood of $g^\circ$ and we denote for any $k\in\mathbb{N}$, and any $\delta >0$
\[
	G^{3+k,\gamma}_\circ \coloneqq G^\circ \cap C^{3+k,\gamma}\,,\quad C^{2+k,\gamma}_\delta(\partial\Sigma^\circ) = \{f \in C^{2+k,\gamma}(\p \Sigma^\circ)\colon \norm{f}_{C^{2+k,\gamma}(\p \Sigma^\circ)} < \delta\}\,.
\]
Notice that when $k = 0$, $C^{2,\gamma}_\delta(\partial\Sigma^\circ) = U_\delta$.

Minimal surfaces for the metric $g^\circ$ are characterized by solving \eqref{eq: MSE}. We first give a different characterization based on the area functional.
\begin{Lemma}\label{lem:newdef}
Let $k\in\mathbb{N}$. The map from \eqref{eq: area formula}, 
\begin{equation}\label{eq:areaformula2}
   G^{3+k,\gamma}_\circ\times C^{2+k,\gamma}(\Sigma^\circ) \ni (g,u)\mapsto A_{g}(u) = \int_{\Sigma_u}\dd\mathrm{Vol}^{g}_{\Sigma_u} \in \mR    
\end{equation}
is $C^{2}$-Fr\'echet differentiable, where we used the notation $\Sigma_u=\{(x,u(x))\colon x\in\Sigma^\circ\}$ for any $u\in C^{2,\gamma}(\Sigma^\circ)$.

Denote by $C^{2+k,\gamma}_{\rm Dir}(\Sigma^\circ) \subset C^{2+k,\gamma}(\Sigma^\circ)$ those functions with $0$ trace on $\partial \Sigma^\circ$, and let $g \in G$. A function $u\in C^{2+k,\gamma}(\Sigma^\circ)$ defines a minimal surface $\Sigma_u$ for $g$ if and only if
\begin{equation}\label{eq:defofminsurf}
		D_u A_{g}(u)\vert_{C^{2+k,\gamma}_{\rm Dir}(\Sigma^\circ)} \colon C^{2+k,\gamma}_{\rm Dir}(\Sigma^\circ) \to \mR\,,\quad\text{satisfies}\quad D_u A_{g}(u)\vert_{C^{2+k,\gamma}_{\rm Dir}(\Sigma^\circ)} = 0\,.
\end{equation}
\end{Lemma}
Before giving the proof, we calculate a Fermi coordinate expression (see \cite[\S~5]{zbMATH06897812}) for the volume form ${\rm dVol}^{g}_{\Sigma_u}$. 
\begin{Lemma}\label{lem: volume form}
Let $G^\circ\subset G$ be a small open neighborhood of $g^\circ$ and $g \in G^\circ$. Suppose $(x,t)\in \Sigma^\circ\times (-\epsilon, \epsilon), \epsilon>0$ is a Fermi coordinate system for the metric $g^\circ$ and let $u\in C^{2,\gamma}(\Sigma^\circ)$. Suppose in normal coordinates of $g^\circ$ the metric $g$ has the expression
$$g = g_{jk}(x,t) dx^k dx^j + g_{tt}(x,t) dt^2 + \omega(x,t) \otimes dt  + dt\otimes \omega(x,t)$$
where $\omega(\cdot, t)$ is a smooth family of one-forms on $\Sigma^\circ$ parametrized by $t\in (-\epsilon,\epsilon)$ and $g_{tt}(x,t) >0$.
The volume form ${\rm dVol}^{g}_{\Sigma_u} :={\rm dVol}^{\iota_{\Sigma_u}^*g}_{\Sigma_u}$ has the local coordinate expression
\begin{multline}\label{eq: coordinate expression dvol}
{\rm dVol}^{g}_{\Sigma_u} \\ 
=|h_u|^{1/2}\sqrt{ {\rm det}\left(I_{n\times n} + g_{tt}(x,u(x))(h_u^{-1} du) \otimes du + (h_u^{-1} \omega) \otimes du+ (h_u^{-1} du) \otimes \omega\right)}\s 
dx_1\wedge\dots\wedge dx_n,
\end{multline}
where $(I_{n\times n} + g_{tt}(x,u(x))(h_u^{-1} du) \otimes du) $ is invertible.
\end{Lemma}
\begin{proof}
Recall that the notation $h_u$   for $g(x,u(x))=g_{jk}(x,u(x))dx^kdx^j$. In Fermi coordinates, the pull-back metric of $g$ pulled back to $\Sigma_u$ is given by
\begin{eqnarray}
\label{eq: pullback metric}\nonumber
\iota_{\Sigma_u}^* g =\big( g_{jk}(x, u(x))  + g_{tt}(x,u(x)) \partial_{x^j} u \partial_{x^k}u\big) dx^j dx^k + \omega(x,u(x)) \otimes du + du \otimes \omega(x,u(x))\,,
\end{eqnarray}
which means that the volume form is given by 
\begin{eqnarray}
&&{\rm det}\left(g_{jk}(x, u(x))dx^k \otimes dx^j  + \partial_{x^j} u \partial_{x^k}u dx^k\otimes dx^j+ \omega(x,u(x)) \otimes du + du \otimes \omega(x,u(x))\right)\\\nonumber
&&  = |h_u| {\rm det}\left(I_{n\times n} + g_{tt}(x,u(x))(h_u^{-1} du) \otimes du + (h_u^{-1} \omega) \otimes du+ (h_u^{-1} du) \otimes \omega\right)\,. %
\end{eqnarray}
By Sylvester's determinant theorem,
$${\rm det}\big(I_{n\times n} + g_{tt}(x,u(x))(h_u^{-1} du) \otimes du\big) = (1 +g_{tt}(x,u(x)) |du|_{h_u}^2),$$
so that $(I_{n\times n} + g_{tt}(x,u(x))(h_u^{-1} du) \otimes du) $ is invertible.
\end{proof}
\begin{proof}[Proof of Lemma~\ref{lem:newdef}]
By direct observation of \eqref{eq:areaformula2} and \eqref{eq: coordinate expression dvol} and the fact that $g\in C^{3,\gamma}$ when $g\in G^{3+k,\gamma}_\circ$, we find the desired Fr\'echet differentiability.

The second statement is precisely the variational definition of minimal surfaces (see the remarks prior to \cite[Chp.~1,~Def.~1.4]{MR2780140}).
\end{proof}

We shall give another characterization of minimal surfaces.
\begin{Lemma}\label{lem:linofmin}Let $ k\in\mathbb{N}$.

	\begin{enumerate}
	\item Let $g\in G^{3+k,\gamma}_\circ$ be near $g^\circ$. There is a nonlinear second order elliptic differential operator $L_{g}$ so that $u\in C^{2+k,\gamma}(\Sigma^\circ)$ defines a minimal surface for $g$ if and only if $L_{g}(u) = 0$. 
	\item For every $s\in \mathbb{Z}_{\geq 0}$, the map
	\[
		G^{3+k+s,\gamma}_\circ\times C^{2+k,\gamma}(\Sigma^\circ) \ni (g, u) \mapsto L_{g}(u) \in C^{k,\gamma}(\Sigma^\circ)
	\]
	is $C^{2+s}$-Fr\'echet. 
	\item When $g\in G^{3+k,\gamma}_\circ$ and $u\in C^{2+k,\gamma}(\Sigma^\circ)$ defines a minimal surface for $g$, then $D_u L_{g}(u)$ is the stability operator for this minimal surface.

    In particular, if $g=g^\circ$, $L_{g^\circ}$ is given, in Fermi coordinates from $g^\circ$ around $\Sigma^\circ$, by the operator defined by the LHS of \eqref{eq: MSE} and its derivative $D_u L_{g^\circ}(0)$ at $u=0$ is given by the linear elliptic operator in \eqref{eq: stability equation}, which is the stability operator for $\Sigma^\circ$.
	\item For each $(g,u)\in G^{3+k,\gamma}_\circ\times C^{2+k,\gamma}(\Sigma^\circ)$ sufficiently near $(g^\circ,0)$, the map
	\[
		D_u L_{g}(u) \colon C^{2+k,\gamma}_{\mathrm{Dir}}(\Sigma^\circ) \to C^{k,\gamma}(\Sigma^\circ)
	\]
	is an invertible linear elliptic second order differential operator.

	Thus,
	\begin{equation}\label{eq:invbleStable}
	 C^{2+k,\gamma}(\Sigma^\circ) \ni v \mapsto (D_u L_{g}(u)v,v\vert_{\partial\Sigma^\circ}) \in C^{k,\gamma}(\Sigma^\circ)\times C^{2+k,\gamma}(\partial\Sigma^\circ)
	\end{equation}
	is a linear homeomorphism: it is continuous and bijective with continuous inverse.
\end{enumerate}
\end{Lemma}
\begin{proof}
Let $g\in G_\circ^{3+k,\gamma}$. Due to \eqref{eq:defofminsurf}, our goal is first to calculate $D_u A_{g}$. 

In Fermi coordinates of $g^\circ$ around $\Sigma^\circ$, let us have
\[
	g = g_{jk}(x,t)\mathrm{d}x^j \mathrm{d}x^k + g_{tt}(x,t) \mathrm{d} t^2 + \omega(x,t) \otimes\mathrm{d} t+\mathrm{d}t\otimes \omega(x,t)
\]
where $\omega(\cdot,t)$ is a smooth family of one-forms on $\Sigma^\circ$ parametrized by $t\in (-\epsilon,\epsilon)$ and $g_{tt}(x,t) > 0$. We use the notation $h_u = h_u(x) = g(x,u(x))$, $(g_{tt})_u = g_{tt}(x,u(x))$ and $\Sigma_u = \{(x,u(x))\colon x\in \Sigma^\circ\}$ for any $u\in C^{2,\gamma}(\Sigma^\circ)$. 

By following the explicit calculations in \cite[\S~3.1]{carstea2024calder} with the volume form determined in \eqref{eq: coordinate expression dvol} for $g$ (in Fermi coordinates given by $g^\circ$ around $\Sigma^\circ$), one finds that 
\begin{equation}\label{eq:derofA}
	D_u A_{g}(u)v = \int_{\Sigma^\circ} L_{g}(u)v\qquad 	\forall v\in C^{2+n^\ast,\gamma}_{\rm Dir}(\Sigma^\circ)\,,
\end{equation}
where
\[
	L_{g}(u) = -\abs{h_u}^{-1/2} \nabla\cdot \left(\abs{h_u}^{1/2}\sqrt{\mathrm{det}(M(u))}M(u)^{-1} h_u^{-1}(\omega+ (g_{tt})_u \nabla u)\right) + B(u)\,,
\]
with
\begin{align*}
	B(u) &= \frac{1}{2}\sqrt{\mathrm{det}(M(u))}M(u)^{-1} \left(((\partial_sg_{tt})_u h_u^{-1} + (g_{tt})_u(\partial_s g^{-1}_u))\nabla u + 2 (\partial_s g^{-1}_u)\omega\right)\cdot \nabla u\\
	&+ \frac{1}{2}\sqrt{\mathrm{det}(M(u))} \mathrm{Tr}(h_u^{-1}\partial_s h_u)\,,
\end{align*}
and
\[
	M(u) = I_{n\times n}+(g_{tt})_uh_u^{-1}\mathrm{d}u\otimes\mathrm{d}u+h_u^{-1}\mathrm{d}u\otimes \omega + h_u^{-1}\omega\otimes \mathrm{d}u\,.
\]
Using \eqref{eq:defofminsurf} and \eqref{eq:derofA}, we have shown (1). 

The statement (2) follows from the explicit formula for $L_{g}$; taking derivatives with respect to $u$ `costs' a derivative of $g$ which gives the stated behavior. %

In order to see the first part of (3), one merely needs to use the observation from \eqref{eq:derofA} that $D_u^2 A_{g}(u)[v]v=\int_{\Sigma^\circ} D_u L_{g}(u)[v]v$ with the definition of the stability operator in \cite[Eq.~(1.143)]{MR2780140}.%

Either by verification of the explicit formula for $L_{g}$ above, after plugging in $g=g^\circ$, or using the calculations in \cite[\S~3.1]{carstea2024calder}, one finds that $L_{g^\circ}$ is given by the operator on the LHS of the first equation in \eqref{eq: MSE}. That the derivative $D_u L_{g^\circ}(0)$ is given by the operator on the LHS in \eqref{eq: stability equation} follows from the explicit calculation in \cite[\S~3.3]{carstea2024calder}, or see the remarks after \cite[Prop.~3.1]{carstea2024calder}, or \cite[Chp.~1~\S~8]{MR2780140}. This shows (3).

Because for each $g\in G^\circ$ sufficiently near $g^\circ$, the operator $L_{g}$ is a second order elliptic differential operator, so too must its derivative $D_u L_{g}(u)$ at any $u$ near $0$ be, which follows by direct computation, see also \cite[\S~17.2]{zbMATH03561752}.

The map in \eqref{eq:invbleStable} is continuous and linear. Thus, if we show that it is bijective, the open mapping theorem will complete proof of (4). In fact, by the the continuity statement (2), it suffices to show that the operator $D_u L_{g^\circ}(0)$ is invertible from $C^{2+k,\gamma}_{\mathrm{Dir}}(\Sigma^\circ)\to C^{k,\gamma}(\Sigma^\circ)$. Due to the assumption that $\Sigma^\circ$ is admissible, its stability operator, which due to (3) is equal to $D_u L_{g^\circ}(0)$, does not have Dirichlet eigenvalue $0$. In other words, $D_u L_{g^\circ}(0)\colon H^1_0(\Sigma^\circ)\to H^{-1}(\Sigma^\circ)$ is injective. Applying \cite[Chp.~5,~Prop.~1.9]{taylor2011partial} and \cite[Lem.~B.1]{LLLS2019inverse} %
the proof is complete.%
\end{proof}

\begin{Remark}
    Lemmas~\ref{lem:newdef},~\ref{lem: volume form} and items (1), (2), and (3) of Lemma~\ref{lem:linofmin} are true if $\Sigma^\circ$ is merely a smooth embedded minimal surface for $(\widetilde{M},g^\circ)$; that is, $\Sigma^\circ$ must not be admissible.
\end{Remark}

\begin{Corollary}\label{cor:contDependenceOfStabilitySol}
    Let $k\in \mathbb{N}$. For each $(g,u) \in G_\circ^{3+k,\gamma}\times C^{2+k,\gamma}(\Sigma^\circ)$ sufficiently near $(g^\circ,0)$, the map $C^{2+k,\gamma}(\Sigma^\circ)\times C^{2+k,\gamma}(\p\Sigma^\circ) \ni (u,f) \mapsto v^f_u \in C^{2+k}(\Sigma^\circ)$ is continuous, where $v^f_u$ is uniquely defined by the fact that $(D_u L_{g}(u)v^f_u,v^f_u\vert_{\p\Sigma^\circ})=(0,f)$. 
\end{Corollary}
\begin{proof}
    Let $g \in G_\circ^{3+k,\gamma}$ be sufficiently near $g^\circ$. For each $u \in C^{2+k,\gamma}(\Sigma^\circ)$ sufficiently near $0$, denote by $\mathcal{T}(u)$ the map
    \[
        \mathcal{T}(u)\colon v\mapsto (D_u L_g(u)v,v\vert_{\p\Sigma^\circ})\,.
    \]
    According to (2) and (4) in Lemma~\ref{lem:linofmin}, each $\mathcal{T}(u)$ is boundedly invertible and 
    \[
        C^{2+k,\gamma}(\Sigma) \ni u \mapsto \mathcal{T}(u) \in B(C^{2+k,\gamma}(\Sigma^\circ),C^{k,\gamma}(\Sigma^\circ)\times C^{2+k,\gamma}(\partial\Sigma^\circ))
    \]
    is continuous, where $B(C^{2+k,\gamma}(\Sigma^\circ),C^{k,\gamma}(\Sigma^\circ)\times C^{2+k,\gamma}(\partial\Sigma^\circ))$ denotes the space of bounded linear operators from $C^{2+k,\gamma}(\Sigma^\circ)$ to $C^{k,\gamma}(\Sigma^\circ)\times C^{2+k,\gamma}(\partial\Sigma^\circ)$. 

    For each fixed $u_0 \in C^{2+k,\gamma}(\Sigma^\circ)$ sufficiently near $0$, if $u$ is sufficiently near $u_0$, we may write 
    \[
        \mathcal{T}(u)^{-1} = \mathcal{T}(u_0)^{-1}(\mathrm{id}+(\mathcal{T}(u)-\mathcal{T}(u_0))\mathcal{T}(u_0)^{-1})^{-1}\,,
    \]
    where the RHS is well-defined and depends continuously on $u$ by a Neumann series argument (recall $\mathcal{T}(u_0)^{-1}$ is bounded). Now $v^f_u = \mathcal{T}(u)^{-1}(0,f)$ depends continuously on $u$ concluding the proof.
\end{proof}

\begin{Proposition}
\label{prop: smooth dependence of solution}
For every $k\in\mathbb{N}$ there exist $C>0$ and $\delta>0$ so that for every $f\in C^{2+k,\gamma}_\delta(\partial\Sigma^\circ)$ there is a unique solution $u_{g^\circ}^f \in C^{2+k,\gamma}(\Sigma^\circ)$ of \eqref{eq: MSE} (ie. $L_{g^\circ}(u^f_{g^\circ})=0$) with $\norm{u}_{C^{2+k,\gamma}(\Sigma^\circ)}\leq C\norm{f}_{C^{2+k,\gamma}(\partial\Sigma^\circ)}$. The function $u^f_{g^\circ} \in C^{2+k,\gamma}$ depends analytically on $f\in C^{2+k,\gamma}$ in their respective topologies.
\end{Proposition}

\begin{proof}
Consider the map 
$$(u,f) \mapsto (L_{g^\circ}(u), u\mid_{\partial \Sigma^\circ} - f)$$
which is clearly jointly analytic as a map from $C^{2+k,\gamma}(\Sigma^\circ)\times C^{2+k,\gamma}(\partial \Sigma^\circ)$ to $C^{k,\gamma}(\Sigma^\circ)\times C^{2+k,\gamma}(\partial \Sigma^\circ)$. The image of $(u,f) = (0,0)$ under this map is $(0,0)$ since $\Sigma^\circ$ is a minimal surface by assumption. The differential of this map with respect to $u$ evaluated at $(u,f) = (0,0)$ is an isomorphism by (4) of Lemma \ref{lem:linofmin}. The analytic implicit function theorem for Banach spaces \cite[p. 1081]{zbMATH03230715} completes the proof.
\end{proof}

Throughout the rest of this section, $L_{g}$ will always denote the operator constructed in Lemma~\ref{lem:linofmin}, and $u^f_{g^\circ}$ will always refer to the solution constructed in Proposition~\ref{prop: smooth dependence of solution}. 
\begin{Corollary}
\label{cor: functional is smooth}
For every $k\in\mathbb{N}$, the map
	\[
		\mathcal{L} \colon C^{2+k,\gamma}_{\delta}(\partial\Sigma^\circ) \times G^{3+2k,\gamma}_\circ \times C_{\rm Dir}^{2+k,\gamma}(\Sigma^\circ) \ni (f, g, u) \mapsto L_{g}(u_{g^\circ}^f + u)\in C^{k,\gamma}(\Sigma^\circ)
	\]
	is $C^{2+k}$-Fr\'echet differentiable and $\mathcal L(0, g^\circ, 0)=0$.
\end{Corollary}

This suffices to prove that we can find minimal surfaces for metrics near $g^\circ$ and boundary values near $0$.
\begin{Lemma}
\label{lem: smooth dependence of sol}
For every $k\in\mathbb{Z}_{\geq 0}$ there exist $C,\delta>0$ so that for every $(f,g)\in C^{2+k,\gamma}_{\delta}(\partial\Sigma^\circ) \times G^{3+2k,\gamma}_\circ$, there is a unique $u^f_{g} \in C^{2+k,\gamma}(\Sigma^\circ)$ with $\norm{u^f_{g}}_{C^{2+k,\gamma}(\Sigma^\circ)}\leq C\norm{f}_{C^{2+k,\gamma}(\p\Sigma^\circ)}$ and boundary value $f$ that defines a minimal surface for $g$. Furthermore, the map $(f,g) \mapsto u^f_{g}$ from $C^{2+k,\gamma}_{\delta}(\partial\Sigma^\circ) \times G^{3+2k,\gamma}_\circ$ to $C^{2+k,\gamma}(\Sigma^\circ)$ 
is $C^{2+k}$-Fr\'echet differentiable.
\end{Lemma}

\begin{proof}%
In light of Corollary~\ref{cor: functional is smooth}, we will apply the implicit function theorem around $f = 0$, $g = g^\circ$, and $u = 0$ to find solutions to $L_{g}(u)  =0$ which depend (jointly) $C^{2+k}$ on $g$ near $g^\circ$ and $f \in C^{2+k,\gamma}_{\delta}(\partial\Sigma^\circ)$.

By Lemma~\ref{lem:linofmin}, $D_u\mathcal{L}(0,g^\circ,0) = D_u L_{g^\circ}(0) \colon C^{2+k,\gamma}_{\mathrm{Dir}}(\Sigma^\circ)\to C^{k,\gamma}(\Sigma^\circ)$ is invertible, where we used that $u^0_{g^\circ}=0$.  
Therefore, applying the implicit function theorem (see for example \cite[Thm.~2.5.7]{MR960687}) we obtain, for all $f \in  C^{2+k,\gamma}_{\delta}(\partial\Sigma^\circ)$ with sufficiently small $C^{2+k,\gamma}(\partial \Sigma^\circ)$ norm, for all $g$ sufficiently close to $g^\circ$ in $G^{3+2k,\gamma}_\circ$, there is some $v^f_{g}\in C^{2+k,\gamma}_{\rm Dir}(\Sigma^\circ)$ which solves
\begin{eqnarray}\label{eq: result of IFT}
\mathcal L(f, g, v^f_{g}) = 0.
\end{eqnarray}
This means that $u_{g}^f := u_{g^\circ}^f + v_{g}^f$ solves the minimal surface equation \eqref{eq:defofminsurf}, has boundary value $f$ and satisfies the desired norm-inequality. Furthermore the implicit function theorem also asserts that the mapping  $(f,g) \mapsto v^f_{g}$ from $C^{2+k,\gamma}_{\delta}(\partial\Sigma^\circ)\times G^{3+2k,\gamma}_\circ$ to $C^{2+k,\gamma}_{\rm Dir}(\Sigma^\circ)$ is $C^{2+k}$-Fr\'echet differentiable. Combining this with Proposition \ref{prop: smooth dependence of solution} shows that $(f,g) \mapsto u^f_{g}$ from $C^{2+k,\gamma}_{\delta}(\partial\Sigma^\circ)\times G^{3+2k,\gamma}_\circ$ to $C^{2+k,\gamma}(\Sigma^\circ)$ is $C^{2+k}$-Fr\'echet differentiable.

To establish uniqueness, observe that any $u \in C^{2+k,\gamma}(\Sigma^\circ)$ satisfying $\norm{u}_{C^{2+k,\gamma}(\Sigma^\circ)}\leq C\norm{f}_{C^{2+k,\gamma}(\p\Sigma^\circ)}$, having boundary value $f$ that defines a minimal surface for $g$ satisfies
\[
    \norm{u-u^{f}_{g^\circ}}_{C^{2+k,\gamma}(\Sigma^\circ)} \leq C'\norm{f}_{C^{2+k,\gamma}(\p\Sigma^\circ)}
\]
for some $C'>0$ independent of $f$, so that if $\delta>0$ was chosen small enough, we must have $u-u^{f}_{g^\circ}=v^{f}_{g}$, since this is the unique small solution of $\mathcal{L}(f,g,v^{f}_{g})=0$ in $C_{\rm{Dir}}^{2+k,\gamma}(\Sigma^\circ)$.
\end{proof}

\begin{Lemma}\label{lem:newminsurf}
    For every $g \in G^\circ$ sufficiently close to $g^\circ$ with $g\in C^\infty$ and every $f\in C^{2,\gamma}_\delta(\partial\Sigma^\circ)\cap C^\infty$ with $\delta>0$ sufficiently small, the set
    \[
        \Sigma^\circ(f,g) = \{(x,u^f_{g}(x))\colon x\in \Sigma^\circ\}
    \]
    is an admissible minimal surface for $(\widetilde{M},g)$.
\end{Lemma}
\begin{proof}
    We know from Lemma~\ref{lem: smooth dependence of sol} that $\Sigma^\circ(f,g)$ is a smooth embedded minimal surface for $g$. All that remains to verify is that its stability operator does not have Dirichlet eigenvalue $0$. Indeed, in Fermi coordinates from $g^\circ$ around $\Sigma^\circ$, Lemma~\ref{lem:linofmin} (3) gives us the formula $D_u L_{g}(u^f_{g})$ for this stability operator, which according to Lemma~\ref{lem:linofmin} (4) does not have Dirichlet eigenvalue $0$.
\end{proof}

We are ready to move on to the
\begin{proof}[Proof of Theorem \ref{thm: smooth dependence}]
In this proof we use results of this section for $k=0$. The fact that the set defined in \eqref{eq: def of Sigmafg} is a minimal surface follows from Lemma~\ref{lem: smooth dependence of sol}.

According to Lemma~\ref{lem:newdef}, $(g,u) \mapsto A_{g}(u)$ is $C^{2}$-Fr\'echet differentiable. Thus, as a consequence of the chain rule for Fr\'echet derivatives, and by Lemma~\ref{lem: smooth dependence of sol}, we observe %
that $F$ is $C^{2}$-Fr\'echet differentiable from $G^\circ \times U_\delta$ to $\mR$. In particular, $F \colon G^\circ \to C(U_\delta;\mR)$ is $C^2$, and Taylor's theorem \cite[Thm.~2.4.15]{MR960687} (and the remarks thereafter), particularly the continuity of $R_{g}(h)$, then imply \eqref{eq:tayloronF}. 

Finally, we note that by the chain rule, for every $(f,g)\in U_\delta\times G^\circ$ we have
\[
	DF(g)(f) =  (D_{g}A_{g})(u_{g}^f) + (D_u A_{g})(u_{g}^f)D_{g} u^f_{g}\,.
\]
Notice, that by the construction of $u^f_{g}$ we have
\[
	D_{g} u^f_{g} = D_{g}(u^f_{g^\circ}+ v^f_{g}) = D_{g} v^f_{g}\,,
\]
where the operator on the RHS is in fact a bounded linear operator from $G^\circ$ to $C^{2,\gamma}_{\rm Dir}(\Sigma^\circ)$, where emphasis is on the vanishing on the boundary (see the proof of Lemma~\ref{lem: smooth dependence of sol}). Because $u^f_{g}$ is a minimal surface for $g$ we know that \eqref{eq:defofminsurf} must hold for $(D_u A_{g})(u_{g}^f)$, so that $(D_u A_{g})(u_{g}^f)D_{g} u^f_{g}=0$ which gives \eqref{eq:dF}.
\end{proof}

\section{$H^k$ continuity of rough pseudodifferential operators} \label{sec_continuity_psdo}

In the situation we are interested in here a simple argument will lead to the desired estimate; nevertheless, we direct the interested reader to references for results on the continuity of rough $\Psi$DOs and FIOs. Let us first mention \cite[Chp.~XI~Thm.~2.2]{zbMATH03709065} and \cite[Thm.~18.1.11']{zbMATH05129478} which provide such estimates on $L^2$ for some classes of symbols of order $0$. For $\Psi$DOs or FIOs in a specific form with symbols in $L^\infty$ with respect to space and $C^\infty$ in the frequency see \cite{zbMATH06357652} and \cite{zbMATH06241375}. For $\Psi$DOs with symbols that are in $C^r$ (or rather similar to this) with respect to space and $C^\infty$ in the frequency see \cite{zbMATH04113185,zbMATH07703335,zbMATH07741454}, where in \cite{zbMATH03895568,zbMATH04113185} one may find a compositional calculus for such symbols (partly also of finite smoothness in the frequency). We refer also to the books \cite{MR1766415,MR1121019} and references therein. For perhaps the most general statements, see \cite{zbMATH00856809} which provides continuity statements for $\Psi$DOs with Besov-space-like symbols (in space and frequency) mapping between Besov- and Triebel-Lizorkin spaces.

\begin{Proposition}\label{prop:hbounded}
	Let $m,k \in \mR, n\in \mathbb{N}$. There is a universal constant $C>0$ depending only on the dimension $n$ and the values $m,k$ so that the following is true. Let $a \colon \mR^n\times\mR^n\times \mR^n \to \mathbb{C}$ be a measurable function so that there exists a constant $M>0$ so that for all multi-indices $\alpha,\beta$ with $\abs{\alpha}+\abs{\beta} \leq 2n+1+2\max\{\abs{k},\abs{m-k}\}$ we have
	\begin{equation}\label{eq:supa}
		\sup_{\xi\in\mR^n}\iint \abs{\partial_x^\alpha\partial_y^\beta a(x,y,\xi) \langle \xi \rangle^{-m}} \mathrm{d} x\mathrm{d} y \leq M\,.
	\end{equation}
	The operator
	\[
		Af(x) = \iint e^{i(x-y)\cdot\xi}a(x,y,\xi)f(y)\mathrm{d} y\mathrm{d}\xi
	\]
	satisfies
	\[
		\norm{A}_{H^{k}\to H^{k-m}} \leq C M\,.
	\]
\end{Proposition}
\begin{proof}
	We follow the proof of \cite[Thm.~A.1]{zbMATH01116206} and thus keep our arguments brief. Defining $\tilde{a}(\eta,\zeta,\xi) \coloneqq \iint e^{-i(x\cdot\eta+y\cdot\zeta)}a(x,y,\xi)\dd x\dd y$, for any $f\in H^k$ we have 
	\[
		\widehat{Af}(\eta) %
		 =  (2\pi)^{-2n}\iint \tilde{a}(\eta-\xi,\xi-\zeta,\xi) \langle \zeta\rangle^{-k} \widehat{\langle D\rangle^k f}(\zeta)\dd\zeta\dd\xi
	\]
	and we let $g \coloneqq  \langle D\rangle^k f \in L^2$ with $\norm{g}_{L^2} = \norm{f}_{H^k}$. 

	Furthermore,
	\[
		\norm{Af}_{H^{k-m}} = %
		\norm{B\hat g}_{L^2}\,,\quad\text{where}\quad 	B\hat g(\eta) = \langle \eta\rangle^{k-m} \widehat{Af}(\eta) = \int b(\eta,\zeta) \hat g(\zeta)\dd \zeta
	\]
	with 
	\[
		b(\eta,\zeta) = (2\pi)^{-2n} \int \tilde{a}(\eta-\xi,\xi-\zeta,\xi)\dd\xi \langle \eta\rangle^{k-m} \langle \zeta\rangle^{-k}\,.
	\]
	Now by the definition of $B$ and the Cauchy-Schwarz inequality (or see \cite[Prop.~0.5.A]{MR1121019}) we have
	\[
		\norm{A}_{H^k\to H^{k-m}} = \norm{B}_{L^2 \to L^2} \leq \max\left\{\int \abs{b(\eta,\zeta)}\dd\eta, \int \abs{b(\eta,\zeta)}\dd\zeta\right\}\,.
	\]
	We remark that if $C_0>0$ is the universal constant (depending only on the dimension), so that for all $\gamma\in\mR^n$, we have $C_0^{-1} (1+\abs{\gamma})\leq \langle \gamma \rangle \leq C_0(1+\abs{\gamma})$, then for all $s\in \mR, \gamma,\delta\in\mR^n$,
	\[
		\left(\frac{\langle \gamma\rangle}{\langle \delta\rangle}\right)^s \leq (C_0)^{3\abs{s}}\langle \gamma-\delta\rangle^{\abs{s}}\,, \qquad \langle\gamma\rangle\langle\delta\rangle \leq C_0^2 (1+\abs{\gamma})(1+\abs{\delta}) \leq C_0^2(1+\abs{\gamma}+\abs{\delta})^2\,,
	\]
	from which we conclude that 
	\[
		\langle \xi \rangle^{m} \langle \eta\rangle^{k-m} \langle \zeta\rangle^{-k}\leq (C_0^3)^{\abs{k-m}+\abs{k}}\langle \eta-\xi\rangle^{\abs{k-m}}\langle \xi-\zeta\rangle^{\abs{k}} \leq \left(C_0^4 (1+\abs{\eta-\xi}+\abs{\xi-\zeta})\right)^{2\max\{\abs{m-k},\abs{k}\}}\,.
	\]
    The assumption \eqref{eq:supa} implies that for some universal $C'>0$, 
    \[
        |\tilde{a}(\eta,\zeta,\xi)\langle \xi \rangle^{-m}| \leq C' M(1+|\eta|+|\zeta|)^{-2n-1}.
    \]
	Combining the previous two facts gives 
	\[
		\abs{\langle \eta\rangle^{k-m} \langle \zeta\rangle^{-k} \tilde{a}(\eta-\xi,\xi-\zeta,\xi)} \leq MC'C_0^{8\max\{\abs{m-k},\abs{k}\}}(1+\abs{\eta-\xi}+\abs{\xi-\zeta})^{-2n-1}\,,
	\]
	so that
	\[
		\int \abs{b(\eta,\zeta)} \dd\eta \leq MC'C_0^{8\max\{\abs{m-k},\abs{k}\}} \iint (1+\abs{\eta-\xi}+\abs{\xi-\zeta})^{-2n-1}\dd\eta\dd\xi\,,
	\]
	which is shown to be bounded by some universal constant as in \cite[Thm.~A.1]{zbMATH01116206} by substitution in the integral. Repeating the argument for $\int \abs{b(\eta,\zeta)} \dd\zeta$ completes the proof.
\end{proof}

\section{Dirichlet eigenvalues of the Schr\"odinger operator}\label{sec:ev}

    Let $(N,g)$ be a compact connected smooth Riemannian manifold with smooth boundary, and $L \coloneqq \Delta_{g}+q$ with $q\in L^\infty$ a Schr\"odinger operator. %
    All of the results stated in this section for the Schr\"odinger operator on Riemannian manifolds are known for any second order elliptic differential operator on domains of $\mR^n$ (see \cite{zbMATH03230686,zbMATH00048198} and the modern \cite{zbMATH07819857}) and for the Laplace-Beltrami operator on Riemannian manifolds (\cite{zbMATH03877889}).

For any $\Omega\subset N$ open, according to \cite[\S~2.2.4]{MR1889089}, the operator $L$ with domain $H^2(\Omega)\cap H^1_0(\Omega)$ is self-adjoint on $L^2(\Omega)$ and is associated to the quadratic form
\[
    H^1_0(\Omega)\times H^1_0(\Omega) \ni (u,v) \mapsto Q(u,v) = \int_\Omega \big((du,d\bar v)_{g}+qu\bar v\big)\s\dd {\rm Vol} \in \mathbb{C}\,,
\]
where ${\rm Vol}$ is the volume form on $(N,g)$. This means that $\langle Lu, v\rangle_{L^2(\Omega)} = Q(u,v)$ for all $u,v\in H^2(\Omega)\cap H^1_0(\Omega)$, see also \cite[Lem.~2.20]{MR1889089}. 

Furthermore, according to \cite[Thm.~2.21]{MR1889089}, the Dirichlet eigenvalues of $L$ on $H^2(\Omega)\cap H_0^1(\Omega)$ for any open $\Omega\subset N$ can be enumerated as 
\[
    \lambda_1(\Omega) < \lambda_2(\Omega)\leq \dots\,,
\]
repeated according to multiplicity, with $\lambda_k(\Omega) \to \infty$ as $k\to\infty$. 

According to the min-max principle \cite{zbMATH03315567}, for any $\Omega\subset N$ open, 
\begin{equation}\label{eq:minmax}
    \lambda_k(\Omega) = \min_{\substack{\mathcal{L}\subset H^1_0(\Omega)\\\dim \mathcal{L}=k}} \max_{u\in \mathcal{L}\setminus 0} \frac{Q(u,u)}{\norm{u}_{L^2(\Omega)}^2} = \min_{\substack{u\in H^1_0(\Omega)\setminus 0\\ u\perp \mathcal{L}_{k-1}}}\frac{Q(u,u)}{\norm{u}_{L^2(\Omega)}^2}
\end{equation}
where $\mathcal{L}_{k-1} = \mathrm{span}\{u_1,\dots,u_{k-1}\}$ and $u_1,\dots,u_{k-1}$ are Dirichlet eigenfunctions for $L$ for the respective eigenvalues $\lambda_1(\Omega),\dots,\lambda_{k-1}(\Omega)$, see also \cite[Prop.~3.1.3,~Rem.~3.1.4]{zbMATH07759244} and \cite[Eq.~(3)]{zbMATH01346906}.

According to the Rellich-Kondrakov theorem \cite[Thm.~2.34]{MR1636569}, the inclusion $H^1_0(\Omega)\to L^2(\Omega)$ is compact as long as $\Omega$ has $C^1$ boundary, so that as a consequence of \cite[Lem.~1.1]{zbMATH01346906} we have
\begin{Lemma}\label{lem:evmon}
    For any open $\Omega, \Omega' \subset N$ with $C^1$ boundary, denoting by $\lambda_k(\Omega),\lambda_k(\Omega')$ the $k$-th Dirichlet eigenvalue of $L$ on $H^2(\Omega)\cap H^1_0(\Omega)$ and $H^2(\Omega')\cap H^1_0(\Omega')$ respectively, if $\Omega'\subset \Omega$, then
    \[
        \lambda_k(\Omega)\leq \lambda_k(\Omega')
    \]
    for all $k\in \mathbb{N}$. If there is equality above for all $k$, then $\Omega = \Omega'$.
\end{Lemma}

Furthermore, by \cite[Thm.~1]{zbMATH01346906},
\begin{Lemma}\label{lem:evcont}
    Let $\Omega_1,\Omega_2,\dots \subset N$ be open sets with $\Omega_j\subset \Omega_{j+1}$ for all $j\in\mathbb{N}$ and open $\Omega\coloneqq \bigcup_{j\in\mathbb{N}} \Omega_j \subset N$ having $C^1$ boundary. We have
    \[
        \lambda_k(\Omega_j) \xrightarrow{j\to\infty} \lambda_k(\Omega)
    \]
    for all $k\in\mathbb{N}$.
\end{Lemma}
The results in \cite{zbMATH01346906} are much more powerful than their consequences we have stated here. %
Lemma~\ref{lem:evmon} is called the domain monotonicity of eigenvalues, and Lemma~\ref{lem:evcont} can be interpreted as continuous dependence of the eigenvalues on the domain. 

We intend to show the \emph{strict} domain monotonicity property of eigenvalues for which we follow the proofs in \cite[Thm.~2.3]{zbMATH03393073}, \cite[Thm.~3.2.1]{zbMATH07759244}, \cite[Thm.~2.21]{zbMATH07819857},
\begin{Proposition}\label{prop:evstrict}
    Let $\Omega,\Omega'\subset N$ be open with $\Omega\setminus \Omega'$ containing an open set $\subset N$. We have
    \[
        \lambda_k(\Omega) < \lambda_k(\Omega')
    \]
    for all $k\in\mathbb{N}$.
\end{Proposition}
\begin{proof}
    Assume for contradiction that there is $k\in \mathbb{N}$ so that $\lambda \coloneqq \lambda_k(\Omega) = \lambda_k(\Omega')$. Because we know that $\lambda_j(\Omega) \to \infty$ as $j\to\infty$, there must be $m\in\mathbb{N}$ so that $\lambda_m(\Omega)> \lambda$. 

    Define open $\Omega_1,\dots, \Omega_m \subset N$ so that
    \[
        \Omega' = \Omega_1 \subsetneq \Omega_2 \subsetneq \dots \subsetneq \Omega_m = \Omega
    \]
    and $\Omega_{j+1}\setminus \Omega_j$ contains an open set for all $j\in 1,\dots,m-1$. By Lemma~\ref{lem:evmon} we must have
    \[
        \lambda = \lambda_k(\Omega) = \lambda_k(\Omega_m) \leq \lambda_k(\Omega_{m-1}) \leq \dots\leq \lambda_k(\Omega_1) = \lambda_k(\Omega') = \lambda
    \]
    and thus equality in each of these inequalities. 

    Choose eigenfunctions $u_j \in H^2(\Omega_j)\cap H^1_0(\Omega_j)$ of $L$ restricted to the domain $H^2(\Omega_j)\cap H^1_0(\Omega_j)$ with eigenvalue $\lambda$. Denote by $u_j$ the extension of $u_j$ by $0$ to $\Omega$. %
    In order to show that $u_1,\dots,u_m$ are linearly independent in $\Omega$, we consider, for some $a_1,\dots,a_m\in \mR$, $h \coloneqq \sum_{j=1}^m a_j u_j \in H^1_0(\Omega)$. Let us assume that $h=0$. Because $h = a_m u_m$ in $\Omega_m \setminus \Omega_{m-1}$, and because by the unique continuation principle $u_m$ cannot vanish on $\Omega_m\setminus \Omega_{m-1}$, we must have $a_m = 0$. Proceeding in this manner, we find that $a_m=\dots=a_1=0$, so that we have indeed shown that $u_1,\dots,u_m$ are linearly independent. %
    
    Next we consider the function $\varphi := c_1 u_1 + \ldots + c_m u_m \in H^1_0(\Omega)$, where the constants $c_j$ are chosen so that $\varphi$ is $L^2$-orthogonal to the first $m-1$ eigenfunctions of $L$ on $H^2(\Omega) \cap H^1_0(\Omega)$ and $\norm{\varphi}_{L^2(\Omega)} = 1$. %
    By the min-max principle \eqref{eq:minmax}, we then know that
    \[
        \lambda < \lambda_m(\Omega) \leq Q(\varphi,\varphi)\,.
    \]
    On the other hand, using the fact that $u_j$ are $H^2(\Omega_j) \cap H^1_0(\Omega_j)$ eigenfunctions in $\Omega_j$ gives  
    \[
        Q(\varphi,\varphi) = \sum_{j,l} c_j c_l Q(u_j, u_l) = \sum_{j,l} c_j c_l \lambda \langle u_j, u_l \rangle_{L^2} = \lambda \norm{\varphi}_{L^2}^2 = \lambda
    \]
    which gives a contradiction, completing the proof.
\end{proof}

We now have all tools required to provide the
\begin{proof}[Proof of Proposition~\ref{prop:makeadmissible}]
    Denote by $\lambda_j(\Omega)$ the Dirichlet eigenvalues of the stability operator of $\Sigma^\circ$ restricted to the domain $H^2(\Omega)\cap H^1_0(\Omega)$ for any $\Omega\subset \Sigma^\circ$. 
    
    Let us assume first that $\Sigma^\circ$ is not admissible. Let $m\in \mathbb{N}$ be the multiplicity of the zero eigenvalue of the stability operator of $\Sigma^\circ$ so that $\lambda_k(\Sigma^\circ) = \dots = \lambda_{k+m-1}(\Sigma^\circ)=0$ for some $k\in \mathbb{N}$ and $\lambda_{k-1}(\Sigma^\circ)<0$ if $k>1$. According to Proposition~\ref{prop:evstrict}, there is $\Sigma\subset \Sigma^\circ$ with $\Sigma^\circ\setminus\Sigma$ open and as small as desired so that $\lambda_k(\Sigma), \dots, \lambda_{k+m-1}(\Sigma)>0$. In particular, we may assume that $\p \Sigma \subset \widetilde{M} \setminus M$.

    If $k=1$ we choose $\Sigma'\coloneqq \Sigma$. Otherwise proceed as follows. Let $\Sigma_1, \Sigma_2, \dots \subset \Sigma^\circ$ so that $\Sigma_{j}\subset \Sigma_{j+1}$, $\Sigma \subset \Sigma_1$, $\cup \Sigma_j = \Sigma^{\circ}$, $\p \Sigma_j \subset \widetilde{M} \setminus M$, and $\Sigma^\circ\setminus\Sigma_j$ is open for all $j$. From Lemma~\ref{lem:evcont} we know that
    \[
        \lambda_{k-1}(\Sigma_j) \to \lambda_{k-1}(\Sigma^\circ) < 0
    \]
    as $j\to \infty$ so that there is some $j^\ast \in \mathbb{N}$ so that $\lambda_{k-1}(\Sigma_{j^\ast}) < 0$. Furthermore, since $\Sigma^\circ\setminus \Sigma_{j^\ast}$ is open we must have 
    \[
        \lambda_{j}(\Sigma_{j^\ast}) > \lambda_j(\Sigma^\circ)
    \]
    for all $j$. In particular, $\lambda_{j}(\Sigma_{j^\ast}) > 0$ for all $j \geq k$. Since $\lambda_{j}(\Sigma_{j^\ast}) \leq \lambda_{k-1}(\Sigma_{j^\ast}) < 0$ for all $j\leq k-1$, we thus have the desired conclusion for $\Sigma' \coloneqq \Sigma_{j^\ast}$. 
    This completes the proof of the first statement. If all underlying structures are analytic, one can choose $\p \Sigma'$ to be analytic (e.g.\ by considering a boundary defining function that is an analytic approximation of the original boundary defining function for $\Sigma'$, obtained by taking a suitable finite part of its eigenfunction expansion).

\end{proof}

\bibliography{refs_min_area_lin} 
        
\bibliographystyle{alpha}

\end{document}